\documentclass[11pt]{article}

\usepackage[utf8]{inputenc}
\usepackage{geometry}
\usepackage{amssymb,amsmath,amsthm,mathtools,amsbsy,amsfonts}
\usepackage{subcaption}
\usepackage{graphicx,wasysym}
\usepackage{stmaryrd}
\usepackage{cases}
\usepackage{empheq}
\usepackage{paralist}
\usepackage{hyperref}
\hypersetup{colorlinks=true,linkcolor=red,urlcolor=black,citecolor=green}
\usepackage{multimedia}
\usepackage{tikz}
\usepackage{enumitem}
\usepackage{cite}
\usepackage{multirow}
\usepackage{hhline}
\usepackage{tabularx}
\usepackage{makecell}
\usepackage[ruled,vlined]{algorithm2e}
\SetKwRepeat{Do}{do}{while}
\usepackage{threeparttable}
\usepackage{siunitx}

\graphicspath{{./fig/}}

\usepackage{fullpage}
\usepackage[explicit]{titlesec}
\titleformat{\section}{\large\bfseries\center\raggedright}{\thesection}{0.5em}{{#1}}[]
\titleformat{\subsection}[runin]{\bfseries}{\thesubsection}{0.5em}{{#1}}[.]
\titleformat{\subsubsection}[runin]{\bfseries\itshape}{\thesubsubsection}{0.5em}{{#1}}[.]
\titleformat{\paragraph}[runin]{\itshape}{\theparagraph}{0.5em}{{#1}}[.---]
\titlespacing*{\section}{0pt}{0.8\baselineskip}{0.6\baselineskip}
\titlespacing*{\subsection}{0pt}{0.6\baselineskip}{0.4\baselineskip}
\titlespacing*{\subsubsection}{0pt}{0.4\baselineskip}{0.4\baselineskip}
\titlespacing*{\paragraph}{0pt}{0.2\baselineskip}{0.2\baselineskip}

\numberwithin{equation}{section}

\setenumerate{itemsep=-0.2em,topsep=0.3em}
\setitemize{itemsep=-0.2em,topsep=0.3em}

\newcommand\namefont{\normalfont\scshape}
\newcommand\numberfont{\normalfont\scshape}
\newcommand\notefont{\normalfont}
\newtheoremstyle{mystyle} 
    {0.3em} 
    {0.3em} 
    {\itshape} 
    {} 
    {\normalfont} 
    {.} 
    {.5em} 
    {{\namefont\thmname{#1}}~{\numberfont\thmnumber{#2}}{\notefont\thmnote{ (#3)}}} 
\theoremstyle{plain}
\newtheorem{thm}{Theorem}[section]

\newtheorem{rem}[thm]{Remark}
\newtheorem{lem}[thm]{Lemma}
\newtheorem{cor}[thm]{Corollary}
\newtheorem{prop}[thm]{Proposition}

\newtheorem{assu}[thm]{Assumption}
\newtheorem{pb}[thm]{Problem}
\newtheorem{defn}[thm]{Definition}

\newtheoremstyle{namedthmstyle} 
        {} 
    {} 
    {\itshape} 
    {} 
    {\bfseries} 
    {} 
    { } 
        {}
\theoremstyle{namedthmstyle}
\newcommand{\thistheoremname}{}
\newtheorem*{genericthm}{\thistheoremname}
\newenvironment{namedthm}[1]
{\renewcommand{\thistheoremname}{#1}%
  \begin{genericthm}}
  {\end{genericthm}}
\makeatletter
\def\namedlabel#1#2{\begingroup
   \def\@currentlabel{#2}%
   \label{#1}\endgroup
}
\makeatother

\allowdisplaybreaks
\mathtoolsset{showonlyrefs}
\definecolor{darkred}{rgb}{0.6,0.1,0.1}
\definecolor{darkgreen}{rgb}{0.1,0.6,0.1}
\definecolor{darkblue}{rgb}{0.1,0.1,0.6}

\newcommand{\mt}[1]{\mathrm{#1}}
\def\st{\, \left|\right. \,} 
\def\:{\colon} 
\def\wto{\rightharpoonup} 

\newcommand{\abs}[1]{\left\vert#1\right\vert}
\newcommand{\ap}[1]{\left\langle#1\right\rangle} 
\newcommand{\norm}[1]{\left\Vert#1\right\Vert}

\def\grad{\nabla}
\def\bgrad{\boldsymbol{\nabla}}

\DeclareMathOperator{\dive}{div}

\DeclareMathOperator{\dist}{dist}

\DeclareMathOperator{\argmin}{argmin}

\DeclareMathOperator{\graph}{Gr}

\def\R{\mathbb{R}} 
\def\N{\mathbb{N}} 
\def\B{\mathcal{B}} 

\def\e{\varepsilon}

\def\d{\,\mathrm{d}}

\def\p{\partial}

\newcommand{\bs}[1]{\boldsymbol{#1}}

\def\bnull{\boldsymbol{0}}
\def\bx{\boldsymbol{x}}

\def\bu{\boldsymbol{u}}
\def\bv{\boldsymbol{v}}

\def\ba{\boldsymbol{a}}

\def\bg{\boldsymbol{g}}
\def\bff{\boldsymbol{f}}
\def\bn{\boldsymbol{n}}

\def\bvarphi{\boldsymbol{\varphi}}
\def\bLambda{\boldsymbol{\Lambda}}

\def\blambda{\boldsymbol{\lambda}}

\def\bL{\boldsymbol{L}}
\def\bV{\boldsymbol{V}}

\def\bp{\boldsymbol{\partial}}

\def\F{\mathcal{F}} 
\def\E{\mathcal{E}} 

\def\D{\mathcal{D}} 
\def\Reg{\gamma} 

\newcolumntype{C}[1]{>{\centering\arraybackslash}m{#1}}
\newcolumntype{L}[1]{>{\raggedright\arraybackslash}m{#1}}

\renewcommand{\leq}{\leqslant}
\renewcommand{\geq}{\geqslant}
\newcommand{\sib}[1]{[\si{#1}]}

\begin{document}

\title{Well-posedness and variational numerical scheme for an adaptive model in highly heterogeneous porous media}

\author{Alessio Fumagalli$^1$ \and Francesco Saverio Patacchini$^2$}

\date{
$^1$ Department of Mathematics, Politecnico di Milano, p.za Leonardo da Vinci 32, Milano 20133, Italy\\%
$^2$ IFP Energies nouvelles, 1 et 4 avenue de Bois-Pr\'eau, 92852 Rueil-Malmaison, France
}

\maketitle

\begin{abstract}
\noindent Mathematical modeling of fluid flow in a porous medium is usually described by a continuity equation and a chosen constitutive law. The latter, depending on the problem at hand, may be a nonlinear relation between the fluid's pressure gradient and velocity. The actual shape of this relation is normally chosen at the outset of the problem, even though, in practice, the fluid may experience velocities outside of its range of applicability. We propose here an adaptive model, so that the most appropriate law is locally selected depending on the computed velocity. From the analytical point of view, we show well-posedness of the problem when the law is monotone in velocity and show existence in one space dimension otherwise. From the computational point of view, we present a new approach based on regularizing via mollification the underlying dissipation, i.e., the power lost by the fluid to the porous medium through drag. The resulting regularization is shown to converge to the original problem using $\Gamma$-convergence on the dissipation in the monotone case. This approach gives rise to a variational numerical scheme which applies to very general problems and which we validate on three test cases.
\end{abstract}




\section{Introduction}
\label{sec:introduction}

We study the stationary flow of a Newtonian fluid in a fully saturated, highly
heterogeneous porous medium. Typically, the heterogeneities come from the
lithological and geometrical properties of the medium. Indeed, very different
sediments (such as sandstone and carbonates) and fractures with irregular
aperture may be involved. These properties impact the permeability of the domain
and thus the fluid's velocity.

Our model is based on the following constitutive, or seepage, law, which is in fact a force balance:
\begin{equation}
  \label{eq:main-intro}
  \blambda(\bu) = - \bgrad p + \bff,
\end{equation}
where $\bu$ and $p$ are respectively the seepage flux and the fluid's pressure.
The term $\blambda(\bu)$ is the opposite of the drag force experienced by the
fluid and the term $\bff$ is a vector of external body forces, like gravity. The assumption on
the operator $\blambda$ most recurrent in the literature of porous media is linearity, meaning that \eqref{eq:main-intro} is Darcy's law \cite{Bear1972,Helmig1997}. This, however,
is known to be valid only for low Reynolds numbers, i.e., low fluid speeds
\cite{Zeng2006}, beyond which Darcy's law tends to overestimate velocities. To describe flows at higher speeds more accurately, it is common to add a quadratic
term to Darcy's law to penalize high velocities and get the so-called Darcy--Forchheimer law, which is an
example of a nonlinear operator $\blambda$
\cite{Girault2008,Frih2008,Knabner2014,Ahmed2018}. Other common laws are obtained by
adding a higher-order term or a Laplacian term to Darcy's, yielding
Forchheimer's generalized law or Brinkman's law \cite{Morales2017}, respectively.

These commonly used models are well known for being well posed and providing
good predictions in a homogeneous medium. They can also be adapted to give
accurate results in a heterogeneous medium by, for instance, taking spatially
dependent permeabilities. However, they do not cover the case when the medium's
heterogeneities yield an operator $\blambda$ discontinuous in $\bu$. Since a
linear law is better adapted to low Reynolds numbers whereas a nonlinear law
gives a better description of high-speed regimes, one may expect that allowing
$\blambda$ to be linear under some given speed threshold and to be nonlinear
above this threshold should deliver improved results. This consideration
motivated us to study discontinuous seepage laws in \cite{FP21}, and we continue
here the work started therein.

To handle mathematically such a discontinuous problem, we make use of a multivalued version of \eqref{eq:main-intro} in the case when $\blambda$ involves no space derivatives of the flux (thus excluding Brinkman's law). We then show its well-posedness when the drag force is maximal monotone in the flux variable, using classical tools from multivalued operator theory. We also prove existence of solutions when the monotonicity fails and the space dimension $d$ equals one. Consequently, we introduce a regularized, monovalued approximation of the multivalued problem, which can be solved numerically using classical fixed-point and finite-element methods. This regularization is based on the mollification of the dissipation (i.e., the power the fluid loses to the surrounding medium because of drag), and we show that it converges to the original problem using variational results, in particular, the $\Gamma$-convergence of the regularized dissipation to the unregularized one when it is convex; when the dissipation is nonconvex, the regularized problem is still shown to have solutions when $d=1$, but is not proved to converge in any case. After applying the fixed-point and finite-element methods, we compare the resulting regularized algorithm to that introduced in~\cite{FP21}, called the transition-zone tracking algorithm. The latter is based on iteratively locating the zones separating any pair of different speed regimes and solving the appropriate law in every region thus defined; it differs from the algorithm derived in this paper, which, instead of tracking the transition zones sharply, spreads them out smoothly and then solves the resulting problem using the same regularized law in the whole medium. The two approaches give very similar results for $d=1$ and for a combination of two different speed regimes, as we show on a simple test case, but the regularized approach offers the advantage of applying immediately for $d>1$ and for any number of regimes, as showcased by two other test cases.

This paper is organized as follows. In Section~\ref{sec:physical-framework}, the
physical and multivalued framework is introduced and motivated, and in Section~\ref{sec:math-fram} the weak formulation is given and the well-posedness results proved in the adequate functional spaces for general constitutive operators including no space derivatives of the flux. Section~\ref{sec:well-posedn-diss} contains the well-posedness theory specifically formulated for some common examples of constitutive laws. In Section~\ref{sec:regul-probl-diss}, the regularizing approach is introduced and its convergence demonstrated, while in Section~\ref{sec:num-appr} we briefly describe the numerical approximation adopted to solve the regularized problem. Section~\ref{sec:numerical_results} contains the numerical results of
three test cases. Finally, in Section~\ref{sec:conclusion}, we give conclusions.
For the reader's convenience, appendices are provided recalling basic notions on multivalued operators, functionals and mollification.


\section{Physical framework}
\label{sec:physical-framework}

We denote by $\Omega$ the porous medium, which we assume to be open, bounded and with Lipschitz boundary $\p\Omega$; we write $\bn$ the outward normal unit vector of $\p\Omega$. The unknowns of the problems discussed throughout the paper are the fluid's pressure $p\:\Omega\to\R$ and the seepage flux $\bu\:\Omega\to\R^d$ defined by $\bu = \rho\phi\bV$, where $\phi$ is the medium's porosity and $\rho$ and $\bV$ are the fluid's density and velocity. This relation between flux and velocity justifies that the terms ``flux'' and ``velocity'' may be used interchangeably. We suppose that $\rho\:\Omega\to(0,\infty)$ and $\phi\:\Omega\to(0,1)$ are space-dependent knowns of the problem.

We wish to study the stationary flow of the fluid through the porous medium. 

\subsection{Classical setting}
\label{sec:classical-setting}

Before discussing our novel approach, let us recall the classical setting for the description of the fluid flow in $\Omega$.

\subsubsection{Problem formulation}
\label{sec:problem-formulation}

The conservation of mass reads
\begin{equation}
  \label{eq:cons-mass}
  \dive\bu = q \quad \text{in $\Omega$},
\end{equation}
where $q\:\Omega\to\R$ is a known fluid mass source, and the conservation of momentum is given by
\begin{equation}
 \label{eq:cons-mom}
  \blambda(\bu) = -\bgrad p + \bff \quad \text{in $\Omega$},
\end{equation}
where $\bff\:\Omega\to\R^d$ is a known vector of external body forces, possibly including gravity. The conservation of momentum shows that the pressure gradient and the external forces balance the \emph{drag force} $-\blambda(\bu)$ undergone by the fluid. We refer to \eqref{eq:cons-mom} as the \emph{seepage law} and to $\blambda$ as the \emph{drag operator}.

To close the problem \eqref{eq:cons-mass}-\eqref{eq:cons-mom} for both the flux and the pressure, we need to fix boundary conditions. Thus, let $\Sigma_{\mt{v}},\Sigma_{\mt{p}} \subset \p\Omega$ be relatively open in $\p\Omega$ (i.e., $\Sigma_{\mt{v}}$ and $\Sigma_{\mt{p}}$ are each the intersection of an open subset of $\R^d$ with $\p\Omega$) and such that $\p\Omega = \overline{\Sigma_{\mt{v}} \cup \Sigma_{\mt{p}}}$ and $\Sigma_{\mt{v}} \cap \Sigma_{\mt{p}} = \emptyset$. Then, impose
\begin{equation}
  \label{eq:bdry-conditions}
  \begin{cases}
    \bu \cdot \bn = u_0 & \text{on $\Sigma_{\mt{v}}$},\\
    p = p_0 & \text{on $\Sigma_{\mt{p}}$},
  \end{cases}
\end{equation}
where $u_0: \Sigma_{\mt{v}} \rightarrow \R$ and $p_0\:\Sigma_{\mt{p}}\to\R$ are given functions setting the conditions on the boundary for $\bu$ and $p$. For simplicity, a map on $\Omega$ and its trace on $\p\Omega$ are denoted by the same symbol. 

Overall, the problem summarizes as follows:
\begin{pb}[classical strong form]
  \label{pb:strong-form}
  Find $\bu\:\Omega\to\R^d$ and $p\:\Omega\to\R$ such that
  \begin{equation}
    \label{eq:strong-form}
    \begin{cases}
      \dive\bu = q & \text{in $\Omega$},\\
      \blambda(\bu) = -\bgrad p + \bff & \text{in $\Omega$},\\
      \bu \cdot \bn = u_0 & \text{on $\Sigma_{\mt{v}}$},\\
      p = p_0 & \text{on $\Sigma_{\mt{p}}$}.
    \end{cases}
  \end{equation}
\end{pb}

\begin{rem}[average pressure]
  \label{rem:press-ave}
  If the boundary piece $\Sigma_{\mt{p}}$ verifies $\mt{Vol}^{d-1}(\Sigma_{\mt{p}}) = 0$, where $\mt{Vol}^{d-1}$ is the $(d-1)$-dimensional Lebesgue measure, to ensure uniqueness of the pressure satisfying Problem~\ref{pb:strong-form}, one imposes a constraint on the average of $p$:
\begin{equation}
  \label{eq:pressure-average}
  \frac{1}{\abs{\Omega}}\int_\Omega p = \bar p,
\end{equation}
for a given $\bar p\in\R$. Tacitly, we therefore require \eqref{eq:pressure-average} in \eqref{eq:strong-form} whenever $\mt{Vol}^{d-1}(\Sigma_{\mt{p}}) = 0$. As seen below, this condition becomes explicit in the weak formulation of \eqref{eq:strong-form} through the definition of the underlying Sobolev space (cf. Section~\ref{sec:weak-formulation}).
\end{rem} 

\subsubsection{Continuous drag operators}
\label{sec:cont-drag-oper}

Classically, the drag operator $\blambda$ is assumed to be \emph{continuous} in flux and can either be linear or not. Common linear operators found in the literature include
\begin{equation}
  \label{eq:laws-linear}
  \blambda_{\mt{D}}(\bu) = \mathbb{D}\bu,\quad \blambda_{\mt{S}}(\bu) = \nu\bs{\Delta}\bu \quad \text{and} \quad \blambda_{\mt{B}}(\bu) = \mathbb{D}\bu + \nu\bs{\Delta}\bu,
\end{equation}
where $\mathbb{D}\:\Omega\to\R^{d\times d}$ is the drag tensor and $\nu\:\Omega\to(0,\infty)$ the fluid's kinematic viscosity, satisfying $\mathbb{D} = \nu \mathbb{K}^{-1}$, with $\mathbb{K}\:\Omega\to\R^{d\times d}$ the medium's permeability. The first two operators in \eqref{eq:laws-linear} correspond to Darcy's and Stokes' laws, respectively, while the third one yields a combination of the two, referred to as Brinkman's law. 

Nonlinearities can occur when high speeds are reached by the fluid. Classical examples reflecting this behavior are given by
\begin{equation}
\label{eq:laws-nonlinear}
  \blambda_{\mt{F}}(\bu) = \lambda\norm{\bu} \bu, \quad \blambda_{\mt{GF}}(\bu) = \lambda\norm{\bu}^\gamma \bu \quad \text{and} \quad \blambda_{\mt{DF}}(\bu) = (\mathbb{D} + \lambda\norm{\bu}\mathbb{I})\bu,
\end{equation}
where $\lambda\:\Omega\to(0,\infty)$ and $\gamma\in(0,\infty)$ are the Forchheimer coefficient and exponent, respectively, and $\mathbb{I}$ stands for the identity matrix. The first operator in \eqref{eq:laws-nonlinear} corresponds to Forchheimer's law, which is the case $\gamma=1$ in Forchheimer's generalized law given by the second operator. The third operator leads to a combination of Darcy's and Forchheimer's laws, referred to as the Darcy--Forchheimer law. Nonlinearities can also come into play at very low speeds, in which case a sublinear part can be added to Darcy's law to get the following operator:
\begin{equation}
\label{eq:law-sublinear}
  \blambda_{\mt{sub}}(\bu) = \frac{\lambda_1\bu}{1+\lambda_2\norm{\bu}} + \mathbb{D}\bu,
\end{equation}
where $\lambda_1,\lambda_2\:\Omega\to(0,\infty)$ are experimental parameter functions~\cite{HYHW13}.

\subsection{New setting}
\label{sec:new-setting}

As mentioned in the introduction, we wish to include heterogeneities which yield a discontinuity of the drag operator with respect to the seepage flux. Below, we show that the conservation of momentum \eqref{eq:cons-mom} needs to be adapted to this discontinuous setting (whereas the conservation of mass \eqref{eq:cons-mass} and the boundary conditions \eqref{eq:bdry-conditions} remain untouched).

\subsubsection{Motivating example}
\label{sec:motivating-example}

In \eqref{eq:laws-nonlinear}, the operator $\blambda_{\mt{DF}}$ offers fair accuracy when both low- and high-speed regimes are encountered: where Reynolds' number is low, the linearity of Darcy's law prevails, whereas where it is high, the nonlinearity of Forchheimer's dominates. Nevertheless, since the Forchheimer term is always present, nonlinear effects may manifest even in low-speed parts of $\Omega$, especially in the neighborhood of the transition zone separating low- and high-speed regions. To counter this effect, we wish to consider the drag operator $\blambda_{\text{D/DF}}$ given as
\begin{equation}
  \label{eq:jump-DDF}
  \blambda_{\text{D/DF}}(\bu) = \begin{cases} \mathbb{D}\bu & \text{if $\norm{\bu}< \bar u$},\\ \mathbb{D}\bu + \lambda\norm{\bu}\bu & \text{if $\norm{\bu}> \bar u$}, \end{cases} 
\end{equation}
where $\bar u>0$ is a \emph{threshold flux} defining the separation between speed regimes and $\norm{\cdot}$ stands for the norm in $\R^d$. Note that $\blambda_{\text{D--DF}}(\bu)$ has a gap in its domain corresponding to the transition zone 
\begin{equation*}
  \Gamma(\bu) := \{\bx\in \Omega \st \norm{\bu(\bx)} = \bar u\};
\end{equation*}
indeed, we want to impose the drag force only in the low- and high-speed regions $\{\bx\in \Omega \st \norm{\bu(\bx)} < \bar u\}$ and $\{\bx\in \Omega \st \norm{\bu(\bx)} > \bar u\}$. Consequently, the conservation of momentum \eqref{eq:cons-mom} is not anymore valid on all of $\Omega$ and needs to be changed into
\begin{equation*}
  \blambda_{\text{D/DF}}(\bu) = -\bgrad p + \bff \quad \text{in $\Omega \setminus \Gamma(\bu)$}.
\end{equation*}
Then, Problem~\ref{pb:strong-form} must be modified accordingly.

\subsubsection{Jump drag operators}
\label{sec:jump-drag-operators}

Following the above motivating example, we now identify at least one class of drag operators which we want our theory to include; we refer to the members of this class as \emph{jump} operators.

Let $n\in\{2,3,\dots\}$ and consider a family $\{\bar u_j\}_{j=1}^{n-1}\subset(0,\infty)$ of strictly ordered threshold fluxes to which we add the convenient values $\bar u_0 = - \bar u_n = -\infty$. Let also $\{\blambda_j\}_{j=1}^n$ be a family of drag operators. Given a flux $\bu$, write $\{\Gamma_j(\bu)\}_{j=1}^{n-1}$ and $\{\Omega_j(\bu)\}_{j=1}^n$ the sets of transition zones and speed regions, respectively, given by
\begin{equation}
  \label{eq:trans-speed-regions}
  \begin{gathered}
    \Gamma_j(\bu) = \{\bx\in\Omega\st \norm{\bu(\bx)} = \bar u_j\}, \qquad j\in\{1,\dots,n-1\},\\
    \Omega_j(\bu) = \{\bx\in\Omega\st \norm{\bu(\bx)} \in (\bar u_{j-1},\bar u_j)\}, \qquad j\in\{1,\dots,n\}.
  \end{gathered}
\end{equation}
Then, for all $j\in\{1,\dots,n\}$, let $\blambda$ be given as
\begin{equation}
  \label{eq:jump-op}
  \blambda(\bu) = \blambda_j(\bu) \quad \text{in $\Omega_j(\bu)$},
\end{equation}
which is what we call a jump drag operator.

\paragraph{Gap formulation}
Writing 
\begin{equation*}
  \Gamma(\bu) = \bigcup_{j=1}^{n-1}\Gamma_j(\bu),
\end{equation*}
the union of all the transition zones, $\blambda$ in \eqref{eq:jump-op} can be equivalently rewritten as
\begin{equation}
  \label{eq:drag-si}
  \blambda(\bu) = \sum_{j=1}^n s_j(\norm{\bu}) \blambda_j(\bu) \quad \text{in $\Omega \setminus \Gamma(\bu)$},
\end{equation}
where, setting $\R_+:=[0,\infty)$, $s_j\:\R_+\setminus\{\bar u_j\}_{j=1}^{n-1} \to [0,1]$ is defined by
\begin{equation}
  \label{eq:selection-map}
  s_j(a) = \begin{cases} 0 & \text{if $a\not\in [\bar u_{j-1},\bar u_j]$},\\ 1 & \text{if $a\in(\bar u_{j-1},\bar u_j)$}.  \end{cases}
\end{equation}
We call $s_j$ the $i$th \emph{selection map} since it selects the drag operator to be used given the magnitude of the flux. As in the motivating example of Section~\ref{sec:motivating-example}, the map $\blambda(\bu)$ has a gap in its domain given by $\Gamma(\bu)$, and, analogously, the conservation of momentum in \eqref{eq:cons-mom} is updated to 
\begin{equation}
  \label{eq:cons-mom-gen}
  \blambda(\bu) = -\bgrad p + \bff \quad \text{in $\Omega \setminus \Gamma(\bu)$}.
\end{equation}
Problem~\ref{pb:strong-form} then becomes
\begin{pb}[strong form---jump drag operators]
  \label{pb:strong-form-jump}
  Find $\bu\:\Omega\to\R^d$ and $p\:\Omega\to\R$ such that
  \begin{equation*}
    \begin{cases}
      \dive\bu = q & \text{in $\Omega$},\\
      \blambda(\bu) = -\bgrad p + \bff & \text{in $\Omega\setminus\Gamma(\bu)$},\\
      \bu \cdot \bn = u_0 & \text{on $\Sigma_{\mt{v}}$},\\
      p = p_0 & \text{on $\Sigma_{\mt{p}}$},
    \end{cases}
  \end{equation*}
where $\blambda$ is of the form \eqref{eq:drag-si}.
\end{pb}

\paragraph{Multivalued formulation}

Problem~\ref{pb:strong-form-jump} is unconstrained in the transition zones since indeed we only impose the momentum conservation outside of these (cf.~\eqref{eq:cons-mom-gen}). This means in particular that the drag force is allowed to satisfy any relation in the transition zones. Although we do not wish to impose a transition drag force since we do not know a priori what it should be, this gap in the formulation of Problem~\ref{pb:strong-form-jump} is not satisfactory for at least two reasons:
\begin{itemize}
  \item it does not disappear when the family of laws $\{\blambda_j\}_{j=1}^n$ satisfy $\blambda_j = \blambda_{j+1}$ for all $j\in\{1,\dots,n-1\}$, so that, in this case, we do not recover the classical, continuous formulation of Problem~\ref{pb:strong-form};
  \item it is physically too permissive in admitting \emph{any} drag forces in the transition zones, while we expect these transition forces not to be ``too far'' from the surrounding, imposed ones.
\end{itemize}

To fix this issue, we propose an alternative version of Problem~\ref{pb:strong-form-jump} based on a set-valued extension of the selection maps in \eqref{eq:selection-map} to the threshold fluxes: for any $j\in\{1,\dots,n\}$, define $S_j\:\R_+\rightrightarrows [0,1]$ so that
\begin{equation}
  \label{eq:Sj-1}
  S_j(a) = \begin{cases} \{s_j(a)\} & \text{if $a\neq \bar u_j$},\\ [0,1] & \text{if $a=\bar u_j$},  \end{cases}
\end{equation}
with the additional condition that, if $j\neq n$, there holds
\begin{equation}
  \label{eq:Sj-2}
  S_j(a) + S_{j+1}(a) = \{1\}.
\end{equation}
Then, we define the \emph{multivalued} jump drag operator $\bLambda$ by
\begin{equation}
  \label{eq:drag-si-multiv}
  \bLambda(\bu) = \sum_{j=1}^n S_j(\norm{\bu}) \blambda_j(\bu).
\end{equation}
In particular, for all $j\in\{1,\dots,n\}$, $\bLambda$ satisfies
\begin{equation*}
  \bLambda(\bu) = \begin{cases} \{\blambda_j(\bu)\} & \text{in $\Omega_j(\bu)$},\\ \mt{conv}(\{\blambda_j(\bu),\blambda_{j+1}(\bu)\}) & \text{in $\Gamma_j(\bu)$}, \end{cases}
\end{equation*}
where $\mt{conv}(A)$ is the convex hull of set $A$. The associated problem is as follows:
\begin{pb}[multivalued strong form---jump drag operators]
  \label{pb:strong-form-jump-multiv}
  Find $\bu\:\Omega\to\R^d$ and $p\:\Omega\to\R$ such that
  \begin{equation*}
    \begin{cases}
      \dive\bu = q & \text{in $\Omega$},\\
      \bLambda(\bu) \ni -\bgrad p + \bff & \text{in $\Omega$},\\
      \bu \cdot \bn = u_0 & \text{on $\Sigma_{\mt{v}}$},\\
      p = p_0 & \text{on $\Sigma_{\mt{p}}$},
    \end{cases}
  \end{equation*}
where $\bLambda$ is of the form \eqref{eq:drag-si-multiv}.
\end{pb}

First, we note that any solution to Problem~\ref{pb:strong-form-jump-multiv} is a solution to Problem~\ref{pb:strong-form-jump}. Second, whenever $\blambda_j = \blambda_{j+1}$ for all $j\in\{1,\dots,n-1\}$, we recover the continuous formulation of Problem~\ref{pb:strong-form}. Third, transition drag forces are required to belong to the convex hull of the surrounding forces and thus stay somewhat ``close'' to them.

\subsubsection{General formulation}
\label{sec:general-formulation}

The discussion on jump operators leads us, for the remainder of the paper, to consider the following, general problem on any multivalued drag operator $\bLambda$:
\begin{pb}[multivalued strong form---general]
  \label{pb:strong-form-new-multiv}
  Find $\bu\:\Omega\to\R^d$ and $p\:\Omega\to\R$ such that
  \begin{equation*}
    \begin{cases}
      \dive\bu = q & \text{in $\Omega$},\\
      \bLambda(\bu) \ni -\bgrad p + \bff & \text{in $\Omega$},\\
      \bu \cdot \bn = u_0 & \text{on $\Sigma_{\mt{v}}$},\\
      p = p_0 & \text{on $\Sigma_{\mt{p}}$}.
    \end{cases}
  \end{equation*}
\end{pb}

Although the formulation of Problem~\ref{pb:strong-form-new-multiv} is very general and makes sense for drag operators involving space derviatives of the flux $\bu$, such as Stokes' and Brinkman's laws mentioned earlier (cf~\eqref{eq:laws-linear}), the analysis that we present below excludes such operators, which we leave for future investigation.

\section{Mathematical framework}
\label{sec:math-fram}

For all $\alpha\in(0,\infty)$, $\beta\in[1,\infty)$ and $A\subset \R^d$ measurable, we denote by $L^\beta(A)$ and $W^{\alpha,\beta}(A)$ the Lebesgue space of measurable functions on $A$ with integrable $\beta$th power and the $\alpha$th-order Sobolev space associated to $L^\beta(A)$; we also write $\bL^\beta(A)$ for $(L^\beta(A))^d$ and use $\norm{\cdot}_\beta$ for the canonical norm on $\bL^\beta(\Omega)$. As usual in these spaces, equality is intended in the almost everywhere sense.

Let $r\in(1,\infty)$ and write $s\in(1,\infty)$ its dual exponent, i.e., $s=r/(r-1)$. We fix $q\in L^r(\Omega)$, $\bff\in \bL^s(\Omega)$, $u_0\in L^{r}(\Sigma_{\mt{v}})$ and $p_0\in W^{\frac{1}{r},s}(\Sigma_{\mt{p}})$, and let $\bLambda\:\bL^r(\Omega)\rightrightarrows \bL^s(\Omega)$ be a multivalued drag operator so that, for all $\bu\in\bL^r(\Omega)$, we have $\bLambda(\bu) \neq \emptyset$. We first wish to derive a weak formulation for Problem~\ref{pb:strong-form-new-multiv} and then provide a well-posedness analysis for it.

\subsection{Weak formulation}
\label{sec:weak-formulation}

For any $a\in\R$ and $b\in W^{\frac{1}{r},s}(\Sigma_{\mt{p}})$, define the space
\begin{equation*}
  W_{a,b}^{1,s}(\Omega) := 
  \begin{cases}
    \left\{ \xi \in W^{1,s}(\Omega) \st \displaystyle \frac{1}{\abs{\Omega}}\int_\Omega \xi = a \right\} & \text{if $\mt{Vol}^{d-1}(\Sigma_{\mt{p}}) = 0$},\\
    \left\{ \xi \in W^{1,s}(\Omega) \st \xi = b\; \text{on $\Sigma_{\mt{p}}$} \right\} & \text{if $\mt{Vol}^{d-1}(\Sigma_{\mt{p}}) > 0$},
  \end{cases}
\end{equation*}
where we recall that $\mt{Vol}^{d-1}$ stands for the $(d-1)$-dimensional Lebesgue measure. In the sequel, we write $W_0^{1,s}(\Omega)$ for the Sobolev space $W_{0,0}^{1,s}(\Omega)$, which we endow with the norm $\norm{\psi}_{W_0^{1,s}(\Omega)} := \norm{\bgrad\psi}_s$ for all $\psi\in W_0^{1,s}(\Omega)$.

We can give a first weak formulation of Problem~\ref{pb:strong-form-new-multiv}:
\begin{pb}[weak form I]
  \label{pb:weak-form-bc}
    Find $(\bu,p) \in \bL^r(\Omega)\times W_{\bar p,p_0}^{1,s}(\Omega)$ so that there exists $\blambda\in\bLambda(\bu)$ satisfying
    \begin{equation*}
      \begin{gathered}
        \ap{\blambda,\bvarphi} = \ap{-\bgrad p + \bff,\bvarphi} \qquad \forall \bvarphi\in \bL^r(\Omega),\\
        \ap{\bgrad \psi,\bu} = -\int_\Omega q\psi + \int_{\Sigma_{\mt{v}}} u_0\psi
        \qquad \forall \psi\in W_0^{1,s}(\Omega),
    \end{gathered}
    \end{equation*}
    where $\ap{\cdot,\cdot}$ is the canonical dual pairing on $\bL^s(\Omega)\times\bL^r(\Omega)$ and $\bar p$ is as in Remark~\ref{rem:press-ave}.
\end{pb}
For simplicity, we want to remove the pressure boundary conditions $\bar p$ and $p_0$ from the formulation in Problem~\ref{pb:weak-form-bc}. For this, we set $\bff_0 = \bff$ if $\mt{Vol}^{d-1}(\Sigma_{\mt{p}}) = 0$ and $\bff_0 = \bff - \bgrad (E p_0)$ if instead $\mt{Vol}^{d-1}(\Sigma_{\mt{p}}) > 0$, with $E\: W^{\frac{1}{r},s}(\Sigma_{\mt{p}}) \rightarrow W^{1,s}(\Omega)$ any extension operator being right-inverse of the $W^{1,s}(\Omega)$ trace operator. By linearity with respect to pressure, Problem~\ref{pb:weak-form-bc} is equivalent to the following:
\begin{pb}[weak form II]
  \label{pb:weak-form-multiv}
    Find $(\bu,p) \in \bL^r(\Omega)\times W_0^{1,s}(\Omega)$ so that there exists $\blambda\in\bLambda(\bu)$ satisfying
    \begin{equation*}
      \begin{gathered}
        \ap{\blambda,\bvarphi} = \ap{-\bgrad p + \bff_0,\bvarphi} \qquad \forall \bvarphi\in \bL^r(\Omega),\\
        \ap{\bgrad \psi, \bu} = -\int_\Omega q\psi + \int_{\Sigma_{\mt{v}}} u_0\psi
        \qquad \forall \psi\in W_0^{1,s}(\Omega).
    \end{gathered}
    \end{equation*}
\end{pb}
\noindent Also note that uniqueness for Problem~\ref{pb:weak-form-bc} holds if and only if it does for Problem~\ref{pb:weak-form-multiv}.

\subsection{Well-posedness}
\label{sec:well-posedness}

We study now the well-posedness of Problem~\ref{pb:weak-form-multiv}. We first state the main results and then provide the proofs. 

For the various notions on multivalued operators used in the following statements and proofs, we refer the reader to Appendix~\ref{sec:noti-mult-oper}. In particular, note that we reserve the term ``continuous'' to monovalued operators and use ``set-continuous'' for possibly multivalued operators; although this choice is nonstandard, we make it to distinguish clearly the classical, monovalued framework (referred so far as continuous) from the new, multivalued setting (referred so far as discontinuous).

\begin{thm}[well-posedness---monotone operator]
  \label{thm:wp-monotone}
  Suppose that the drag operator $\bLambda$ is maximal monotone, and $\sigma$-coercive and $\sigma$-bounded for some $\sigma\geq1$. Then, Problem~\ref{pb:weak-form-multiv} has a solution $(\bu,p)$. If furthermore $\bLambda$ is strictly monotone, then $\bu$ is unique; if in addition it is monovalued, then $(\bu,p)$ is unique.
\end{thm}

\begin{rem}[non-monotone case]
 \label{rem:nonmonotone-case}
  Theorem \ref{thm:wp-monotone} only applies to monotone drag operators, which seem to be the most commonly used in the continuous setting as the examples in Section~\ref{sec:cont-drag-oper} indicate. However, in the discontinuous setting illustrated by the jump operators of the form discussed in Section~\ref{sec:jump-drag-operators}, this is less so since jumping from a low-speed region to a high-speed one could co-occur with a drop in the drag force and thus invalidate monotonicity. When $d>1$, we leave the well-posedness analysis of the non-monotone case to a future investigation, as it involves nonconvex analytical tools which we do not wish to consider here for concision. When $d=1$, these tools are not needed and the non-monotone case is included in Theorem~\ref{thm:wp-dim-one-neumann} below.
\end{rem}

The following theorem ensures well-posedness, or at least existence, for very general drag operators, as opposed to only monotone ones, when $d=1$. Note that we drop the boldface notation when we work specifically in dimension one.

\begin{thm}[well-posedness---dimension one]
  \label{thm:wp-dim-one-neumann}
  Let $d=1$. We identify two cases:
  \begin{enumerate}[label=(\roman*)]
    \item $\mt{Vol}^0(\Sigma_{\mt{v}})>0$. Then, Problem~\ref{pb:weak-form-multiv} has a solution $(u,p)$ such that $u$ is unique. If $\Lambda$ is monovalued, then $p$ also is unique.
    \item $\mt{Vol}^0(\Sigma_{\mt{v}})=0$. Suppose that $\Lambda$ is set-continuous and that $\Lambda(u)$ is a convex set for all $u\in L^r(\Omega)$. Assume moreover that $\Lambda$ is $\sigma$-coercive and $\sigma$-bounded for some $\sigma\geq1$. Then, Problem~\ref{pb:weak-form-multiv} has a solution.
  \end{enumerate}
\end{thm}

\subsubsection{Preliminaries}
\label{sec:preliminaries}

Write $V\subset\bL^r(\Omega)$ and $V^\perp\subset\bL^s(\Omega)$ the sets defined as
\begin{equation}
  \label{eq:V}
  \begin{gathered}
    V = \{ \bvarphi \in \bL^r(\Omega) \st \forall\, \psi\in W_0^{1,s}(\Omega),\, \ap{\bgrad \psi,\bvarphi} = 0 \},\\
    V^\perp = \{ \bg \in \bL^s(\Omega) \st \forall\, \bvarphi\in V,\, \ap{\bg,\bvarphi} = 0\}.
  \end{gathered}
\end{equation}
The set $V^\perp\subset\bL^s(\Omega)$ is often referred to as the polar space or annihilator of $V$. Naturally, we equip $V$ and $V^\perp$ with the respective canonical norms $\norm{\cdot}_r$ and $\norm{\cdot}_s$. We have the two lemmas below whose proofs can be found in~\cite{AFM18,FP21}.
\begin{lem}
  \label{lem:isomorphism}
  The gradient map $\bgrad\:W_0^{1,s}(\Omega) \to V^\perp$ is an isomorphism.
\end{lem}
\begin{lem}
  \label{lem:bdry-V0}
  There exists a unique $[\hat\bu]\in \bL^r(\Omega)/V$ such that
  \begin{equation*}
    \ap{\bgrad \psi,\hat\bu} = -\int_\Omega q\psi + \int_{\Sigma_{\mt{v}}} u_0\psi \quad \text{for all $\psi\in W_0^{1,s}(\Omega)$},
  \end{equation*}
where $\bL^r(\Omega)/V$ stands for the quotient space of $\bL^r(\Omega)$ by $V$.
\end{lem}

We can now reformulate Problem~\ref{pb:weak-form-multiv} as a problem restricted to $V$. To this end, we first introduce the following definition:
\begin{defn}[restricted drag operator]
  \label{defn:restricted-drag}
  We call \emph{restricted drag operator} the multivalued map $\bLambda_V^*\:V\rightrightarrows V^*$ defined by
  \begin{equation*}
    \bLambda_V^*(\bv) = \{ \blambda^*\in V^* \st \exists\, \blambda\in \bLambda(\hat\bu + \bv),\, \forall\, \bvarphi\in V,\; \blambda^*(\bvarphi) = \ap{\blambda,\bvarphi} \} \quad \text{for all $\bv \in V$},
  \end{equation*}
  where $\hat\bu$ is as in Lemma~\ref{lem:bdry-V0}.
\end{defn}
\noindent Then, we consider a corresponding restricted problem, which we show right away is equivalent to Problem~\ref{pb:weak-form-multiv}:
\begin{pb}[restricted multivalued form]
  \label{pb:restricted-V}
  Find $\bv\in V$ so that there is $\blambda^*\in \bLambda_V^*(\bv)$ satisfying
  \begin{equation*}
    \blambda^*(\bvarphi) = \ap{\bff_0,\bvarphi} \quad \text{for all $\bvarphi\in V$}.
  \end{equation*}
\end{pb}

\begin{lem}
  \label{lem:pb-restr}
  Problems~\ref{pb:weak-form-multiv} and~\ref{pb:restricted-V} are equivalent.
\end{lem}
\begin{proof}
  We first suppose that $(\bu,p)$ is a solution to Problem~\ref{pb:weak-form-multiv}. We decompose $\bu$ as $\bu = \hat\bu + (\bu - \hat\bu) =: \hat\bu + \bv$. By Problem~\ref{pb:weak-form-multiv}, we directly get there exists $\blambda \in \bLambda(\hat\bu+\bv)$ so that
  \begin{equation*}
    \ap{\blambda,\bvarphi} = \ap{\bff_0,\bvarphi} \quad \text{for all $\bvarphi \in V$}.
  \end{equation*}
Moreover, we check that $\ap{\bgrad\psi,\bv} = \ap{\bgrad\psi,\bu} - \ap{\bgrad\psi,\hat\bu} = 0$ for all $\psi\in W_0^{1,s}(\Omega)$, so $\bv\in V$. Then, the map $\blambda^*\in V^*$ defined by $\blambda^*(\bvarphi) = \ap{\blambda,\bvarphi}$ for all $\bvarphi\in V$ satisfies $\blambda^*\in \bLambda_V^*(\bv)$. We deduce that $\bv$ satisfies Problem~\ref{pb:restricted-V}.

  Suppose now that $\bv$ satisfies Problem~\ref{pb:restricted-V} and write $\bu = \hat\bu+\bv$. Then, one can find $\blambda\in\bLambda(\bu)$ so that $\blambda - \bff_0 \in V^{\mathrm{\perp}}$. By Lemma~\ref{lem:isomorphism}, we know $\bgrad$ is an isomorphism from $W_0^{1,s}(\Omega)$ to $V^{\mathrm{\perp}}$, and thus there exists a unique $p\in W_0^{1,s}(\Omega)$ such that
  \begin{equation*}
    \ap{\blambda - \bff_0,\bvarphi} = -\ap{\bgrad p,\bvarphi} \quad \text{for all $\bvarphi\in\bL^r(\Omega)$}.
  \end{equation*}
Furthermore, using $\bv\in V$, we have
\begin{equation*}
  \ap{\bgrad\psi,\bu} = \ap{\bgrad\psi,\hat\bu} + \ap{\bgrad\psi,\bv} = \ap{\bgrad\psi,\hat\bu}  = - \int_\Omega q\psi + \int_{\Sigma_{\mt{v}}} u_0\psi.
\end{equation*}
We therefore get that $(\bu,p)$ satisfies Problem~\ref{pb:weak-form-multiv}.
\end{proof}

\subsubsection{Proof of Theorem~\ref{thm:wp-monotone}}
\label{sec:proof-wp-monotone}

We first show existence and then uniqueness. For the existence, we make use of the following theorem from operator analysis:
\begin{thm}[Browder~{\cite[Theorem~3]{Browder68}}]
  \label{thm:browder}
  Let $X$ be a reflexive real Banach space with strictly convex topological dual $X^*$, and suppose that $A\:X\rightrightarrows X^*$ is maximal monotone and $1$-coercive. Then, the range of $A$ equals $X^*$.
\end{thm}

\paragraph{Existence}
We wish to apply Theorem~\ref{thm:browder} to $\bLambda_V^*$ and then use Lemma~\ref{lem:pb-restr}. 

We first note that $\bLambda_V^*$ is $1$-coercive. Indeed, for $(\bv,\blambda^*)\in \graph(\bLambda_V^*)$ and $\bu:=\hat\bu+\bv$, and for some $\blambda\in \bL^s(\Omega)$ such that $(\bu,\blambda)\in \graph(\bLambda)$, Hölder's inequality leads to
\begin{align*}
  \blambda^*(\bv) &= \ap{\blambda,\bu} - \ap{\blambda,\hat\bu}\\
                    &\geq c(\norm{\bu}_r)\norm{\bu}_r^{\sigma} - \norm{\blambda}_s\norm{\hat\bu}_r\\
                    &\geq \left( c(\norm{\bu}_r) - C\norm{\hat\bu}_r \right)\norm{\bu}_r^{\sigma} - C\norm{\hat\bu}_r\\
                  &= \left( c(\norm{\hat\bu+\bv}_r) - C\norm{\hat\bu}_r \right)\norm{\hat\bu+\bv}_r^{\sigma-1}\norm{\hat\bu+\bv}_r - C\norm{\hat\bu}_r,
\end{align*}
where $c\:\R_+\to\R$ is such that $c(a)\to\infty$ as $a\to\infty$ and $C>0$ (cf. Definitions~\ref{defn:coercive-op} and~\ref{defn:bounded-op}); rearranging terms and redefining $c$ adequately, we find
\begin{equation*}
  \blambda^*(\bv) = \ap{\blambda,\bv} \geq c(\norm{\bv}_r)\norm{\bv}_r,
\end{equation*}
where again $c\:\R_+\to\R$ satisfies $c(a)\to\infty$ as $a\to\infty$.

To prove maximal monotonicity of $\bLambda_V^*$, let us define $\bLambda^* \: \bL^r(\Omega) \rightrightarrows (\bL^r(\Omega))^*$ by
\begin{equation*}
    \bLambda^*(\bu) = \{ \bs{\ell}^*\in (\bL^r(\Omega))^* \st \exists\, \blambda\in \bLambda(\hat\bu + \bu),\, \forall\, \bvarphi\in \bL^r(\Omega),\; \bs{\ell}^*(\bvarphi) = \ap{\blambda,\bvarphi} \},
  \end{equation*}
for all $\bu \in \bL^r(\Omega)$. Note that $\bLambda_V^*$ is the restriction of $\bLambda^*$ to $V$ in the sense that, for all $\bv\in V$, there holds 
\begin{equation*}
  \bLambda_V^*(\bv) = \{ \blambda^*\in V^*\st \exists\, \bs{\ell}^* \in \bLambda^*(\bv),\, \forall\, \bvarphi\in V,\; \blambda^*(\bvarphi) = \bs{\ell}^*(\bvarphi)\}.
\end{equation*}
As we assume that $\bLambda$ is maximal monotone, so is $\bLambda^*$. Thus, writing $\bs{N}_V\: \bL^r(\Omega) \rightrightarrows (\bL^r(\Omega))^*$ the normal cone of $V$, i.e., $\bs{N}_V(\bu) = \{\bs{n}^*\in (\bL^r(\Omega))^* \st \forall\, \bv\in V,\;  \bs{n}^*(\bu) \geq \bs{n}^*(\bv) \}$ for all $\bu\in\bL^r(\Omega)$, it follows from~\cite[Corollary~15]{Borwein07} that the operator sum $\bLambda^* + \bs{N}_V$ is also maximal monotone. Therefore, by \cite[Lemma~1]{Voisei11}, we get that $\bLambda_V^*$ is maximal monotone.

To show that $V^*$ is strictly convex, i.e., that the unit ball in $V^*$ is strictly convex, note that $V^*$ is isometrically isomorphic with $\bL^s(\Omega)/V^\perp$ (the quotient space of $\bL^s(\Omega)$ by $V^\perp$) via the linear map $T\:\bL^s(\Omega)/V^\perp\to V^*$ defined by
\begin{equation*}
  T([\bg])(\bv) = \ap{\bg,\bv} \quad \text{for all $[\bg]\in \bL^s(\Omega)/V^\perp$ and $\bv\in V$}.
\end{equation*}
Then, because $\bL^s(\Omega)/V^\perp$ is strictly convex~\cite[Proposition~3.2]{Klee59}, we get that $V^*$ is also strictly convex. 

We can now use Theorem~\ref{thm:browder}. Write $\bff_V\in V^*$ the map defined by $\bff_V(\bvarphi) = \ap{\bff_0,\bvarphi}$ for all $\bvarphi\in V$. Then, Theorem~\ref{thm:browder} yields the existence of $\bv\in V$ such that there exists $\blambda^*\in\bLambda_V^*(\bv)$ with $\blambda^*(\bvarphi) = \bff_V(\bvarphi) = \ap{\bff_0,\bvarphi}$ for all $\bvarphi\in V$, which means that $\bv$ is solution to Problem~\ref{pb:restricted-V}. Finally, by Lemma~\ref{lem:pb-restr}, we conclude that Problem~\ref{pb:weak-form-multiv} has a solution.

\paragraph{Uniqueness}
When $\bLambda$ is strictly monotone, the uniqueness of the flux is direct by Definition~\ref{defn:monotone-op}. If in addition $\bLambda$ is monovalued, then the unique solution $\bv$ to Problem~\ref{pb:restricted-V} yields a unique $\blambda$ with $(\hat\bu+\bv,\blambda)\in \graph(\bLambda)$ such that $\blambda - \bff_0\in V^\perp$. We then deduce the uniqueness of the pressure by following the second part of the proof of Lemma~\ref{lem:pb-restr}.

\paragraph{Remark}
From the above proof, we note that if $q$, $u_0$ and $\Sigma_v$ are such that
  \begin{equation*}
    \int_\Omega q\psi = \int_{\Sigma_{\mt{v}}} u_0\psi \quad \text{for all $\psi\in W_0^{1,s}(\Omega)$},
  \end{equation*}
then the $\sigma$-boundedness condition on $\bLambda$ in Theorem~\ref{thm:wp-monotone} can be removed, since in the case we can choose $\hat\bu = \bnull$.

\subsubsection{Proof of Theorem~\ref{thm:wp-dim-one-neumann}}
\label{sec:proof-wp-dim-one-neumann}

Suppose here that $d=1$.

\paragraph{Case $\mt{Vol}^{d-1}(\Sigma_{\mt{v}})>0$}
We know that $V$ and $V^*$ are trivial, that is, $V=\{0\}$ and $V^*=\{0\}$ (cf.~\cite[Theorem~4.11]{FP21}). 

Problem~\ref{pb:restricted-V} is trivially and uniquely solved for $v=0$. Then, Lemma~\ref{lem:pb-restr} yields the existence of a solution to Problem~\ref{pb:weak-form-multiv} with the uniqueness of the flux. The uniqueness of the pressure when $\Lambda$ is monovalued directly follows from the second part of the proof of Lemma~\ref{lem:pb-restr}.
 
\paragraph{Case $\mt{Vol}^{d-1}(\Sigma_{\mt{v}})=0$}
It holds that $V$ and $V^*$ are isomorphic with $\R$ (cf.~\cite[Theorem~4.11]{FP21}). Consequently, for all $v\in \R$, the set $\Lambda_V^*(v)$ is isomorphic with a subset of $\R$, which we denote by $\Lambda_\R^*(v)$. Problem~\ref{pb:restricted-V} simplifies into the following: find $v\in \R$ such that $\hat f_0 \in \Lambda_\R^*(v)$, where $\hat f_0$ stands for $(1/|\Omega|) \int_\Omega f_0$.

Following the same arguments as in Section~\ref{sec:proof-wp-monotone}, we know that $\Lambda_\R^*$ is $1$-coercive. Thus, there exists $v_1\in\R$ such that $\lambda^*<\hat f_0$ for all $\lambda^*\in\Lambda_\R^*(v_1)$ and there exists $v_2\in\R$ such that $\lambda^*>\hat f_0$ for all $\lambda^*\in\Lambda_\R^*(v_2)$. By the assumed continuity and convexity property of $\Lambda$, the multivalued operator $\Lambda_\R^*$ is continuous and $\Lambda_\R^*(v)$ is convex for all $v\in \R$. Thus, $\Lambda_\R^*$ has the Darboux property (cf.~\cite[Theorems~1 and~2]{CK92}) and so there exists $v\in[v_1,v_2]$ such that $\hat f_0 \in \Lambda_\R^*(v)$. Hence Problem~\ref{pb:restricted-V} admits a solution and Lemma~\ref{lem:pb-restr} gives the existence of a solution to Problem~\ref{pb:weak-form-multiv}.

\section{Well-posedness for dissipative drag operators}
\label{sec:well-posedn-diss}

We pick $r\in[2,\infty)$ and let $s\in (1,2]$ be the dual exponent of $r$, that is, $s=r/(r-1)$. As in Section~\ref{sec:math-fram}, we fix $q\in L^r(\Omega)$, $\bff\in \bL^s(\Omega)$, $u_0\in L^{r}(\Sigma_{\mt{v}})$ and $p_0\in W^{\frac{1}{r},s}(\Sigma_{\mt{p}})$, and we let $\bLambda\:\bL^r(\Omega)\rightrightarrows \bL^s(\Omega)$ be such that, for all $\bu\in\bL^r(\Omega)$, there holds $\bLambda(\bu)\neq \emptyset$. 

Furthermore, we write $\bs{M}_{\mt{sym}}^+$ the set of $d\times d$, real, symmetric, positive definite matrices. Given $\mathbb{M}\in\bs{M}_{\mt{sym}}^+$, we write $\norm{\cdot}_{\mathbb{M}}$ the Euclidean norm weigthed by $\mathbb{M}$, that is, for all $\bx\in\R^d$,
\begin{equation*}
  \norm{\bx}_{\mathbb{M}} = \sqrt{\mathbb{M}\bx\cdot\bx};
\end{equation*}
denoting the spectrum of $\mathbb{M}$ by $\mt{Sp}(\mathbb{M})$, we have
\begin{equation*}
  \sqrt{\min \mt{Sp}(\mathbb{M})}\norm{\bx} \leq \norm{\bx}_{\mathbb{M}}\leq \sqrt{\max \mt{Sp}(\mathbb{M})}\norm{\bx}.
\end{equation*}
When $\mathbb{M}\:\Omega \to \bs{M}_{\mt{sym}}^+$ and $\bs{\phi}\:\Omega\to \R^d$, we write $\norm{\bs{\phi}}_{\mathbb{M}}$ the map $\bx\mapsto \norm{\bs{\phi}(\bx)}_{\mathbb{M}(\bx)}$ and we set
\begin{equation*}
  \sup\mathbb{M} = \sup_{\bx\in\Omega} \max \mt{Sp}(\mathbb{M}(\bx)) \quad \text{and} \quad \inf\mathbb{M} = \inf_{\bx\in\Omega} \min \mt{Sp}(\mathbb{M(\bx)}).
\end{equation*}

\subsection{Underlying assumption}
\label{sec:underly-assumpt}

We wish to apply the well-posedness results of Section~\ref{sec:well-posedness} to a specific type of drag operators which we refer to as \emph{dissipative}, since they can be derived from an underlying functional called the dissipation (cf. Appendix~\ref{sec:functionals} for the notions on functionals used below). In~\cite{SV01_2}, the authors consider such dissipative operators, and, as a multivalued generalization of their model, we consider the following assumption:

\begin{assu}[dissipative drag operator]
  \label{assu:diss-drag}
  Fix $\mathbb{D}\:\Omega \to \bs{M}_{\mt{sym}}^+$ so that $\sup\mathbb{D}<\infty$ and $\inf\mathbb{D}>0$. Let $\bLambda$ be of the form
  \begin{equation*}
    \bLambda(\bu) = \Phi(\norm{\bu}_{\mathbb{D}}^2)\,\mathbb{D}\bu \quad \text{for all $\bu\in \bL^r(\Omega)$},
  \end{equation*}
  where the multivalued map $\Phi\:\R_+\rightrightarrows \R_+$ is set-continuous and such that $\Phi(a) \neq \emptyset$ for all $a\geq0$. Let $c,C>0$ be such that
  \begin{equation}
    \label{eq:cont-upper-bound}
    ca^{\frac{r-2}{2}} \leq \phi \leq C \left( 1+ a^{\frac{r-2}{2}} \right) \quad \text{for all $(a,\phi) \in\graph(\Phi)$},
  \end{equation}
  and let $\Psi\:\R_+ \to \R$ be locally Lipschitz continuous with $\p\Psi = \Phi$.
\end{assu}

Recall the subdifferential chain rule which, under Assumption~\ref{assu:diss-drag}, ensures that $\p(\Psi\circ(\cdot)^2)(a) =2a\Phi(a^2)$ for all $a\geq0$ (cf.~\cite{LSM20} for instance).

\subsubsection{Preliminary check}
\label{sec:preliminary-check}

We want check that, when $\bLambda$ satisfies Assumption~\ref{assu:diss-drag}, the upper bound in \eqref{eq:cont-upper-bound} ensures that $\bLambda$ indeed maps $\bL^r(\Omega)$ to $\bL^s(\Omega)$. For this, take $\bu\in\bL^r(\Omega)$ and $\blambda\in\bLambda(\bu)$, fix $\bx\in\Omega$, let $\phi\in \Phi(\norm{\bu(\bx)}_{\mathbb{D}(\bx)}^2)$ be such that $\blambda(\bx)=\phi\,\mathbb{D}(\bx)\bu(\bx)$, write $M=\sup\mathbb{D}$, and compute
\begin{align*}
  \norm{\blambda(\bx)} &\leq \phi \norm{\mathbb{D}(\bx)\bu(\bx)} \leq M \norm{\bu(\bx)}\phi\\
  &\leq M C \norm{\bu(\bx)} \left( 1+\norm{\bu(\bx)}_{\mathbb{D}(\bx)}^{r-2} \right)\\
  &\leq M C \norm{\bu(\bx)} \left( 1+M^{\frac{r-2}{2}} \norm{\bu(\bx)}^{r-2} \right)\\
  &\leq M C \max\left\{ 1,M^{\frac{r-2}{2}} \right\} \left( \norm{\bu(\bx)}+\norm{\bu(\bx)}^{r-1} \right),
\end{align*}
where $C>0$ is as in \eqref{eq:cont-upper-bound}. Then, taking this computation to the power of $s$, we get
\begin{align*}
  \norm{\blambda(\bx)}^s &\leq \left( M C \max\left\{ 1,M^{\frac{r-2}{2}} \right\} \right)^s \left( \norm{\bu(\bx)}+\norm{\bu(\bx)}^{r-1} \right)^s\\
  &\leq 2^{1-s} \left( M C \max\left\{ 1,M^{\frac{r-2}{2}} \right\} \right)^s \left( \norm{\bu(\bx)}^r+\norm{\bu(\bx)}^s \right),
\end{align*}
where the second inequality is obtained using the identity $(a+b)^s \leq 2^{1-s}(a^s+b^s)$ whenever $a,b\geq0$. Hence
\begin{equation}
  \label{eq:bounded-lambda}
  \norm{\blambda(\bx)}^s \leq C' \left( \norm{\bu(\bx)}^r + \norm{\bu(\bx)}^s \right),
\end{equation}
with $C'>0$ defined appropriately. Since $r\geq2$, and so $r\geq s$, this shows that indeed $\blambda\in\bL^s(\Omega)$ and so $\bLambda\:\bL^r(\Omega)\rightrightarrows\bL^s(\Omega)$.

\subsubsection{Dissipation}
\label{sec:dissipation}

Supposing that Assumption~\ref{assu:diss-drag} holds, we define the functional $\D\:\bL^r(\Omega) \to \R$, called the \emph{dissipation}, by
\begin{equation}
  \label{eq:dissipation-cont}
  \D(\bu) = \frac12 \int_\Omega \Psi(\norm{\bu}_{\mathbb{D}}^2) \quad \text{for all $\bu\in\bL^r(\Omega)$}.
\end{equation}
Physically, the dissipation represents the mechanical power lost through drag by the fluid to the rock matrix (cf.~\cite{SV01}). We want to check that the dissipation is indeed well defined under Assumption~\ref{assu:diss-drag}. In fact, we have the following lemma:
\begin{lem}
  \label{lem:psi-diss}
  Let Assumption~\ref{assu:diss-drag} hold. Then, $\Psi$ is nondecreasing, $(r/2)$-bounded and $(r/2)$-coercive. Moreover, the dissipation $\D$, as given in \eqref{eq:dissipation-cont}, is well defined and there exist $\tilde c_0,\tilde c_1,\tilde C>0$ such that
  \begin{equation}
    \label{eq:bounds-D}
    \tilde c_1\norm{\bu}_r^r - \tilde c_0 \leq \D(\bu) \leq \tilde C(1+\norm{\bu}_r^r) \quad \text{for all $\bu\in\bL^r(\Omega)$}
  \end{equation}
\end{lem}
\begin{proof}
  Note that, thanks to~\cite[Theorem~1.3]{Chieu10}, we have
\begin{equation}
  \label{eq:funda-calc}
  \Psi(a) = \Psi(0) + \int_0^a \phi(b)\d b \quad \text{for all $a\geq0$},
\end{equation}
where $\phi\:\R_+\to\R_+$ is a map with $\phi(b) \in \Phi(b)$ for all $b\geq0$. Since, by assumption, $\Phi(b)\subset\R_+$ for all $b\geq0$, we get by \eqref{eq:funda-calc} that $\Psi$ is nondecreasing. Moreover, by the right-hand inequality in \eqref{eq:cont-upper-bound} and again \eqref{eq:funda-calc}, we yield
\begin{equation}
  \label{eq:psi-bdd}
  \abs{\Psi(a)} \leq \abs{\Psi(0)} + Ca\left(1+a^{\frac{r-2}{2}}\right)\leq \left(\abs{\Psi(0)} + 2C\right)\left(1+a^{\frac r2}\right) \quad \text{for all $a\geq0$},
\end{equation}
which shows that $\Psi$ is $(r/2)$-bounded. Similarly, using this time the left-hand inequality in \eqref{eq:cont-upper-bound},
\begin{equation}
  \label{eq:psi-coerc}
  \Psi(a) \geq \Psi(0) + \frac{2c}{r} a^{\frac{r}{2}} \quad \text{for all $a\geq0$},
\end{equation}
which gives the $(r/2)$-coercivity of $\Psi$.

The inequality in \eqref{eq:psi-bdd} directly yields that $\D$ is well defined and that the right-hand inequality in \eqref{eq:bounds-D} holds. The left-hand inequality is obtained using \eqref{eq:psi-coerc}. 
\end{proof}

Let us now establish the relation between $\D$ and $\bLambda$, which also justifies why an operator satisfying Assumption~\ref{assu:diss-drag} may be called dissipative:
\begin{lem}
  \label{lem:drag-subdiff}
  Let Assumption~\ref{assu:diss-drag} hold. Then, $\D$ is locally Lipschitz continuous and $\bp\D \subset \bLambda$. If $\Psi\circ(\cdot)^2$ is convex on $\R_+$, then $\D$ is convex and $\bLambda=\bp\D$. If furthermore $\Psi$ is strictly increasing, then $\D$ is strictly convex.
\end{lem}
\begin{proof}
  Define $\eta\:\R^d\to \R$ by
\begin{equation*}
  \eta(\ba) = \frac12 \Psi(\norm{\ba}_{\mathbb{D}}^2) \quad \text{for all $\ba\in\R^d$},
\end{equation*}
so that, in particular, 
\begin{equation*}
  \D(\bu) = \int_\Omega \eta(\bu) \quad \text{for all $\bu\in\bL^r(\Omega)$}.
\end{equation*} 

The Lipschitz continuity of $\D$ is obtained from~\cite[Proposition~12]{GP17} in combination with the right-hand inequality in \eqref{eq:bounds-D}. For all $\bu\in\bL^r(\Omega)$, note that
\begin{equation}
  \label{eq:subdiff-1}
  \bLambda(\bu) = \left\{ \blambda\in\bL^s(\Omega) \st \forall\, \bx\in \Omega,\; \blambda(\bx)\in \bp\eta(\bu(\bx)) \right\};
\end{equation}
the fact that $\bp\D(\bu) \subset \bLambda(\bu)$ then directly follows from~\cite[Section~3]{Giner98}.

Assume now $\Psi\circ(\cdot)^2$ is convex on $\R_+$. Since $\Psi$ is nondecreasing (cf. Lemma~\ref{lem:psi-diss}), we know that $\Psi\circ(\cdot)^2$ is nondecreasing on $\R_+$ and we get that $\eta$ is convex, and so $\D$ is convex. (In the case when we also have that $\Psi$ is strictly increasing, the analogous argument leads to $\D$ strictly convex.) Let now $\blambda\in\bLambda(\bu)$ for some $\bu\in\bL^r(\Omega)$. Then, from \eqref{eq:subdiff-1}, we have $\blambda(\bx)\in \bp\eta(\bu(\bx))$ for all $\bx\in\Omega$; by Proposition~\ref{prop:clarke}, for all $\bv\in\bL^r(\Omega)$, there holds 
\begin{equation*}
  \liminf_{\delta\downarrow 0} \frac{\eta(\bu(\bx)+\delta \bv(\bx)) - \eta(\bu(\bx))}{\delta} \geq \blambda(\bx)\cdot\bv(\bx).
\end{equation*}
Taking the integral over $\Omega$ of the above and applying Fatou's lemma, we get
\begin{equation*}
  \liminf_{\delta\downarrow 0} \int_\Omega \frac{\eta(\bu(\bx)+\delta \bv(\bx)) - \eta(\bu(\bx))}{\delta} \d\bx \geq \int_\Omega \blambda(\bx)\cdot\bv(\bx) \d\bx,
\end{equation*}
and so
\begin{equation*}
  \liminf_{\delta\downarrow 0} \frac{\D(\bu+\delta \bv) - \D(\bu)}{\delta} \geq \int_\Omega \blambda(\bx)\cdot\bv(\bx) \d\bx,
\end{equation*}
which shows that $\blambda\in\bp\D(\bu)$, i.e., $\bLambda(\bu)\subset \bp\D(\bu)$, which concludes the proof.
\end{proof}

\subsection{Results}
\label{sec:results}

We now present the well-posedness corollaries following from Theorems~\ref{thm:wp-monotone} and~\ref{thm:wp-dim-one-neumann} under the assumption of a dissipative drag force. We split the results depending on whether the dissipation is convex or not.

\begin{cor}[well-posedness---convex case]
  \label{cor:scal-drag-coeff}
  Let Assumption~\ref{assu:diss-drag} hold and $\Psi\circ(\cdot)^2$ be convex on $\R_+$. Then, Problem~\ref{pb:weak-form-multiv} has a solution $(\bu,p)$. If moreover $\Psi$ strictly increasing, then $\bu$ is unique, and if it is also differentiable, then $(\bu,p)$ is unique.
\end{cor}
\begin{proof}
  We want to use Theorem~\ref{thm:wp-monotone} on $\bLambda$. 

We start by proving the coercivity and boundedness of $\bLambda$. Let $\bu\in\bL^r(\Omega)$ and $\blambda\in\bLambda(\bu)$. Let $\phi\:\Omega\to\R$ be such that $\phi(\bx)\in \Phi(\norm{\bu(\bx)}_{\mathbb{D}(\bx)}^2)$ and $\blambda(\bx) = \phi(\bx)\, \mathbb{D}(\bx)\bu(\bx)$ for all $\bx\in\Omega$. Then, 
\begin{align*}
  \ap{\blambda,\bu} &= \int_\Omega \phi\, \mathbb{D}\bu\cdot\bu \geq \int_\Omega \phi \norm{\bu}_{\mathbb{D}}^2\geq c \int_\Omega \norm{\bu}_{\mathbb{D}}^r\\
  &\geq c\, (\inf\mathbb{D})^{\frac r2}  \int_\Omega \norm{\bu}^r = c\, (\inf\mathbb{D})^{\frac r2}\norm{\bu}_r\norm{\bu}_r^{r-1}.
\end{align*}
Hence $\bLambda$ is $(r-1)$-coercive, with $r-1\geq 1$ (since, in this section, $r\geq2$ by assumption). Furthermore, following the same steps leading to \eqref{eq:bounded-lambda}, we get
\begin{align*}
  \norm{\blambda}_s^s &\leq C' \int_\Omega \left( \norm{\bu}^r + \norm{\bu}^s \right) = C' \left( \norm{\bu}_r^r + \int_\Omega \norm{\bu}^s \right)\\
  &\leq C' \left( \norm{\bu}_r^r + \abs{\Omega}^{\frac{r-2}{r-1}} \norm{\bu}_r^{\frac{r}{r-1}} \right),
\end{align*}
so that, redefining $C'$ as needed, we yield
\begin{equation}
  \label{eq:lambda-s}
  \norm{\blambda}_s \leq C' \left( \norm{\bu}_r^{r-1} + \abs{\Omega}^{\frac{r-2}{r}} \norm{\bu}_r \right) \leq C'\left(1+\abs{\Omega}^{\frac{r-2}{r}}\right)\left(1+\norm{\bu}_r^{r-1}\right),
\end{equation}
and $\bLambda$ is $(r-1)$-bounded. 

We now turn to showing maximal monotonicity of $\bLambda$. Note that, since the dissipation $\D$ is convex and lower semicontinuous (cf. Lemma~\ref{lem:drag-subdiff}), the subdifferential of $\D$ is maximal monotone by~\cite[Theorem~A]{Rockafellar70}. Then, Lemma~\ref{lem:drag-subdiff} gives that $\bLambda$ is  maximal monotone. Theorem~\ref{thm:wp-monotone} thus shows that Problem~\ref{pb:weak-form-multiv} has a solution. 

For the uniqueness part, notice that strict convexity of $\D$ implies strict monotonicity of $\bLambda$ and that differentiability of $\D$ yields monovaluedness of $\bLambda$. Theorem~\ref{thm:wp-monotone} then directly gives the result.
\end{proof}

We now turn to the case when the dissipation is not convex. As mentioned in Remark~\ref{rem:nonmonotone-case}, we only cover $d=1$ and leave $d>1$ to an upcoming work.

\begin{cor}[existence---nonconvex case]
  \label{cor:noncvx-drag}
  Let $d=1$ and let $\Lambda$ satisfy Assumption~\ref{assu:diss-drag}. Then, Problem~\ref{pb:weak-form-multiv} has a solution.
\end{cor}
\begin{proof}
  We want to apply Theorem~\ref{thm:wp-dim-one-neumann} to $\Lambda$. Since the coercivity and boundedness of $\Lambda$ follow exactly as in the proof of Corollary~\ref{cor:scal-drag-coeff}, we only have to show that $\Lambda$ is set-continuous and $\Lambda(u)$ is convex for all $u\in L^r(\Omega)$.

Since $\Phi$ is assumed to be set-continuous, the set-continuity of $\Lambda$ is direct. Moreover, since $\Phi = \p\Psi$, it is a fact that $\Phi(a)$ is a convex set for all $a\geq0$ (cf. Proposition~\ref{prop:clarke}) and it follows that $\Lambda(u)$ is convex for all $u\in L^r(\Omega)$. Theorem~\ref{thm:wp-dim-one-neumann} then concludes the proof.
\end{proof}

\subsection{Examples}
\label{sec:examples}

Let us discuss some examples of drag operators covered by Corollaries~\ref{cor:scal-drag-coeff} and~\ref{cor:noncvx-drag}.

\subsubsection{Continuous case}
\label{sec:continuous-case}

Here, for any $\bu\in\bL^r(\Omega)$, we write $\bLambda(\bu) = \{\blambda(\bu)\}$ our monovalued operator. Corollary~\ref{cor:scal-drag-coeff} covers any case of the form
\begin{equation}
  \label{eq:drag-cont-tensor-sum}
  \blambda(\bu) = \left(\sum_{i=0}^m \lambda_i \norm{\bu}_{\mathbb{D}}^i\right) \mathbb{D}\bu,
\end{equation}
where $\mathbb{D}\:\Omega \to \bs{M}_{\mt{sym}}^+$ satisfies $\sup\mathbb{D}<\infty$ and $\inf\mathbb{D}>0$, and $\lambda_i\geq0$ for all $i\in\{0,\dots,m-1\}$ and $\lambda_m>0$, in which case $r=m+2$. When $\mathbb{D} \equiv\mathbb{I}$, this operator simplifies into
\begin{equation*}
  \blambda(\bu) =\left( \sum_{i=0}^m \lambda_i \norm{\bu}^i\right) \bu.
\end{equation*}
Note that the dissipation (cf. \eqref{eq:dissipation-cont}) associated with \eqref{eq:drag-cont-tensor-sum} is the following:
\begin{equation*}
  \D(\bu) = \sum_{i=0}^m \frac{\lambda_i}{i+2}\int_\Omega \norm{\bu}_{\mathbb{D}}^{i+2}.
\end{equation*}
Let us highlight the fact that the drag operator in \eqref{eq:drag-cont-tensor-sum} results from Taylor expanding the function $\phi$ up to order $m$ around $0$ in
\begin{equation*}
  \blambda(\bu) = \phi(\norm{\bu}_{\mathbb{D}})\mathbb{D}\bu,
\end{equation*}
and setting $\lambda_i = \phi^{(i)}(0)/(i!)$ for all $i\in\{0,\dots,m\}$. 

\subsubsection{Jump case}
\label{sec:jump-case}

In this section, we use again the notation introduced in Section~\ref{sec:jump-drag-operators}, though with a slight generalization. Indeed, we define the transition zones and speed regions (cf.~\eqref{eq:trans-speed-regions}) with respect to a weighted Euclidean norm: fix $\mathbb{D}\:\Omega\to\bs{M}_{\mt{sym}}^+$ satisfying $\sup\mathbb{D}<\infty$ and $\inf\mathbb{D}>0$ and, given $\bu\in\bL^r(\Omega)$, redefine the transition zones $\{\Gamma_j(\bu)\}_{j=1}^{n-1}$ and $\{\Omega_j(\bu)\}_{j=1}^n$ and speed regions according to
\begin{equation}
  \label{eq:regions}
  \begin{gathered}
    \Gamma_j(\bu) = \{\bx\in\Omega\st \norm{\bu(\bx)}_{\mathbb{D}(\bx)} = \bar u_j\}, \qquad j\in\{1,\dots,n-1\},\\
    \Omega_j(\bu) = \{\bx\in\Omega\st \norm{\bu(\bx)}_{\mathbb{D}(\bx)} \in (\bar u_{j-1},\bar u_j)\}, \qquad j\in\{1,\dots,n\},
  \end{gathered}
\end{equation}
with the additional convention that $\Gamma_n(\bu)=\emptyset$. In view of this, consider the example when $\bLambda$ is of the form
\begin{equation}
  \label{eq:jump-complete}
  \bLambda(\bu) = \left(\sum_{j=1}^nS_j(\norm{\bu}_{\mathbb{D}})\sum_{i=0}^{m_j} \lambda_{ij} \norm{\bu}_{\mathbb{D}}^i\right) \mathbb{D}\bu,
\end{equation}
where $\{m_j\}_{j=1}^n\subset\N$, $\{S_j\}_{j=1}^n$ is as in \eqref{eq:Sj-1}-\eqref{eq:Sj-2}, and, for all $j\in\{1,\dots,n\}$ and $i\in\{0,\dots,m_j\}$, we have $\lambda_{ij}\geq 0$ and $\lambda_{m_jj}>0$. In this case, $r=m_n+2$.

We distinguish two cases: that when the jump through each transition zone is nondecreasing, and that when this does not hold.

\paragraph{Nondecreasing jump}

Corollary~\ref{cor:scal-drag-coeff} covers any drag operator of the form \eqref{eq:jump-complete} provided that the jumps through the transition zones are nondecreasing:
\begin{equation}
  \label{eq:jump-incr}
  \sum_{i=0}^{m_j} \lambda_{ij}\bar u_j^i \leq \sum_{i=0}^{m_{j+1}} \lambda_{i,j+1}\bar u_j^i \quad \text{for all $j\in\{1,\dots,n-1\}$}.
\end{equation}
Indeed, in this case, the associated dissipation is convex, namely,
\begin{equation*}
  \D(\bu) = \sum_{j=1}^n \int_{\Omega_j\cup\Gamma_j(\bu)} \left(c_j + \sum_{i=0}^{m_j} \frac{\lambda_{ij}}{i+2}\norm{\bu}_{\mathbb{D}}^{i+2}\right),
\end{equation*}
where, for all $j\in\{1,\dots,n\}$, the real scalar $c_j$ is an integration constant ensuring the local Lipschitz continuity of the integrand of $\D$ across the transition zones:
\begin{equation}
  \label{eq:cj}
  c_1=0 \quad \text{and} \quad c_{j+1} = c_j + \sum_{i=0}^m \frac{\lambda_{ij}-\lambda_{i,j+1}}{i+2}\, \bar u_j^{i+2} \quad \text{for all $j\in\{1,\dots,n-1\}$},
\end{equation}
where $m=\max\{m_j,m_{j+1}\}$, and $\lambda_{mj}=0$ if $m_j<m$ and $\lambda_{m,j+1}=0$ if $m_{j+1}<m$.

This case includes, for example, the ``double Darcy'' operator
\begin{equation}
  \label{eq:DD}
  \bLambda_{\mt{D/D}}(\bu) = \begin{cases} \lambda_{01}\bu & \text{if $\norm{\bu} < \bar u$},\\ [\lambda_{01},\lambda_{02}]\,\bu & \text{if $\norm{\bu} = \bar u$},\\ \lambda_{02}\bu & \text{if $\norm{\bu}> \bar u$}, \end{cases} 
\end{equation}
where $\bar u=\bar u_1$ and $0<\lambda_{01}\leq \lambda_{02}$, and the triple-regime operator
\begin{equation}
  \label{eq:DDFF}
  \bLambda_{\mt{D/DF/F}}(\bu) = \begin{cases} \lambda_{01}\bu & \text{if $\norm{\bu} < \bar u_1$},\\ [\lambda_{01},\lambda_{02}+\lambda_{12}\bar u_1]\, \bu & \text{if $\norm{\bu} = \bar u_1$},\\ \lambda_{02}\bu + \lambda_{12}\norm{\bu}\bu & \text{if $\bar u_1 < \norm{\bu} < \bar u_2$},\\ [\lambda_{02}+\lambda_{12}\bar u_2,\lambda_{23}\bar u_2^2]\,\bu & \text{if $\norm{\bu} = \bar u_2$},\\ \lambda_{23}\norm{\bu}^2\bu & \text{if $\norm{\bu} > \bar u_2$}, \end{cases} 
\end{equation}
where $0<\lambda_{01}\leq \lambda_{02}+\lambda_{12}\bar u_1$ and $\lambda_{02}+\lambda_{12}\bar u_2\leq \lambda_{23}\bar u_2^2$.

\paragraph{Increasing jump}

When $\bLambda$ is of the form \eqref{eq:jump-complete} but the nondecreasing condition \eqref{eq:jump-incr} does not hold, Corollary~\ref{cor:scal-drag-coeff} does not apply. If $d=1$, then Corollary~\ref{cor:noncvx-drag} gives us at least existence; if $d>1$, then the question remains open and, as already mentioned, we leave this case for future research. Thus, when $d=1$, the operators given in \eqref{eq:DD} and \eqref{eq:DDFF} yield existence of solutions even when the ordering restrictions on the coefficients $\lambda_{01}$, $\lambda_{02}$, $\lambda_{12}$, $\lambda_{23}$ are not satisfied.

\subsubsection{Sum case}
\label{sec:sum-case}

By~\cite[Theorem~1]{Rockafella70_2}, the finite sum of maximal monotone operators with domain $\bL^r(\Omega)$ stays maximal monotone. Thus, any finite sum of drag operators discussed in Sections~\ref{sec:continuous-case} and~\ref{sec:jump-case} is still covered by Corollary~\ref{cor:scal-drag-coeff}.

For instance, Corollary~\ref{cor:scal-drag-coeff} includes the following continuous form:
\begin{equation}
  \label{eq:drag-cont}
  \blambda(\bu) = \lambda_0\mathbb{D}_0\bu + \lambda_1\norm{\bu}_{\mathbb{D}_1}^\gamma \mathbb{D}_1\bu,
\end{equation}
where $\gamma>0$, and $\mathbb{D}_0,\mathbb{D}_1\:\Omega \to \bs{M}_{\mt{sym}}^+$ are possibly different with $\sup\mathbb{D}_0,\sup\mathbb{D}_1<\infty$ and $\inf \mathbb{D}_0 + \inf\mathbb{D}_1 >0$, and $\lambda_0,\lambda_1\geq0$ and $\lambda_0+\lambda_1>0$. Taking $\lambda_0=1$ and $\lambda_1=0$, the operator in \eqref{eq:drag-cont} becomes Darcy's law ($r=2$); taking $\lambda_0=0$, $\lambda_1=1$ and $\mathbb{D}_1 = \lambda \mathbb{I}$ with $\lambda\:\Omega\to(0,\infty)$ and $\inf\lambda>0$, it becomes Forchheimer's generalized law ($r=\gamma+2$); taking $\gamma=1$, $\lambda_0= \lambda_1=1$ and $\mathbb{D}_1 = \lambda \mathbb{I}$ with $\lambda$ as just discussed, it becomes Darcy--Forchheimer law ($r=3$). We also cover the continuous, sublinear operator in \eqref{eq:law-sublinear}:
\begin{equation*}
  \blambda(\bu) = \frac{\lambda_1\bu}{1+\lambda_2\norm{\bu}} + \norm{\bu}_{\mathbb{D}}^\gamma \mathbb{D}\bu,
\end{equation*}
where $\gamma\geq0$, $\lambda_1,\lambda_2\:\Omega\to(0,\infty)$ and $\mathbb{D}\:\Omega \to \bs{M}_{\mt{sym}}^+$ verifies $\sup\mathbb{D}<\infty$ and $\inf \mathbb{D} >0$; here, $r=\gamma+2$.

By adding a continuous Darcy operator to a jump operator, we see that Corollary~\ref{cor:scal-drag-coeff} further covers the jump operator given in \eqref{eq:jump-DDF}, which we reformulate now in our multivalued setting and in generalized weighted norm:
\begin{equation*}
  \bLambda_{\mt{D/DF}}(\bu) = \mathbb{D}_0\bu + \begin{cases} \bnull & \text{if $\norm{\bu}_{\mathbb{D}_1} < \bar u$},\\ [0,\lambda\bar u]\,\mathbb{D}_1\bu & \text{if $\norm{\bu}_{\mathbb{D}_1} = \bar u$},\\ \lambda\norm{\bu}_{\mathbb{D}_1}\mathbb{D}_1\bu & \text{if $\norm{\bu}_{\mathbb{D}_1} > \bar u$}, \end{cases} 
\end{equation*}
where $\bar u = \bar u_1$, $\lambda>0$, and $\mathbb{D}_0$ and $\mathbb{D}_1$ are as above; here, $r=3$ and the term $\mathbb{D}_0\bu$ plays the role of a background drag force.

\section{Regularized problem for dissipative operators}
\label{sec:regul-probl-diss}

We would like to use classical numerical schemes to solve Problem~\ref{pb:weak-form-multiv}, such as Picard iterations combined with the Raviart--Thomas or the mixed virtual element methods (cf. Section~\ref{sec:num-appr}). To this end, we propose first to appoximate Problem~\ref{pb:weak-form-multiv} by a monovalued problem obtained from a convolutional regularization of the dissipation. Indeed, we restrict here to $\bLambda\:\bL^r(\Omega) \rightrightarrows\bL^s(\Omega)$, $r\geq2$, being a dissipative operator, i.e., an operator satisfying Assumption~\ref{assu:diss-drag}. Also, as done in Sections~\ref{sec:math-fram} and~\ref{sec:well-posedn-diss}, we fix $q\in L^r(\Omega)$, $\bff\in \bL^s(\Omega)$, $u_0\in L^{r}(\Sigma_{\mt{v}})$ and $p_0\in W^{\frac{1}{r},s}(\Sigma_{\mt{p}})$.

Furthermore, we let $\{\Reg_\e\}_{\e>0}$ be a mollifying sequence; see Appendix~\ref{sec:conv-moll} for the classical concepts used in this section on convolutions and mollifiers.

\subsection{Regularization}
\label{sec:regularization}

Write $\Psi_0\:\R\to\R$ the following continuous extension of $\Psi$:
\begin{equation*}
  \Psi_0(a) =
  \begin{cases}
    \Psi(0) & \text{for all $a<0$},\\
    \Psi(a) & \text{for all $a\geq0$};
  \end{cases}
\end{equation*}
then, we may define the \emph{regularization} $\{\Psi_\e\}_{\e>0}$ of $\Psi$ by mollification of $\Psi_0$ according to $\Psi_\e\:\R_+\to\R$ and
\begin{equation}
  \label{eq:regularization}
  \Psi_\e(a) = \Reg_\e*\Psi_0(a) \quad \text{for all $a\geq0$ and $\e>0$}.
\end{equation}

\subsubsection{Regularized dissipation and drag operator}
\label{sec:regul-diss}

For all $\e>0$, we define the \emph{regularized} dissipation $\D_\e\:\bL^r(\Omega) \to \R$ by
\begin{equation}
  \label{eq:diss-reg}
  \D_\e(\bu) = \frac12 \int_\Omega \Psi_\e(\norm{\bu}_{\mathbb{D}_i}^2) \quad \text{for all $\bu\in\bL^r(\Omega)$},
\end{equation}
as well as the \emph{regularized} drag operator $\blambda_\e\:\bL^r(\Omega)\to\bL^s(\Omega)$ by
\begin{equation}
  \label{eq:reg-drag}
  \blambda_\e(\bu) = \phi_\e(\norm{\bu}_{\mathbb{D}}^2)\,\mathbb{D}\bu \quad \text{for all $\bu\in\bL^r(\Omega)$},
\end{equation}
where $\phi_\e:=\Psi_\e'=\Reg_\e'*\Psi_0$. We have the following analogues of Lemmas~\ref{lem:psi-diss} and~\ref{lem:drag-subdiff}:
\begin{lem}
  \label{lem:psi-diss-reg}
  Let $\e>0$. Then, $\Psi_\e$ is nondecreasing, $(r/2)$-bounded and $(r/2)$-coercive. Additionally, the regularized dissipation $\D_\e$, as given in \eqref{eq:diss-reg}, is well defined and there are $c_{0,\e}, c_{1,\e}, C_\e>0$ so that
  \begin{equation}
    \label{eq:bounds-D-reg}
    c_{1,\e}\norm{\bu}_r^r - c_{0,\e} \leq \D_\e(\bu) \leq C_\e(1+\norm{\bu}_r^r) \quad \text{for all $\bu\in\bL^r(\Omega)$}.
  \end{equation}
The families $\{c_{0,\e}\}_{\e>0}$, $\{c_{1,\e}\}_{\e>0}$ and $\{C_\e\}_{\e>0}$ are bounded. 
\end{lem}
\begin{proof}
  By Lemma~\ref{lem:psi-diss}, we get that $\Psi_0$ is nondecreasing and so, by the nonnegativity of $\Reg_\e$, the nondecreasing monotonicity of $\Psi_\e$ follows. Then, for all $a\geq0$, note that
  \begin{equation*}
    \Psi_\e(a) = \int_{-\infty}^\infty \Reg_\e(a-b)\Psi_0(b) \d b = M\Psi(0) + \int_0^\infty \Reg_\e(a-b) \Psi(b) \d b,
  \end{equation*}
where $M:=\int_{-\infty}^0 \Reg_\e$ is in fact independent of $\e$. Using that $\Psi$ is $(r/2)$-bounded from Lemma~\ref{lem:psi-diss} and writing $K$ its boundedness constant, for all $a\geq0$, we get
\begin{align*}
  \abs{\Psi_\e(a)}  &\leq M\abs{\Psi(0)} + K\int_0^\infty \Reg_\e (a-b) \left(1+b^{\frac{r}{2}}\right) \d b\\
                    &\leq M\abs{\Psi(0)} + K\left(1 + \int_{-\infty}^a \Reg_\e (b) (a+\abs{b})^{\frac{r}{2}} \d b\right)\\
  &\leq M\abs{\Psi(0)} + K\left(1 + 2^{1-\frac{r}{2}} \int_{-\infty}^\infty \Reg_\e (b) \left( a^{\frac{r}{2}} + \abs{b}^{\frac{r}{2}} \right) \d b \right)\\
  &\leq M\abs{\Psi(0)} + K\left( 1 + 2^{1-\frac{r}{2}} M_\e + 2^{1-\frac{r}{2}} a^{\frac{r}{2}} \right)\\
  &\leq \left(M\abs{\Psi(0)} + K\left(1 + 2^{1-\frac{r}{2}}M_\e\right)\right) \left(1+a^{\frac{r}{2}}\right),
\end{align*}
where $M_\e:=\int_{-\infty}^\infty  \Reg_\e (b)\abs{b}^{\frac{r}{2}}$ is such that $M_\e\to0$ as $\e\to 0^+$. Thus, $\Psi$ is $(r/2)$-bounded. Furthermore, we also know from Lemma~\ref{lem:psi-diss} that $\Psi$ is $(r/2)$-coercive; writing $k$ its coercivity constant, for all $a\geq0$, we yield
\begin{align*}
  \Psi_\e(a) &\geq M\Psi(0) + k\int_0^\infty \Reg_\e(a-b) b^{\frac{r}{2}} \d b\\
  &\geq M\Psi(0) + k \left( \int_0^a \Reg_\e(a-b) b^{\frac{r}{2}} \d b + \left( \int_a^\infty \Reg_\e(a-b) \d b \right)a^{\frac{r}{2}}  \right)\\
  &\geq M\Psi(0) + k M a^{\frac{r}{2}},
\end{align*}
which shows $\Psi$ is $(r/2)$-coercive.

The above inequalities, as well as the nondecreasing monotoncity of $\Psi_\e$ directly give the fact that $\D_\e$ is well defined and that \eqref{eq:bounds-D-reg} holds.
\end{proof}
\begin{lem}
  \label{lem:drag-grad}
  Let $\e>0$. Then, $\D_\e$ is locally Lipschitz continuous and differentiable with $\blambda_\e=\bgrad\D_\e$. If $\Psi$ is convex, then $\Psi_\e$ and $\D_\e$ are convex; if furthermore $\Psi$ is strictly increasing, then $\Psi_\e$ is strictly increasing and $\D_\e$ is strictly convex.
\end{lem}
\begin{proof}
  Define $\eta_\e\:\R^d\to \R$ by
  \begin{equation*}
    \eta_\e(\ba) = \frac12 \Psi_\e(\norm{\ba}_{\mathbb{D}}^2) \quad \text{for all $\ba\in\R^d$},
  \end{equation*}
  so that
  \begin{equation*}
    \D_\e(\bu) = \int_\Omega \eta_\e(\bu) \quad \text{for all $\bu\in\bL^r(\Omega)$}.
  \end{equation*}

  The local Lipschitz continuity of $\D_\e$ stems from~\cite[Proposition~12]{GP17} and the right-hand inequality in \eqref{eq:bounds-D-reg}. Furthermore, let $\bu\in\bL^r(\Omega)$ and note that
  \begin{equation*}
    \blambda_\e(\bu)(\bx) = \bgrad\eta_\e(\bu(\bx)) \quad \text{for all $\bx\in \Omega$};
  \end{equation*}
  by~\cite[Section~3]{Giner98}, we get from this that $\bp \D_\e(\bu) \subset \{\blambda_\e(\bu)\}$. Since we know $\bp\D_\e(\bu)$ is nonempty (cf. Proposition~\ref{prop:clarke}), we get that $\bp \D_\e(\bu)$ is the singleton $\{\blambda_\e(\bu)\}$, so that $\D_\e$ is differentiable with $\bgrad \D_\e = \blambda_\e$.

  Suppose that $\Psi$ is convex. Then, $\Psi_0$ is convex since $\Psi$ is convex and nondecreasing (cf. Lemma~\ref{lem:psi-diss}). Thus, by Proposition~\ref{prop:cv-moll}, $\Psi_\e$ is convex. Since $\Psi_\e$ is also nondecreasing (cf. Lemma~\ref{lem:psi-diss-reg}), we get that $\eta_\e$ is convex and the convexity of $\D_\e$ follows; the analogous argument holds to get the strict monotonicity of $\Psi_\e$ and the strict convexity of $\D_\e$ in case $\Psi$ is strictly increasing.
\end{proof}

\subsubsection{Regularized problem and results}
\label{sec:regul-probl-results}

Here follows the resulting regularized, monovalued problem:
\begin{namedthm}{Problem~5.3($\e$)}[regularized]
  \namedlabel{pb:reg}{\thesection.3}
  Find $(\bu_\e,p_\e) \in \bL^r(\Omega)\times W_0^{1,s}(\Omega)$ so that
  \begin{equation*}
    \begin{gathered}
      \ap{\blambda_\e(\bu_\e),\bvarphi} = \ap{-\bgrad p_\e + \bff_0,\bvarphi} \qquad \forall \bvarphi\in \bL^r(\Omega),\\
      \ap{\bgrad \psi, \bu_\e} = -\int_\Omega q\psi + \int_{\Sigma_{\mt{v}}} u_0\psi
      \qquad \forall \psi\in W_0^{1,s}(\Omega).
    \end{gathered}
  \end{equation*}
\end{namedthm}

The question is now to determine whether Problem~\ref{pb:reg} admits a solution for each $\e>0$ and, if so, whether a sequence $(\bu_\e,p_\e)_{\e>0}$ of solutions to Problem~\ref{pb:reg} converges weakly to a solution to Problem~\ref{pb:weak-form-multiv}. We state our results in this regard below and then provide the proofs.

\begin{cor}[well-posedness---regularized problem]
\label{cor:wp-reg}
  Let $\e>0$.
  \begin{enumerate}[label=(\roman*)]
    \item\label{it:cvx} If $\Psi$ is convex, then Problem~\ref{pb:reg}($\e$) has a solution. If additionally $\Psi$ is strictly increasing, then Problem~\ref{pb:reg}($\e$) has a unique solution.
    \item\label{it:non-cvx} If $d=1$, then Problem~\ref{pb:reg}($\e$) has a solution.
  \end{enumerate}
\end{cor}

\begin{thm}[convergence---regularized problem]
\label{thm:cv-reg}
  Let $\Psi$ be convex and strictly increasing, and let $(\bu_\e,p_\e)_{\e>0}$ be a sequence in $\bL^r(\Omega)\times W_0^{1,s}(\Omega)$ such that $(\bu_\e,p_\e)$ is solution to Problem~\ref{pb:reg}($\e$) for all $\e>0$. Then, there exists $(\bu,p) \in \bL^r(\Omega)\times W_0^{1,s}(\Omega)$ such that, up to subsequences, $\bu_\e \wto \bu$ in $\bL^r(\Omega)$ and $p_\e \wto p$ in $W_0^{1,s}(\Omega)$ as $\e \to 0^+$, and $(\bu,p)$ is solution to Problem~\ref{pb:weak-form-multiv}.
\end{thm}

Note that, even when $d=1$ and thus Problem~\ref{pb:reg}($\e$) has a solution for all $\e>0$ regardless of convexity, the above convergence result does not apply in the nonconvex case. As is clear in the proof of Corollary~\ref{cor:wp-reg} below, this stems from the facts that, in this case, we are not able to show that saddle points and solutions coincide or that there is $\Gamma$-convergence of the dissipation. 

Also note that the condition $\Psi$ convex of Corollary~\ref{cor:wp-reg}\ref{it:cvx} and Theorem~\ref{thm:cv-reg} is slightly more restrictive than the condition $\Psi\circ(\cdot)^2$ convex of Corollary~\ref{cor:scal-drag-coeff}. In fact, Corollary~\ref{cor:wp-reg}\ref{it:cvx} does not hold if $\Psi$ corresponds to the sublinear operator in \eqref{eq:law-sublinear}; indeed, in this case, $\Psi\circ(\cdot)^2$ is convex but $\Psi$ is not since $\Psi(a) = 2\sqrt{a} - \ln(1+\sqrt{a})$ for all $a\geq0$. Circumventing this issue should not be difficult by mollifying $\Psi_0\circ(\cdot)^2$ instead of only $\Psi_0$ in the definition of the regularization in \eqref{eq:regularization}, although we do not explore this possibility here.

\subsection{Proof of Corollary~\ref{cor:wp-reg}}
\label{sec:proof-wp-reg}

Let us fix $\e>0$ throughout this section. We want to prove that the regularized drag operator $\blambda_\e$ in \eqref{eq:reg-drag} satisfies Assumption~\ref{assu:diss-drag} in place of $\bLambda$ (i.e., that $\blambda_\e$ is indeed dissipative), and then that we can apply Corollaries~\ref{cor:scal-drag-coeff} and~\ref{cor:noncvx-drag} to show Items~\ref{it:cvx} and~\ref{it:non-cvx}, respectively, in Corollary~\ref{cor:wp-reg}.

\subsubsection{Dissipativity of the regularization}
\label{sec:diss-regul}

We first show that $\blambda_\e$ in \eqref{eq:reg-drag} satisfies Assumption~\ref{assu:diss-drag} in place of $\bLambda$. We only need to prove that $\phi_\e$ satisfies \eqref{eq:cont-upper-bound} instead of $\phi$; indeed, $\phi_\e = \Psi_\e'$ by definition, the nonnegativity of $\phi_\e$ is direct from the nondecreasing monotonicity of $\Psi_\e$ by Lemma~\ref{lem:psi-diss-reg}, and the continuity and Lipschitz continuity of $\phi_\e$ and $\Psi_\e$, respectively, is obvious by smoothness of mollification.

To show \eqref{eq:cont-upper-bound} for $\phi_\e$, recall first that, by Rademacher's theorem, $\Psi_0$ is differentiable almost everywhere since it is locally Lipschitz continuous. Therefore, integrating by parts, for all $a\geq0$, we get 
\begin{equation*}
    \phi_\e(a) = \int_{-\infty}^\infty \gamma_\e'(a-b) \Psi_0(b) \d b = \int_{-\infty}^\infty \gamma_\e(a-b) (\Psi_0)'(b) \d b = \int_0^\infty \gamma_\e(a-b) \phi(b) \d b,
\end{equation*}
where $\phi\:\R_+\to\R_+$ is such that $\phi(b)\in\Phi(b)$ for all $b\geq0$. Using now \eqref{eq:cont-upper-bound} for $\phi$ and following the same steps as in the calculations in the proof of Lemma~\ref{lem:psi-diss-reg}, we get that \eqref{eq:cont-upper-bound} holds also for $\phi_\e$.
 
\subsubsection{Convex case}
\label{sec:convex-case}

Thanks to Section~\ref{sec:diss-regul}, Item~\ref{it:cvx} of Corollary~\ref{cor:wp-reg} is now a straightforward application of Corollary~\ref{cor:scal-drag-coeff}, since indeed, by Lemma~\ref{lem:psi-diss-reg}, the convexity and strictly increasing monotonicity of $\Psi_\e$ follow from the convexity and strictly increasing monotonicity of $\Psi$, respectively.

\subsubsection{Nonconvex case}
\label{sec:nonconvex-case}

Item~\ref{it:non-cvx} is direct because $\blambda_\e$ satisfies Assumption~\ref{assu:diss-drag}, as shown in Section~\ref{sec:diss-regul}; the result is then a trivial application of Corollary~\ref{cor:noncvx-drag}.

\subsection{Proof of Theorem~\ref{thm:cv-reg}}
\label{sec:proof-cv-reg}

Assume that $\Psi$ is convex and strictly increasing and define $\B\:W_0^{1,s}(\Omega)\to \R$ as
\begin{equation*}
  \B(p) = - \int_\Omega q p + \int_{\Sigma_{\mt{v}}} u_0 p \quad \text{for all $p\in W_0^{1,s}(\Omega)$}.
\end{equation*}
Also, let $\hat\bu$ be as in Lemma~\ref{lem:bdry-V0} and recall the definition of $V$ in \eqref{eq:V}. Then, introduce $\D^V\:V\to\R$, $\E\:\bL^r(\Omega)\times W_0^{1,s}(\Omega)\to \R$ and $\D^*\:W_0^{1,s}(\Omega)\to\R$ defined, for all $\bv \in V$, $\bu\in\bL^r(\Omega)$ and $p\in W_0^{1,s}(\Omega)$, by
\begin{gather*}
  \D^V(\bv) = \D(\hat\bu + \bv) - \ap{\bff_0,\hat\bu + \bv},\\
  \E(\bu,p) = \B(p) - \ap{\bgrad p,\bu} - \D(\bu) + \ap{\bff_0,\bu},\\
  \D^*(p) = \max_{\bvarphi\in\bL^r(\Omega)} \E(\bvarphi,p).
\end{gather*}
Similarly, we define the regularized counterparts of these functionals for all $\e>0$, namely, $\D_\e^V\:V\to\R$, $\E_\e\:\bL^r(\Omega)\times W_0^{1,s}(\Omega)\to \R$ and $\D_\e^*\:W_0^{1,s}(\Omega)\to\R$, for all $\bv_\e \in V$, $\bu_\e\in\bL^r(\Omega)$ and $p_\e\in W_0^{1,s}(\Omega)$, by
\begin{gather*}
  \D_\e^V(\bv_\e) = \D_\e(\hat\bu + \bv_\e) - \ap{\bff_0,\hat\bu + \bv_\e},\\
  \E_\e(\bu_\e,p_\e) = \B(p_\e) - \ap{\bgrad p_\e,\bu_\e} - \D_\e(\bu_\e) + \ap{\bff_0,\bu_\e},\\
  \D_\e^*(p_\e) = \max_{\bvarphi\in\bL^r(\Omega)} \E_\e(\bvarphi,p_\e).
\end{gather*}
Note that $\D^*$ and $\D_\e^*$ are indeed well defined since $\D$ and $\D_\e$ are strictly convex by Lemmas~\ref{lem:drag-subdiff} and~\ref{lem:drag-grad}, as well as coercive by Lemmas~\ref{lem:psi-diss} and~\ref{lem:psi-diss-reg}.

\subsubsection{Variational characterization of solutions}
\label{sec:vari-char-solut}

In the lemma and remark below, we provide characterizations of the solutions to Problems~\ref{pb:weak-form-multiv} and~\ref{pb:reg}($\e$) for all $\e>0$ in terms of saddle points and minimizers.

\begin{lem}
  \label{lem:prel-saddle}
  Let $(\bu,p)\in\bL^r(\Omega)\times W_0^{1,s}(\Omega)$. The following assertions are equivalent:
  \begin{enumerate}[label=(\roman*)]
    \item\label{it:sol} $(\bu,p)$ is solution to Problem~\ref{pb:weak-form-multiv}.
    \item\label{it:saddle-p} $(\bu,p)$ is a saddle point of $\E$.
    \item\label{it:min} $p\in\argmin\D^*$ and there exists $\bv\in V$ so that $\bu = \hat\bu + \bv$ and $\bv = \argmin \D^V$.
  \end{enumerate}
\end{lem}

\begin{proof}
  Let us prove this result via the (nonminimal) set of implications below.
  
\paragraph{\ref{it:sol}$\implies$\ref{it:saddle-p}}
Suppose that $(\bu,p)$ is solution to Problem~\ref{pb:weak-form-multiv}. Using the second equation in Problem~\ref{pb:weak-form-multiv}, we get
  \begin{align*}
    \E(\bu,\psi) &= \B(\psi) - \ap{\bgrad \psi,\bu} - \D(\bu) + \ap{\bff_0,\bu}\\
    &= \ap{\bff_0,\bu} - \D(\bu) = \E(\bu,p) \quad \text{for all $\psi\in W_0^{1,s}(\Omega)$},
  \end{align*}
and, from the first equation in Problem~\ref{pb:weak-form-multiv}, we know there is $\blambda\in\bLambda(\bu)$ such that
\begin{align*}
    \E(\bvarphi,p) &= \B(p) - \ap{\bgrad p,\bvarphi} - \D(\bvarphi) + \ap{\bff_0,\bvarphi}\\
    &= \B(p) + \ap{\blambda,\bvarphi} - \D(\bvarphi)\leq \E(\bu,p) \quad \text{for all $\bvarphi\in \bL^r(\Omega)$},
  \end{align*}
where the last inequality comes from the fact that $\bvarphi\mapsto \D(\bvarphi)-\ap{\blambda,\bvarphi}$ is strictly convex (by Lemma~\ref{lem:drag-subdiff}) with critical point $\bu$, so that, by Proposition~\ref{prop:prop-critical-min}, $\bu$ is the global minimizer of $\bvarphi\mapsto \D(\bvarphi) - \ap{\blambda,\bvarphi}$. Hence
\begin{equation*}
  \E(\bvarphi,p) \leq \E(\bu,p) \leq \E(\bu,\psi) \quad \text{for all $(\bvarphi,\psi)\in \bL^r(\Omega)\times W_0^{1,s}(\Omega)$},
\end{equation*}
so that $(\bu,p)$ is a saddle point of $\E$.

\paragraph{\ref{it:saddle-p}$\implies$\ref{it:sol}}
Let $(\bu,p)$ be a saddle point of $\E$. Then, from~\cite[Theorem~4.8]{Leonard07}, we get that $(\bu,p)$ is also a critical point of $\E$. Since $\bp\D = \bLambda$ (cf. Lemma~\ref{lem:drag-subdiff}), one can then check that $(\bu,p)$ is solution to Problem~\ref{pb:weak-form-multiv}.

\paragraph{\ref{it:sol}$\implies$\ref{it:min}}
Suppose $(\bu,p)$ is solution to Problem~\ref{pb:weak-form-multiv}. The proof of Lemma~\ref{lem:pb-restr} shows that then $\bu = \hat\bu + \bv$, where $\bv$ is solution to Problem~\ref{pb:restricted-V}. Thus, $\bnull\in \bLambda(\hat\bu+\bv) - \{\bff_0\} = \bp\D^V(\bv)$ (cf. Lemma~\ref{lem:drag-subdiff}), so that $\bv$ is a critical point of $\D^V$. Hence $\bv$ is a global minimizer of $\D^V$ since $\D^V$ is strictly convex (cf. Proposition~\ref{prop:prop-critical-min}). The fact that $p\in\argmin\D^*$ is direct by~\cite[Theorem~4.8]{Leonard07} given that $(\bu,p)$ is a saddle point of $\E$ (cf. \ref{it:sol}$\implies$\ref{it:saddle-p}).

\paragraph{\ref{it:min}$\implies$\ref{it:saddle-p}}
Assume that $p\in\argmin\D^*$ and $\bv\in V$ is such that $\bv = \argmin \D^V$, and $\bu = \hat\bu + \bv$. Then, $\bv$ is a critical point of $\D^V$ and so it solves Problem~\ref{pb:restricted-V}. By the proof of Lemma~\ref{lem:pb-restr}, there exists $\tilde p$ such that $(\bu,\tilde p)$ is solution to Problem~\ref{pb:weak-form-multiv}. Since we already know that \ref{it:sol}$\implies$\ref{it:saddle-p}, there holds
\begin{equation*}
  \E(\bvarphi,\tilde p) \leq \E(\bu,\tilde p) \leq \E(\bu,\psi) \quad \text{for all $\bvarphi\in\bL^r(\Omega)$ and $\psi\in W_0^{1,s}(\Omega)$}.
\end{equation*}
Then, recalling that $p\in\argmin\D^*$, we get
\begin{equation*}
  \E(\bu,p) \leq \max_{\bvarphi\in\bL^r(\Omega)} \E(\bvarphi,p) = \D^*(p) \leq \D^*(\tilde p) = \max_{\bvarphi\in\bL^r(\Omega)} \E(\bvarphi,\tilde p) \leq \E(\bu,\tilde p). 
\end{equation*}
Also,
\begin{equation*}
  \E(\bu,\tilde p) \leq \min_{\psi\in W_0^{1,s}(\Omega)} \E(\bu,\psi) \leq \E(\bu,p). 
\end{equation*}
Combining these last three equations, we yield
\begin{equation*}
  \E(\bvarphi,p) \leq \E(\bu, p) \leq \E(\bu,\psi) \quad \text{for all $\bvarphi\in\bL^r(\Omega)$ and $\psi\in W_0^{1,s}(\Omega)$},
\end{equation*}
so that $(\bu,p)$ is a saddle point of $\E$.
\end{proof}

\begin{rem}
  \label{rem:prel-saddle-eps}
  Since the regularized operator satisfies the same assumptions as the unreguarized one (cf. Sections~\ref{sec:diss-regul}-\ref{sec:convex-case}), Lemma~\ref{lem:prel-saddle} holds for all $\e>0$ replacing Problem~\ref{pb:weak-form-multiv}, $\D^V$, $\E$ and $\D^*$ by, respectively, Problem~\ref{pb:reg}($\e$), $\D_\e^V$, $\E_\e$ and $\D_\e^*$.
\end{rem}

\subsubsection{Compactness}
\label{sec:compactness}

Let $(\bu_\e,p_\e)_{\e>0}\subset\bL^r(\Omega)\times W_0^{1,s}(\Omega)$ be such that $(\bu_\e,p_\e)$ is solution to Problem~\ref{pb:reg}($\e$) for all $\e>0$. We want to show there exists $(\bu,p) \in \bL^r(\Omega)\times W_0^{1,s}(\Omega)$ so that, up to subsequences, $\bu_\e \wto \bu$ and $p_\e \wto p$ as $\e \to 0^+$.

For all $\e>0$, by Lemma~\ref{lem:prel-saddle} and Remark~\ref{rem:prel-saddle-eps}, there exists $\bv_\e\in V$ such that $\bu_\e = \hat\bu + \bv_\e$ and $\bv_\e$ is the unique global minimizer of $\D_\e^V$, so that, in particular, $\D_\e^V(\bv_\e)\leq \D_\e^V(\bnull)$; then, Lemma~\ref{lem:psi-diss-reg} gives
\begin{equation*}
  \D_\e^V(\bnull) \geq \D_\e^V(\bv_\e) = \D_\e(\hat\bu + \bv_\e) + \ap{\bff_0,\bu_\e}\geq \left(c_{1,\e}\norm{\bu_\e}_r^{r-1} - c_{0,\e} - \norm{\bff_0}_s \right) \norm{\bu_\e}_r.
\end{equation*}
Using Fatou's lemma and Proposition~\ref{prop:cv-moll}, we yield
\begin{align*}
  \limsup_{\e\to0^+} \D_\e^V(\bnull) &= \limsup_{\e\to0^+} \D_\e(\hat\bu) - \ap{\bff_0,\hat\bu} \leq \frac12 \int_\Omega \limsup_{\e\to0^+} \Psi_\e(\norm{\hat\bu}_{\mathbb{D}}^2)  - \ap{\bff_0,\hat\bu}\\
  &\leq \frac12 \int_\Omega \Psi(\norm{\hat\bu}_{\mathbb{D}}^2)  - \ap{\bff_0,\hat\bu}.
\end{align*}
Thus, using the last two equations, there exists $K>0$ so that, for all $\e>0$ small enough, we have 
\begin{equation}
  \label{eq:u-eps-bdd}
  K \geq \left(c_{1,\e}\norm{\bu_\e}_r^{r-1} - c_{0,\e} - \norm{\bff_0}_s \right) \norm{\bu_\e}_r.
\end{equation}
Since $r\geq2$, and $\{c_{0,\e}\}_{\e>0}$ and $\{c_{1,\e}\}_{\e>0}$ are bounded, this shows that $\norm{\bu_\e}_r$ is bounded as $\e\to0^+$. Therefore, there exist $\bu\in\bL^r(\Omega)$ and a subsequence of $(\norm{\bu_\e}_r)_{\e>0}$, still denoted $(\norm{\bu_\e}_r)_{\e>0}$,  such that $\bu_\e \wto \bu$ as $\e\to0^+$. 

To show compactness for the pressure, note that, for all $\e>0$, Problem~\ref{pb:reg}($\e$) and Hölder's inequality lead to
\begin{equation}
  \label{eq:estimate-phi}
  \ap{-\bgrad p_\e + \bff_0,\bvarphi} = \ap{\blambda_\e(\bu_\e),\bvarphi} \leq \norm{\blambda_\e(\bu_\e)}_s \norm{\bvarphi}_r \quad \text{for all $\bvarphi\in\bL^r(\Omega)$}.
\end{equation}
Following the same steps yielding \eqref{eq:bounded-lambda} and \eqref{eq:lambda-s} on the regularized problem, there exists a bounded sequence $\{K_\e\}_{\e>0}\subset (0,\infty)$, such that, for all $\e>0$,
\begin{equation*}
  \norm{\blambda_\e(\bu_\e)}_s^s \leq K_\e \left( 1 + \norm{\bu_\e}_r^{r-1} \right),
\end{equation*}
and \eqref{eq:estimate-phi} gives
\begin{equation*}
  \norm{\bgrad p_\e}_s \leq K_\e^{\frac1s} \left( 1 + \norm{\bu_\e}_r^{r-1} \right)^{\frac1s} + \norm{\bff_0}_s.
\end{equation*}
Therefore, since, by \eqref{eq:u-eps-bdd}, $\norm{\bu_\e}_r$ is bounded for all $\e>0$ small enough, so is $\norm{\bgrad p_\e}_s$. By the Rellich--Kondrachov theorem, we can thus extract a subsequence of $(p_\e)_{\e>0}$, still denoted $(p_\e)_{\e>0}$, such that $p_\e\wto p$ as $\e\to0^+$ for some $p\in W_0^{1,s}(\Omega)$.

\subsubsection{$\Gamma$-convergence of the regularized dissipation}
\label{sec:gamma-conv-regul}

Let $(\bu_\e)_{\e>0}$ and $\bu$ be as in Section~\ref{sec:compactness}. We wish to show that $\D_\e \to_{\Gamma} \D$ as $\e\to0^+$. 

Using Proposition~\ref{prop:cv-moll}, for all $\e>0$ and $\bvarphi\in\bL^r(\Omega)$, we get
\begin{equation*}
  \D_\e(\bvarphi) = \frac{1}{2} \int_\Omega \Psi_\e(\norm{\bvarphi}_{\mathbb{D}}^2)\geq \frac{1}{2} \int_\Omega \Psi_0(\norm{\bvarphi}_{\mathbb{D}}^2). 
\end{equation*}
Then, since $\Psi_0$ is convex, we know by Tonelli's theorem of functional analysis that $\bvarphi\mapsto \int_\Omega \Psi(\norm{\bvarphi}_{\mathbb{D}}^2)$ is weakly lower semicontinuous in $\bL^p(\Omega)$, we yield
\begin{equation*}
  \liminf_{\e\to0^+} \D_\e(\bvarphi_\e) \geq \frac{1}{2} \liminf_{\e\to0^+}\int_\Omega \Psi_0(\norm{\bvarphi_\e}_{\mathbb{D}}^2) \geq \frac{1}{2} \int_\Omega \Psi_0(\norm{\bvarphi}_{\mathbb{D}}^2) = \D(\bvarphi), 
\end{equation*}
for any $(\bvarphi_\e)_{\e>0}\subset \bL^r(\Omega)$ and $\bvarphi\in\bL^r(\Omega)$ such that $\bvarphi_\e\wto\bvarphi$, which shows the ``liminf'' condition for $(\D_\e)_{\e>0}$ in the definition of $\Gamma$-convergence (cf. Definition~\ref{defn:gamma-cv}).

For the ``limsup'' condition, let $\bvarphi\in V$ and consider the trivial sequence $(\bvarphi_\e)_{\e>0}$ such that $\bvarphi_\e = \bvarphi$ for all $\e>0$. Then, $(\bvarphi_\e)_{\e>0}$ is a recovery sequence of $\bvarphi$. Indeed, Fatou's lemma and Proposition~\ref{prop:cv-moll} give
\begin{equation*}
  \limsup_{\e\to0^+} \D_\e(\bvarphi_\e) = \limsup_{\e\to0^+} \D_\e(\bvarphi) \leq \frac12 \int_\Omega \limsup_{\e\to0^+} \Psi_\e(\norm{\bvarphi}_{\mathbb{D}}^2)\leq  \frac12 \int_\Omega \Psi(\norm{\bvarphi}_{\mathbb{D}}^2) = \D(\bvarphi),
\end{equation*}
which is the ``limsup'' condition for $(\D_\e)_{\e>0}$. All in all, we have $\D_\e\to_\Gamma \D$ as $\e\to0$.

\subsubsection{Convergence of the flux and pressure}
\label{sec:conv-flux-press}

Let $(p_\e)_{\e>0}$, $(\bu_\e)_{\e>0}$, $p$ and $\bu$ be as in Section~\ref{sec:compactness}. We now show that $(\bu,p)$ is solution to Problem~\ref{pb:weak-form-multiv}.

By Theorem~\ref{lem:prel-saddle} and Remark~\ref{rem:prel-saddle-eps}, we know that $(\bu_\e,p_\e)$ is a saddle point of $\E_\e$ for all $\e>0$:
\begin{equation*}
  \E_\e(\bvarphi,p_\e) \leq \E_\e(\bu_\e,p_\e) \leq \E_\e(\bu_\e,\psi) \quad \text{for all $\bvarphi\in\bL^r(\Omega)$ and $\psi\in W_0^{1,s}(\Omega)$}.
\end{equation*}
Therefore, for all $\bvarphi\in\bL^r(\Omega)$, we compute
\begin{align*}
  \liminf_{\e\to 0^+}\E_\e(\bu_\e,p_\e) &\geq \liminf_{\e\to0^+}\E_\e(\bvarphi,p_\e)\\
  &\geq \liminf_{\e\to0^+} \left( \B(p_\e) - \ap{\bgrad p_\e,\bvarphi} - \D_\e(\bvarphi) + \ap{\bff_0,\bvarphi_\e} \right)\\
  &= \B(p) - \ap{\bgrad p,\bvarphi} + \ap{\bff_0,\bvarphi} - \limsup_{\e\to0^+} \D_\e(\bvarphi)\\
  &\geq \B(p) - \ap{\bgrad p,\bvarphi} + \ap{\bff_0,\bvarphi} - \D(\bvarphi) = \E(\bvarphi,p),
\end{align*}
where the last inequality comes from the ``limsup'' condition in the $\Gamma$-convergence of the regularized dissipation (cf. Section~\ref{sec:gamma-conv-regul}). Similarly, for all $\psi\in W_0^{1,s}(\Omega)$,
\begin{align*}
  \limsup_{\e\to 0^+}\E_\e(\bu_\e,p_\e) &\leq \limsup_{\e\to0^+}\E_\e(\bu_\e,\psi)\\
  &\leq \limsup_{\e\to0^+} \left( \B(\psi) - \ap{\bgrad \psi,\bu_\e} - \D_\e(\bu_\e) + \ap{\bff_0,\bu_\e} \right)\\
  &= \B(\psi) - \ap{\bgrad \psi,\bu} + \ap{\bff_0,\bu} - \liminf_{\e\to0^+} \D_\e(\bu_\e)\\
  &\leq \B(\psi) - \ap{\bgrad \psi,\bu} + \ap{\bff_0,\bu} - \D(\bu) = \E(\bu,\psi),
\end{align*}
where, this time, the last inequality comes from the ``liminf'' condition in the $\Gamma$-convergence of the regularized dissipation. Combining these computations, we get
\begin{equation*}
  \E(\bvarphi,p) \leq \E(\bu,p)\leq \E(\bu,\psi) \quad \text{for all $\bvarphi\in\bL^r(\Omega)$ and $\psi\in W_0^{1,s}(\Omega)$},
\end{equation*}
that is, $(\bu,p)$ is a saddle point of $\E$. By Lemma~\ref{lem:prel-saddle}, this means that $(\bu,p)$ is solution to Problem~\ref{pb:weak-form-multiv}, which ends the proof of Theorem~\ref{thm:cv-reg}.

\begin{rem}[convergence of solutions as minimizers]
  Alternatively, one could show that indeed $(\bu,p)$ is solution to Problem~\ref{pb:weak-form-multiv} by showing that both $(\D_\e^V)_{\e>0}$ and $(\D_\e^*)_{\e>0}$ $\Gamma$-converge to $\D^V$ and $\D^*$ as $\e\to0^+$ along global minimizers, and then by using Proposition~\ref{prop:gamma-cv-min} and Lemma~\ref{lem:prel-saddle}. For $(\D_\e^V)_{\e>0}$, this is direct from Section~\ref{sec:gamma-conv-regul}; for $(\D_\e^*)_{\e>0}$, this is also a consequence of the $\Gamma$-convergence of $(\D_\e)_{\e>0}$ and we leave the details to the reader.
\end{rem}

\section{Numerical approximation}
\label{sec:num-appr}

Let us describe, for all $\e>0$, the numerical approximation of Problem~\eqref{pb:reg}($\e$), which is inherently nonlinear. In fact, even if the law in question is of the the jump type discussed in Section~\ref{sec:jump-case} with $m=0$ (i.e., the law in each speed region is linear), the resulting regularized law $\bLambda_\e$ is nonlinear in $\bu$. Inspired by the standard fixed-point algorithm, we propose the following algorithm to solve Problem~\eqref{pb:reg}($\e$): given a $\bu^0\in\bL^r(\Omega)$, find $\bu^n$ such that
\begin{equation}\label{eq:discr_mollified}
  \begin{cases}
    \dive \bu^n = q,\\
    \sum_{i=0}^m\Phi_{i,\e}(\norm{\bu^{n-1}}_{\mathbb{D}_i}^2)\mathbb{D}_i\bu^n = -\bgrad p^n + \bff,\\
  \end{cases}
  \quad \text{in $\Omega$},
\end{equation}
for all $n\geq1$ such that the following condition is not verified:
\begin{gather*}
    \norm{\bu^{n} - \bu^{n-1}}_r < \tau \norm{\bu^{n-1}}_r,
\end{gather*}
where $\tau>0$ is an arbitrary tolerance. The computation of $\Phi_{i,\e}$ for all $i\in\{0,\dots,m\}$ can be done once at the beginning of the loop and be evaluated at every iteration.

At each iteration in $n$, the problem in \eqref{eq:discr_mollified} is linear and in mixed form, with given inverse permeabilities $\Phi_{i,\e}(\norm{\bu^{n-1}}_{\mathbb{D}_i}^2)$. To numerically discretize it several strategies are possible, we consider here the classical lowest-order Raviart--Thomas approximation~\cite{Raviart1977,Roberts1991} if the computational grid is made of simplices or the lowest-order mixed virtual-element method~\cite{Brezzi2014,BeiraoVeiga2016,Fumagalli2017,Fumagalli2017a,Dassi2020} otherwise. The latter is able to handle cell grids of almost any shape and is suitable for complex problems when $\Omega \subset \mathbb{R}^d$ for $d\geq 2$. Since an accurate description of these numerical schemes is out of the scope of this work, we refer the interested reader to the aforementioned citations for more details.


\section{Numerical results}\label{sec:numerical_results}

In this section, we propose three test cases to validate and show the capabilities of the proposed model and of the variational numerical scheme of Sections~\ref{sec:regul-probl-diss} and~\ref{sec:num-appr}. We focus on drag operators of the jump type discussed in Section~\ref{sec:jump-case} with $\mathbb{F}=\mathbb{I}$, which include in particular the motivating examples discussed in Section~\ref{sec:new-setting}. First, in Section~\ref{subsec:examples:case1}, we compare it against the transition-zone tracking algorithm proposed in~\cite{FP21}. The second example, described in Section~\ref{subsec:example2}, is a problem where three flow regimes may coexist in the domain; we consider both linear and nonlinear laws for each regime. Finally, the last case, reported in Section~\ref{subsec:2dexample}, is a complex two-dimensional problem, where the background permeability field is given by a layer of the SPE10 benchmark.

In all the examples, the mollifying sequence used to regularize the problem is given by the following Gaussian distribution:
\begin{equation*}
  \Reg_\e(a) = \frac{1}{\e \sqrt{2\pi}} \exp\left(-\frac12 \left(\frac{a}{\e}\right)^2\right) \quad \text{for all $a\in\R$}.
\end{equation*}
For the first two cases, the problem in \eqref{eq:discr_mollified} is discretized by the lowest-order Raviart--Thomas method, while the last example is with the lowest-order mixed virtual-element method. All examples were developed with the open source library PorePy~\cite{Keilegavlen2020}; the associated scripts are freely accessible.

\subsection{Comparison with transition-zone tracking}\label{subsec:examples:case1}

In this case, we validate the proposed approach by comparison against the transition-zone tracking algorithm proposed in~\cite{FP21}. Contrary to the present regularized algorithm, which makes the transition zones smooth, the transition-zone tracking algorithm represents the transition zones as sharp interfaces.
For the validation, we retake the problem in~\cite[Section~6.2.1]{FP21}, which considers a linear laws for both regimes, on the one hand, and a linear and nonlinear combination, on the other hand.

\subsubsection{Linear case} \label{subsubsec:examples:case1:linear}

Let the domain be $\Omega = (0, 1)$ and let the scalar (fluid mass) and vector (external body) source terms are set as
\begin{gather}
  \label{eq:qx-f}
    q(x) =
    \begin{dcases*}
        1 & if $x \leq 0.3$,\\
        -1 & if $0.3 < x < 0.7$,\\
        1 & if $x \geq 0.7$,
    \end{dcases*}
    \quad \text{and} \quad
    f(x) = 5\cdot10^{-2}
    .
\end{gather}
The drag operator $\Lambda$ is given by
\begin{gather*}
    \Lambda(u) =
    \begin{dcases*}
        u & in $\Omega_1(u)$,\\
        [0.1,1]\,u & in $\Gamma(u)$,\\
        0.1u & in $\Omega_2(u)$,
    \end{dcases*}
\end{gather*}
where $\Omega_1(u)$, $\Omega_2(u)$ and $\Gamma(u)$ are as in \eqref{eq:regions} with threshold velocity $\bar u = 0.15$.
We set the mollification parameter $\varepsilon = 10^{-4}$ and tolerance $\tau=10^{-6}$ in \eqref{eq:discr_mollified}. The graphs of $\Psi_\e:=\Psi_{0,\e}$ and $\phi_\e:=\phi_{0,\e}$ (here, $m=0$) are given in Figure~\ref{fig:ex1_linear_convolution}.
\begin{figure}[tbp]
    \centering
    \includegraphics[width=1\textwidth]{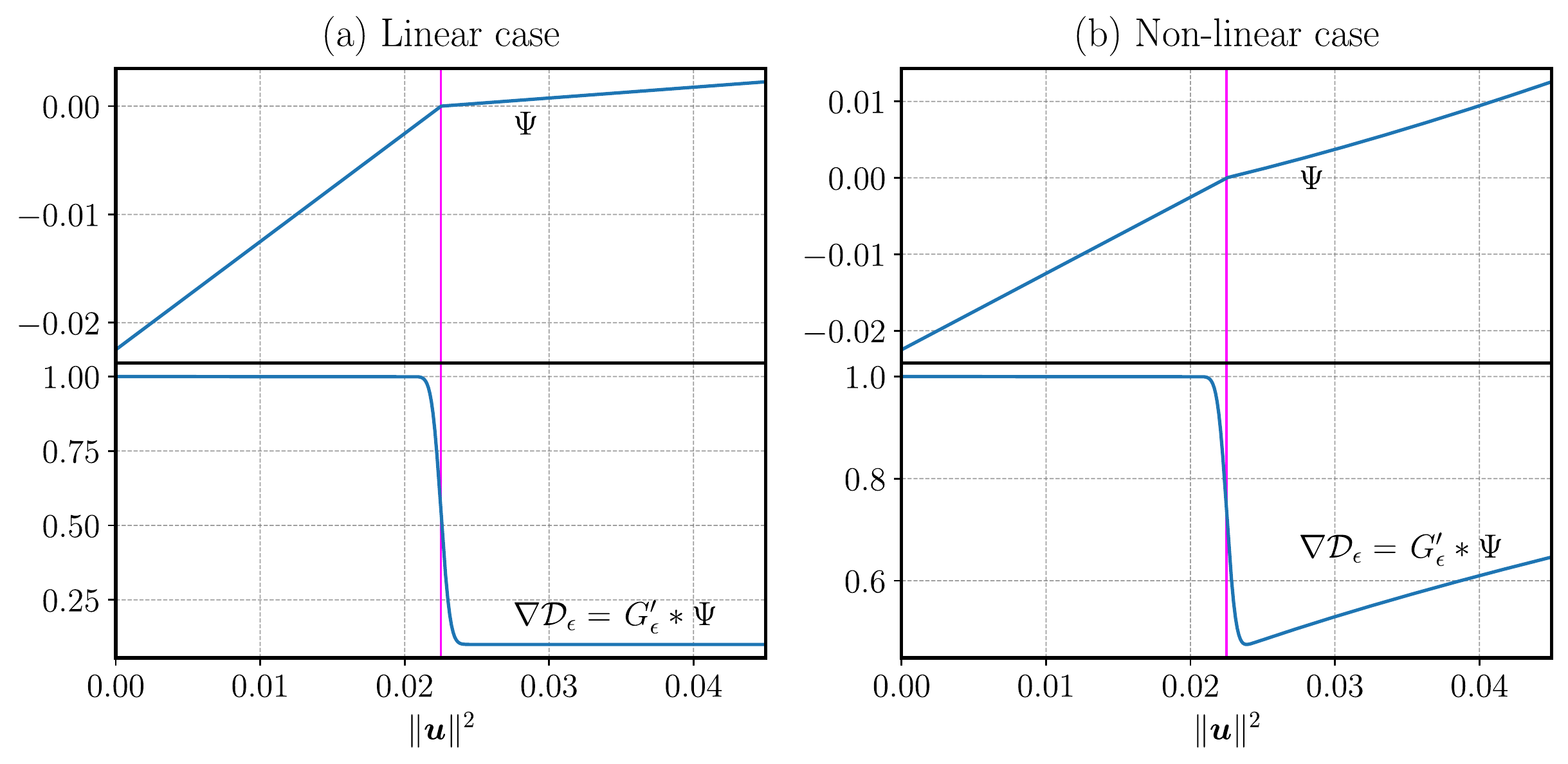}%
    \caption{Regularized-dissipation integrand $\Psi_\e$ (top) and regularized inverse permeability $\phi_\e$ (bottom) as functions of the square velocity for both the linear and nonlinear cases of Section~\ref{subsec:examples:case1}; threshold $\bar u_1^2 = 0.0225$ represented by vertical line}%
    \label{fig:ex1_linear_convolution}
\end{figure}

The algorithm \eqref{eq:discr_mollified} requires two iterations to reach a stable solution, which is reported in Figure~\ref{fig:ex1_linear_solutions}(a). We notice the variation of both pressure and velocity according to the appropriate law chosen by the algorithm.
\begin{figure}[tbp]
    \centering
    \includegraphics[width=1\textwidth]{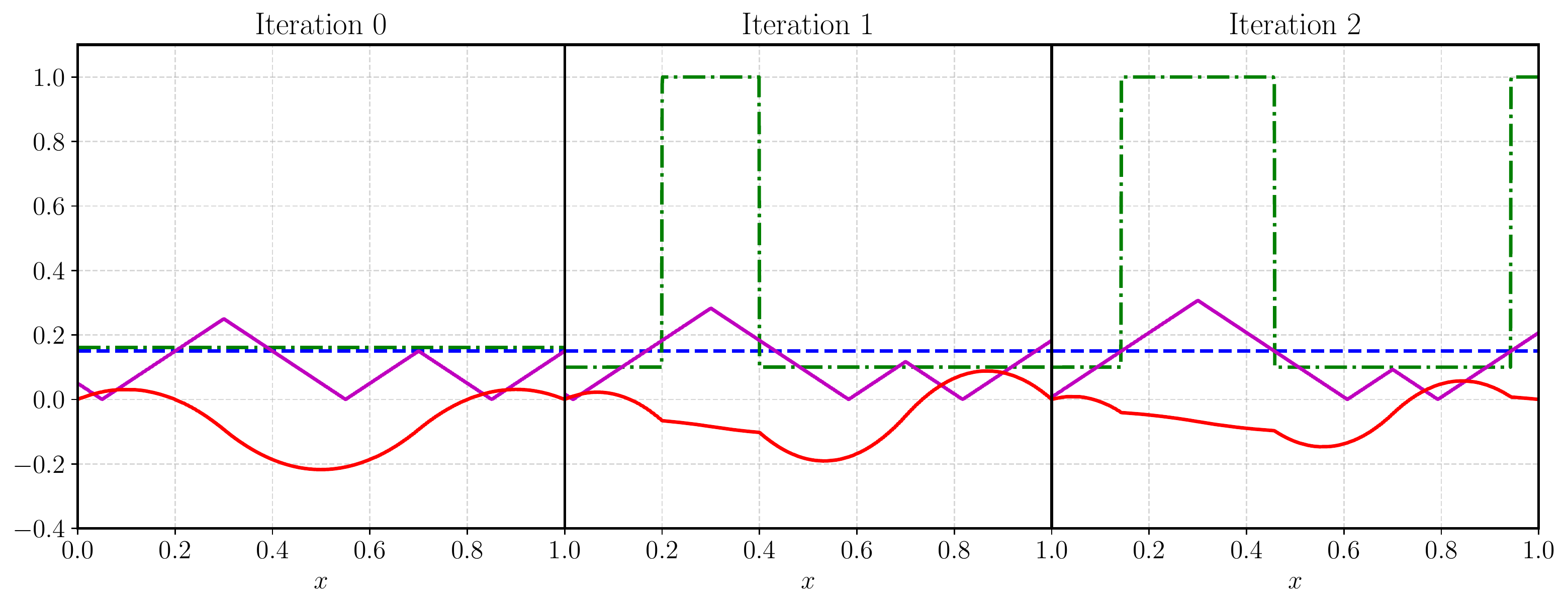}
    \includegraphics[width=0.6\textwidth]{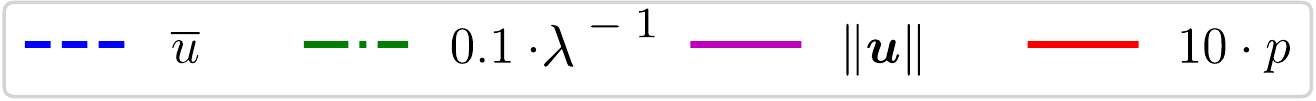}%
    \caption{Solutions for the problem of Section~\ref{subsubsec:examples:case1:linear} at different algorithm iterations}%
    \label{fig:ex1_linear_solutions}
\end{figure}
The errors computed between the proposed algorithm and the one in~\cite{FP21} are presented in Figure~\ref{fig:ex1_linear_error}(a); they are defined as
\begin{gather*}
    err_p = \frac{\norm{p - p_{\mt{ref}}}}{\abs{p_{\mt{ref}}}}
    \quad \text{and} \quad
    err_{\norm{\bu}} = \frac{\norm{u - u_{\mt{ref}}}}{\abs{u_{\mt{ref}}}},
\end{gather*}
where the reference pressure and the reference velocity are the converged solutions of the transition-zone tracking algorithm, since the analytical solution is not known.
We notice that for both pressure and velocity, the magnitude of the error is small.
\begin{figure}[tbp]
    \centering
    \includegraphics[width=0.85\textwidth]{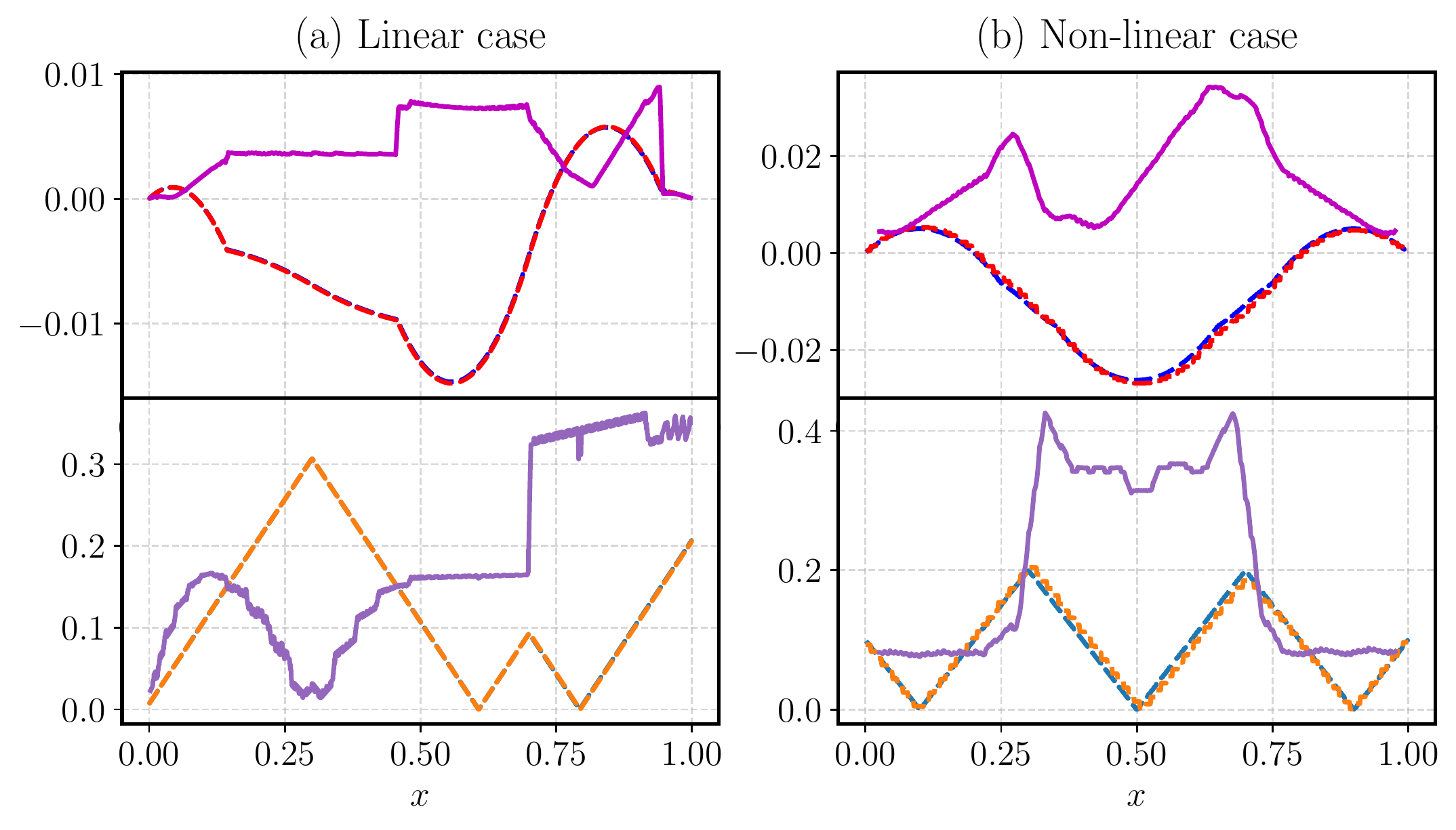}
    \includegraphics[width=0.9\textwidth]{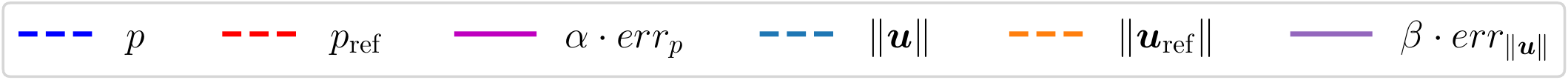}%
    \caption{Pressure (top) and velocity (bottom) errors at convergence between the transition-zone tracking algorithm and the regularized algorithm for both the linear and nonlinear cases of Section~\ref{subsec:examples:case1}; the values of $\alpha$ and $\beta$ are $50$ and $5000$, respectively, for the linear case, and $10$ and $100$, respectively, for the nonlinear case}%
    \label{fig:ex1_linear_error}
\end{figure}

In this test case, given the obtained results, we claim that the two considered algorithms perform equivalently.

\subsubsection{Nonlinear case}\label{subsubsec:examples:case1:nonlinear}

We consider the same data as the case in Section~\ref{subsubsec:examples:case1:linear}, except for the drag operator and the vector source term:
\begin{gather*}
  \Lambda(u) =
  \begin{dcases*}
    u & in $\Omega_1(u)$,\\
    [0.01+3\bar u,1]\, u & in $\Gamma(u)$,\\
    (0.01 + 3\abs{u})u & in $\Omega_2(u)$.
  \end{dcases*}
  \quad \text{and} \quad
  f(x) = 0
  .
\end{gather*}
For the graphs of $\Psi_\e:=\Psi_{0,\e}+\Psi_{1,\e}$ and $\phi_\e:=\phi_{0,\e}+\phi_{1,\e}$ ($m=1$), see Figure~\ref{fig:ex1_linear_convolution}(b).

The solution is reported in Figure~\ref{fig:ex1_nonlinear_solutions}. Comparing with the previous, linear case, we notice the different shape of the inverse permeability. In fact, now there is a nonlinear relation with the velocity, which is not present in the previous case.
\begin{figure}[tbp]
    \centering
    \includegraphics[width=0.66\textwidth]{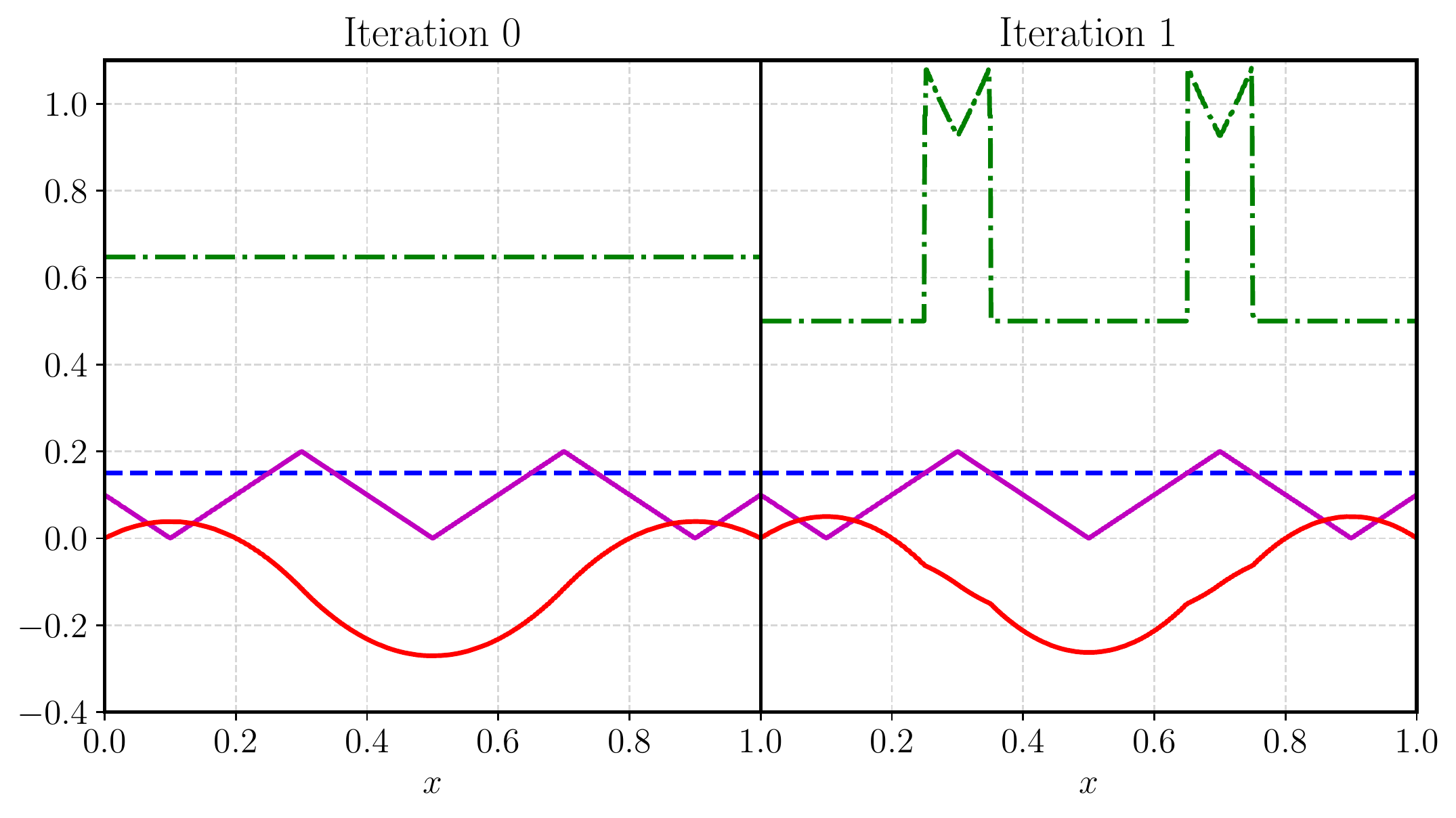}
    \includegraphics[width=0.6\textwidth]{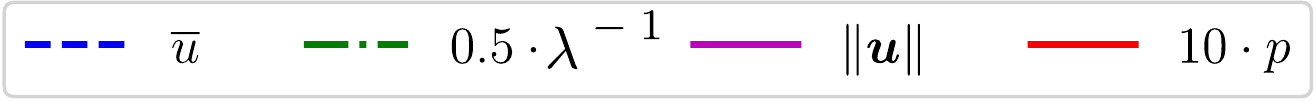}%
    \caption{Solutions for the problem of Section~\ref{subsubsec:examples:case1:nonlinear} at different algorithm iterations}%
    \label{fig:ex1_nonlinear_solutions}
\end{figure}
The algorithm in \eqref{eq:discr_mollified} stops only after two iterations. The errors are reported in Figure~\ref{fig:ex1_linear_error}(b) and, also in this case, are small. We notice that the velocity is slightly shifted with respect to the reference solution, this might be due to a grid effect which disappears for smaller discretization size.

Also in this test case, the proposed algorithm performs similarly to that in~\cite{FP21}.

\subsection{Three transition laws} \label{subsec:example2}

In this part, we validate the proposed procedure in the case of three transition laws. Note that the algorithm in~\cite{FP21}, as it is now, cannot handle this case. The ordered thresholds are given by $\bar u_1 = 0.075$ and $\bar u_2=0.15$ and, as before, we consider both linear and nonlinear transition laws. We consider the domain $\Omega = (0,1)$ and the mollification parameter $\e = 10^{-3}$ and tolerance $\tau=10^{-6}$ in \eqref{eq:discr_mollified}. Mesh size is set to be $10^{-3}$. Also in this case the scalar and vector source terms are given by \eqref{eq:qx-f}.

\subsubsection{Linear case} \label{subsubsec:examples:case2:linear}
We consider first the linear case, where, for all $a\geq0$,
 \begin{gather*}
    \Psi_1(a) = \lambda_{01} a,\\
    \Psi_2(a) = \lambda_{02} a + (\lambda_{01}-\lambda_{02})\bar u_1^2,\\
    \Psi_3(a) = \lambda_{03} a + (\lambda_{02}-\lambda_{03})\bar u_2^2 + (\lambda_{01}-\lambda_{02})\bar u_1^2,
\end{gather*}
where $\lambda_{01} = 1$, $\lambda_{02} = 0.5$, and $\lambda_{03} = 0.25$ and where the integration constants are determined using \eqref{eq:cj}. The graphs of $\Psi_\e:=\Psi_{0,\e}$ and $\phi_\e:=\phi_{0,\e}$ ($m=0$) are given in Figure~\ref{fig:ex2_convolution}(a), where we can identify the three laws.

\begin{figure}[tbp]
    \centering
    \includegraphics[width=1\textwidth]{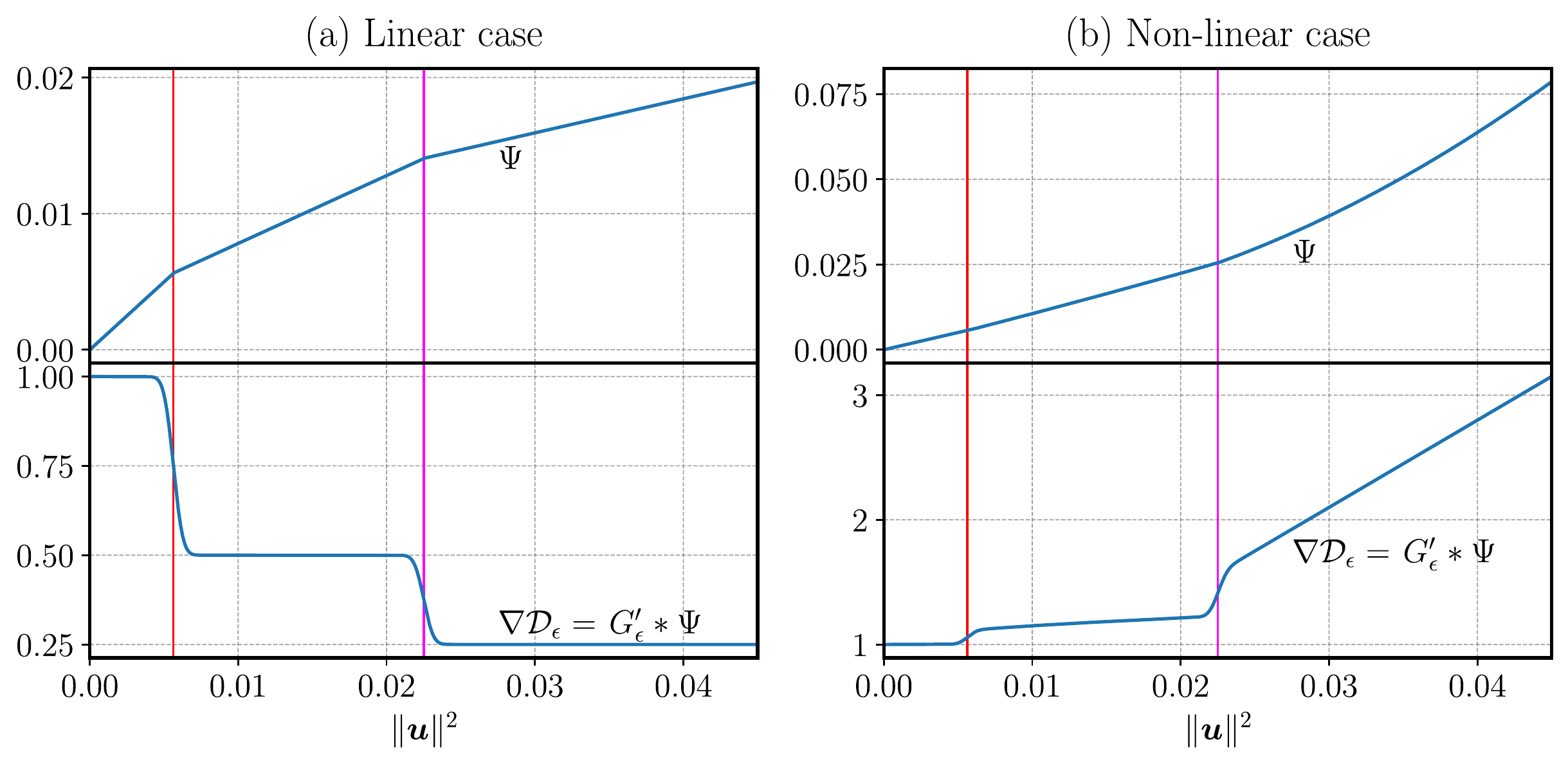}%
    \caption{Regularized-dissipation integrand $\Psi_\e$ (top) and regularized inverse permeability $\phi_\e$ (bottom) as functions of the square velocity for both the linear and nonlinear cases of Section~\ref{subsec:example2}; thresholds $\bar u_1^2 = 0.005625$ and $\bar u_2^2=0.0225$ represented by vertical lines}%
    \label{fig:ex2_convolution}
\end{figure}

The algorithm \eqref{eq:discr_mollified} converges in two iterations with relative error of the order of the machine precision; the solution obtained are reported in Figure~\ref{fig:ex2_linear_solutions}.
\begin{figure}[tbp]
    \centering
    \includegraphics[width=0.66\textwidth]{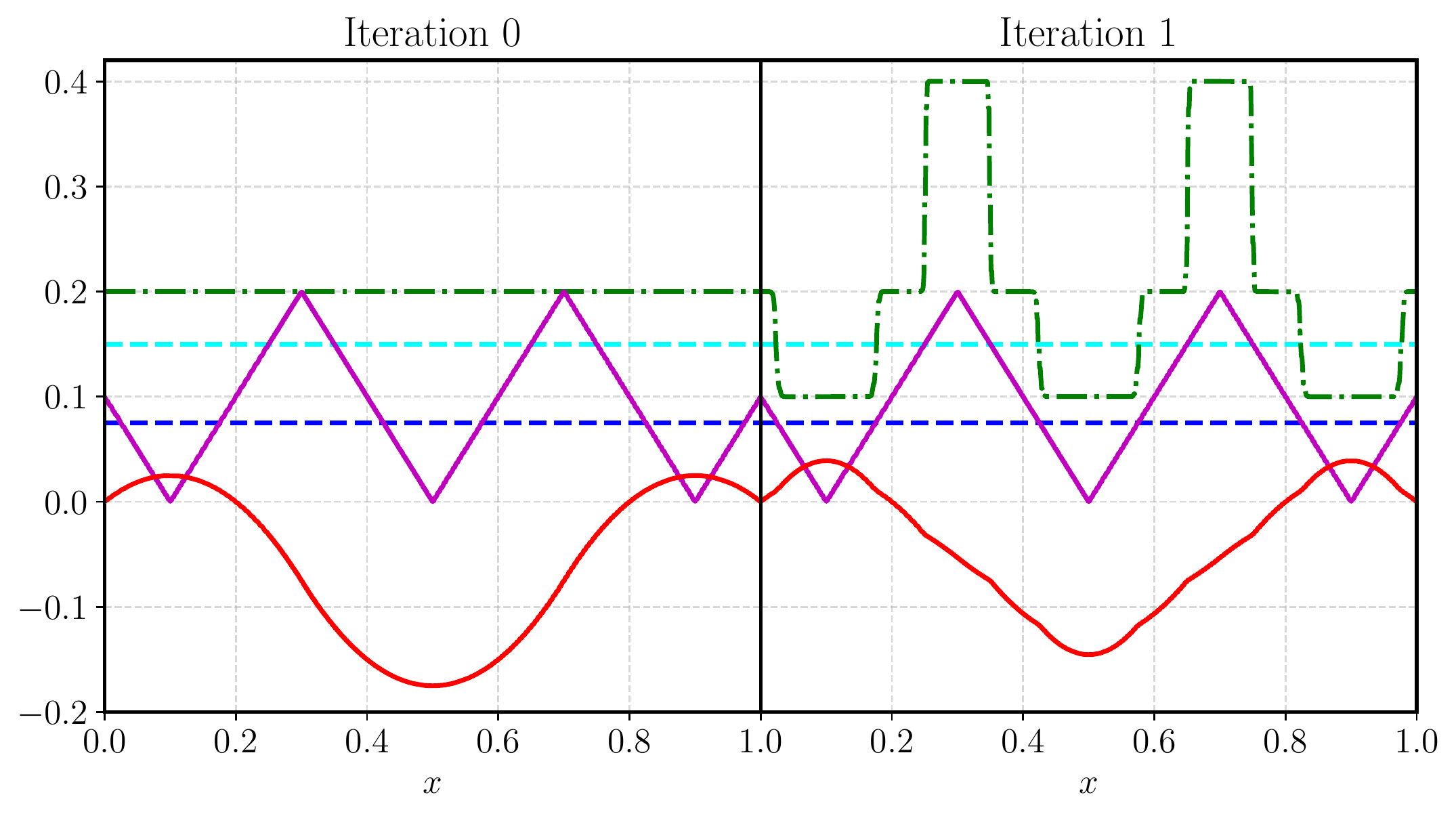}
    \includegraphics[width=0.75\textwidth]{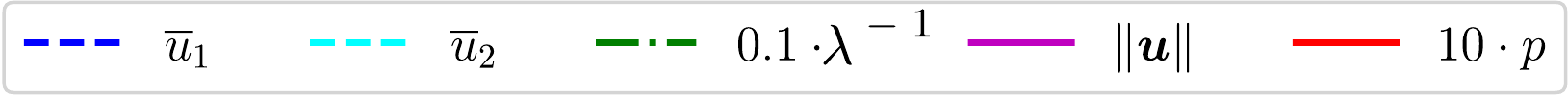}%
    \caption{Solutions for the problem of Section~\ref{subsubsec:examples:case2:linear} at different algorithm iterations}%
    \label{fig:ex2_linear_solutions}
\end{figure}
We notice the effect of the two thresholds that create three different inverse permeability plateaus, smoothly connected.

This simple test case showcases the flexibility of the new approach with multiple transition laws, the generalization to even more laws being immediate.

\subsubsection{Nonlinear case} \label{subsubsec:examples:case2:nonlinear}

We still assume the same velocity thresholds as before but we consider that, for low velocity, the Darcy part is predominant and it is the only one that needs to be modeled. By increasing the velocity, the nonlinear effects start to appear and the Darcy--Forchheimer law is more appropriate. Finally, for high velocity, the nonlinear part is predominant and we thus consider only a Forchheimer law. For all $a\geq0$, we set
\begin{gather*}
  \Psi_1(a) = \lambda_{01}a,\\
  \Psi_2(a) = \lambda_{01} a + \lambda_{12} (a ^{3/2} - \bar u_1^3),\\
  \Psi_3(a) = \lambda_{01}\bar u_2^2 + \lambda_{12} (\bar u_2^{3/2} - \bar u_1^3) + \lambda_{23}  (a ^{2} - \bar u_2^4),
\end{gather*}
where $\lambda_{01}=1$, $\lambda_{12}=1$ and $\lambda_{23}=35$ and the integration constant are computed using \eqref{eq:cj}. The graphical representations of $\Psi_\e := \Psi_{0,\e}+\Psi_{1,\e}+\Psi_{2,\e}$ and $\phi_\e := \phi_{0,\e}+\phi_{1,\e}+\phi_{2,\e}$ ($m=2$) are given in Figure~\ref{fig:ex2_convolution}(b), where we can recognize the two linear laws and the linear one.

The solution, for different iterations, is represented in Figure~\ref{fig:ex2_non-linear_solutions}, where we can notice the three different regions associated to the different flow regimes. Also in this case, the extension to multiple nonlinear laws is rather immediate once the functions $\Psi_i$ are properly defined.
\begin{figure}[tbp]
    \centering
    \includegraphics[width=0.66\textwidth]{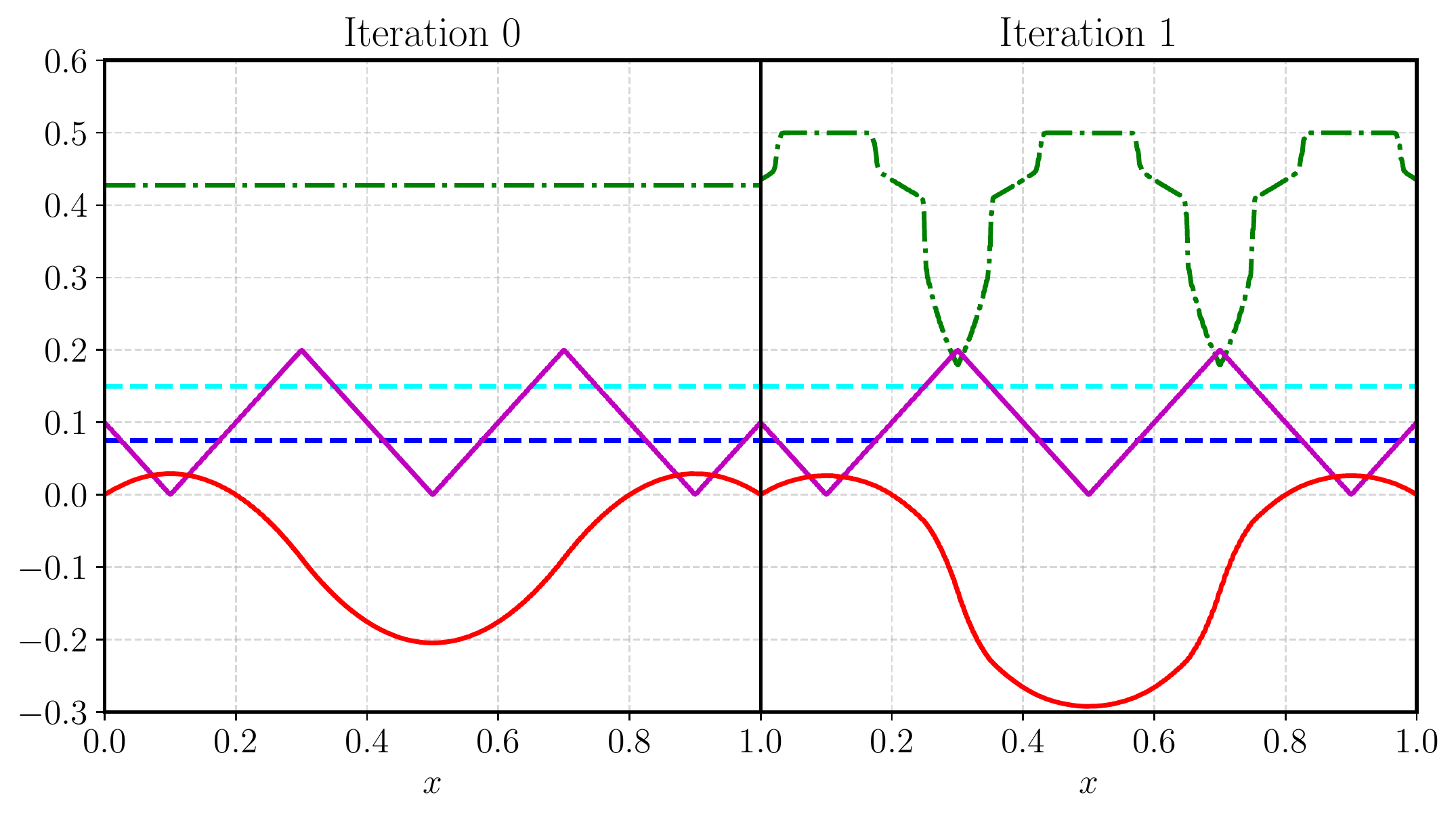}
    \includegraphics[width=0.75\textwidth]{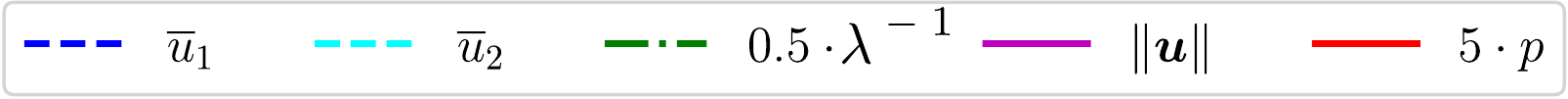}%
    \caption{Solutions for the problem of Section~\ref{subsubsec:examples:case2:nonlinear} at different algorithm iterations}%
    \label{fig:ex2_non-linear_solutions}
\end{figure}

\subsection{Two-dimensional example}\label{subsec:2dexample}

We consider Layer 35 (starting the numeration from 1) of the well known 10th SPE Comparative Solution Project (SPE10) dataset, described in~\cite{Christie2001}. It is a two-dimensional domain of size $365.76\times670.56$ metres composed of a structured grid of $60\times 220$ elements. In each element, a background permeability is associated and can vary abruptly between two neighboring elements, an example is reported on the left in Figure~\ref{fig:ex2d_permeabiltiy}. We choose $\e=10^{-3}$ as mollification parameter and $\tau=10^{-6}$ as tolerance in \eqref{eq:discr_mollified}. Further, we set the source terms as $q\equiv 0$ and $\bff\equiv \bnull$.
\begin{figure}[tbp]
    \centering
    \includegraphics[width=0.45\textwidth]{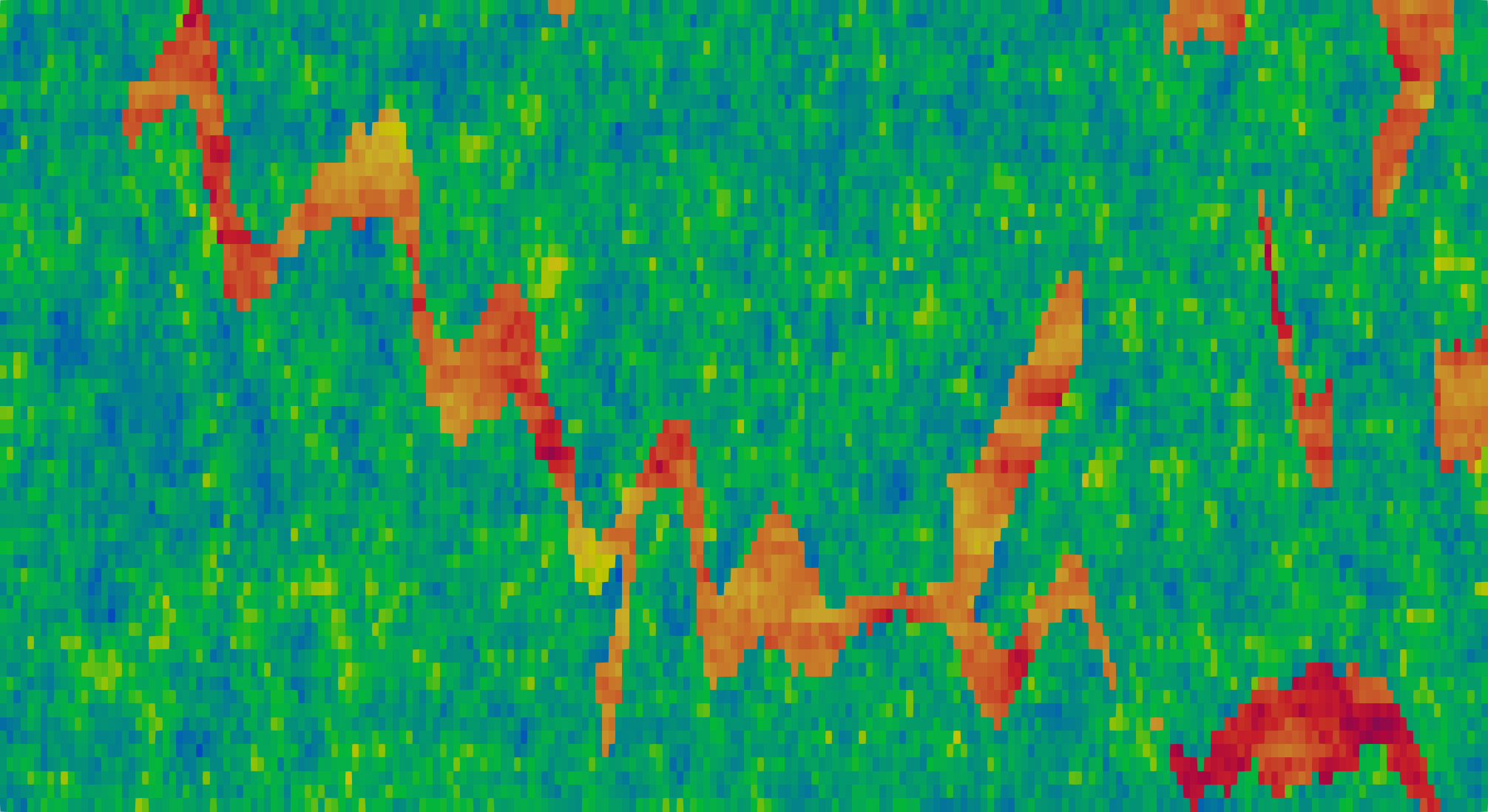}%
    \hspace*{0.00625\textwidth}%
    \includegraphics[width=0.45\textwidth]{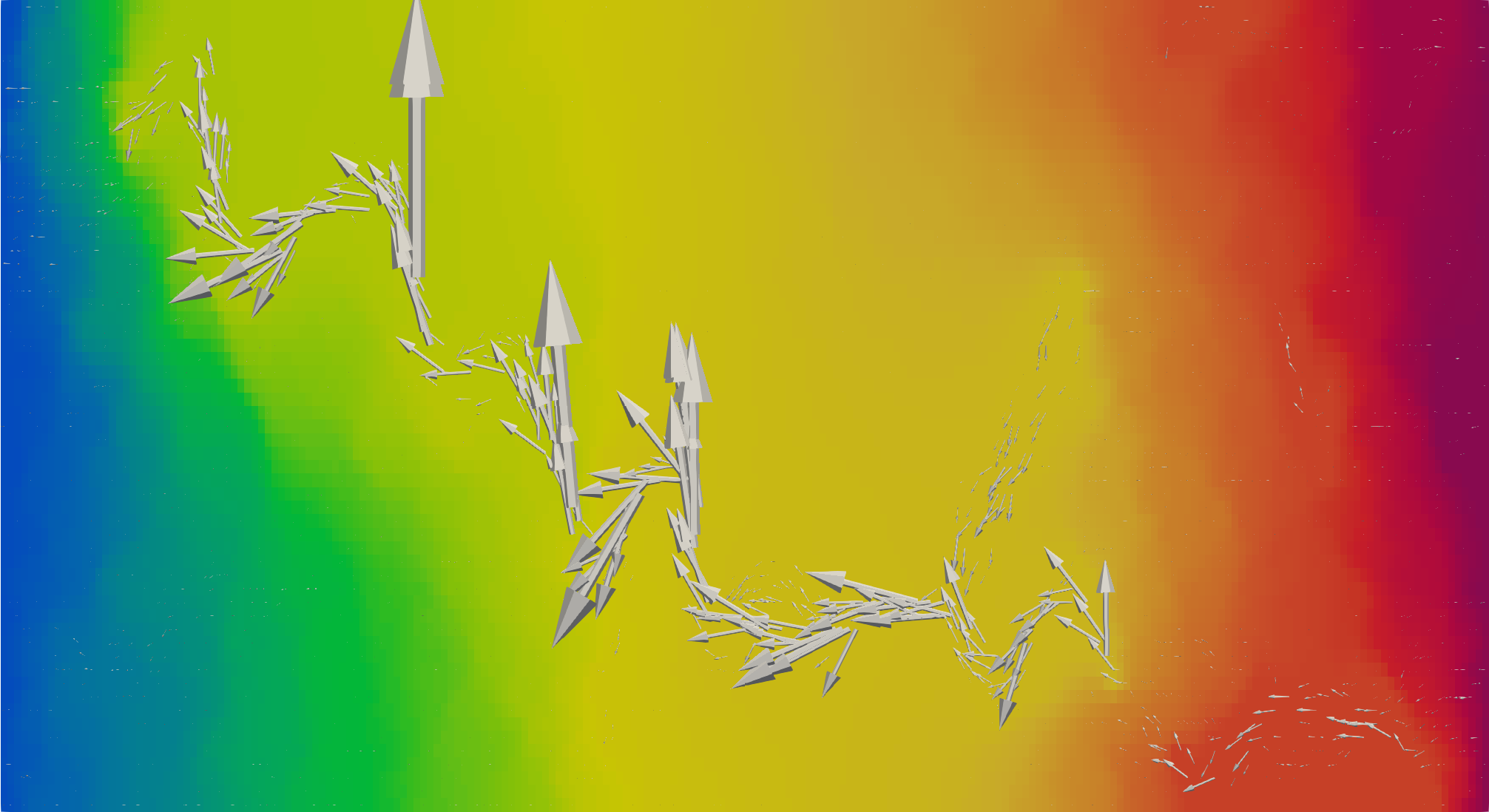}
    \includegraphics[width=0.45\textwidth]{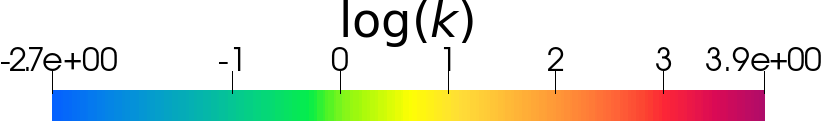}%
    \hspace*{0.00625\textwidth}%
    \includegraphics[width=0.45\textwidth]{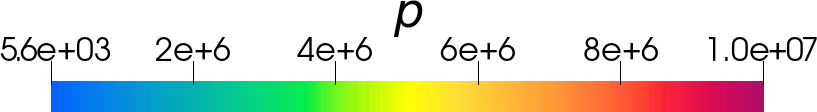}
    \caption{Permeability field of Layer 35 of the SPE10 test case (left) and pressure and velocity solution for the case of Section~\ref{subsec:2dexample} with $\alpha=1$ and $\beta=10$ (right)}%
    \label{fig:ex2d_permeabiltiy}
\end{figure}

In the domain, we allow both a Darcy model and a Darcy--Forchheimer model, the former being used in the slow region and the latter in the fast region. We consider thus the following drag operator:
\begin{gather*}
    \bLambda(\bu) =
    \begin{dcases*}
        \lambda_{\mt{bg}} \bu & in $\Omega_1(\bu)$,\\
        \lambda_{\mt{bg}}[0.1 + \beta \bar u,1] \bu & in $\Gamma(\bu)$,\\
        \lambda_{\mt{bg}}(0.1 + \beta \norm{\bu})\,\bu & in $\Omega_2(\bu)$.
    \end{dcases*}
\end{gather*}
where the threshold velocity $\bar u$ is given by $\alpha 10^{-7}$ and $\lambda_{\mt{bg}}$ denotes the background inverse permeability given by the benchmark data; here, $\alpha$ and $\beta$ are two parameters that may change. We set pressure boundary conditions on the left and right parts of the domain, respectively, with values $0$ and $10^7 \sib{\pascal}$. The top and bottom boundaries are set to have no flow. A representative solution obtained using the regularized algorithm is given on the right in Figure~\ref{fig:ex2d_permeabiltiy}.

We first vary the value of the threshold velocity to understand its impact: the lower, the more elements should belong to $\Omega_2(\bu)$ (with the Darcy--Forchheimer law). Since the latter is the fast region, we expect it is focused on the regions of high background permeability. We set $\beta=10$. Figure~\ref{fig:ex2d_thresold} shows the regions $\Omega_1(\bu)$ and $\Omega_2(\bu)$ for smaller values of $\bar u$ with $\alpha \in \{1, 2^{-1}, 2^{-2}, 2^{-3}, 2^{-4}, 2^{-5} \}$.
\begin{figure}[tbp]
    \centering
    \includegraphics[width=0.45\textwidth]{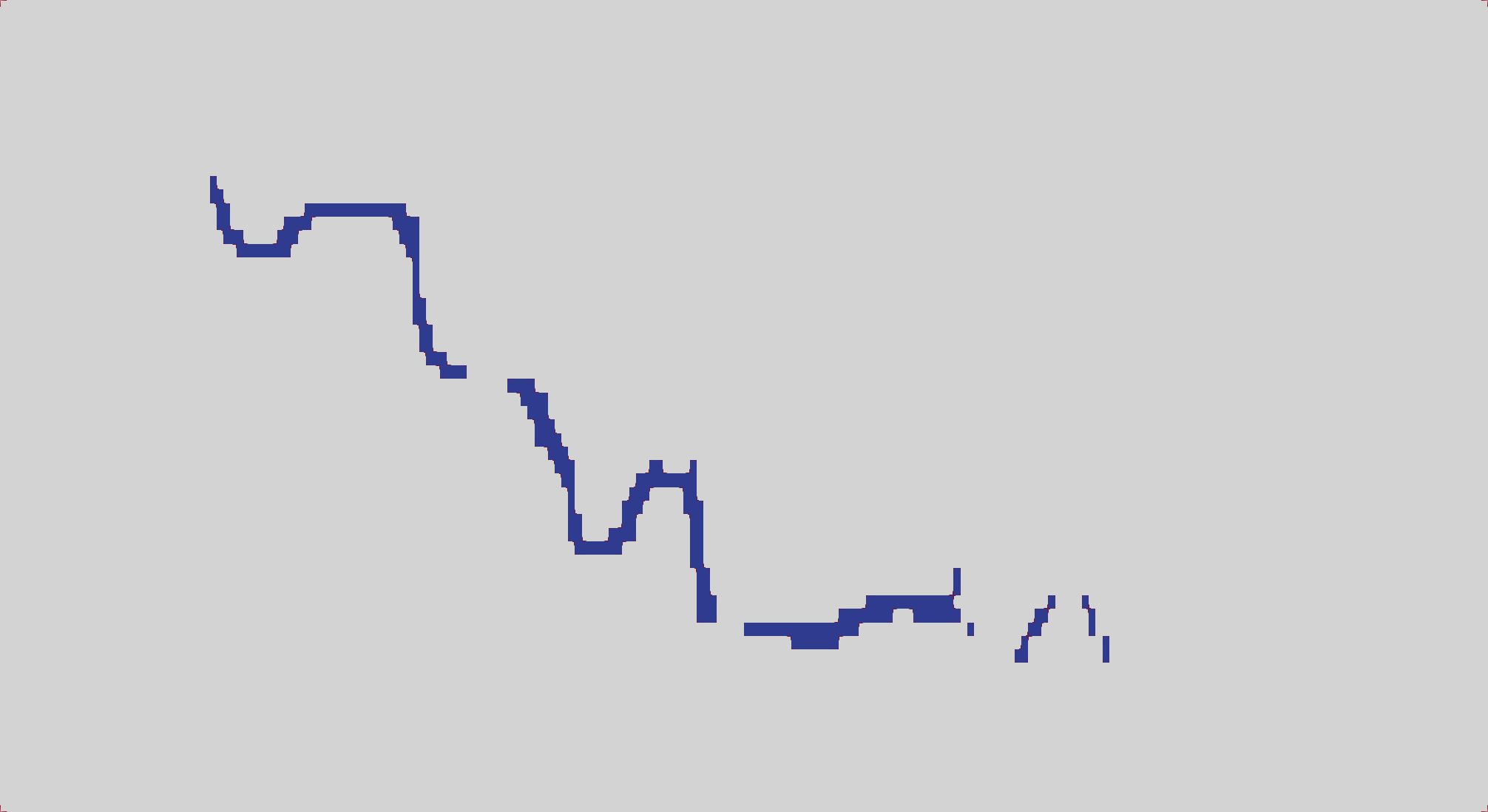}%
    \hspace*{0.00625\textwidth}%
    \includegraphics[width=0.45\textwidth]{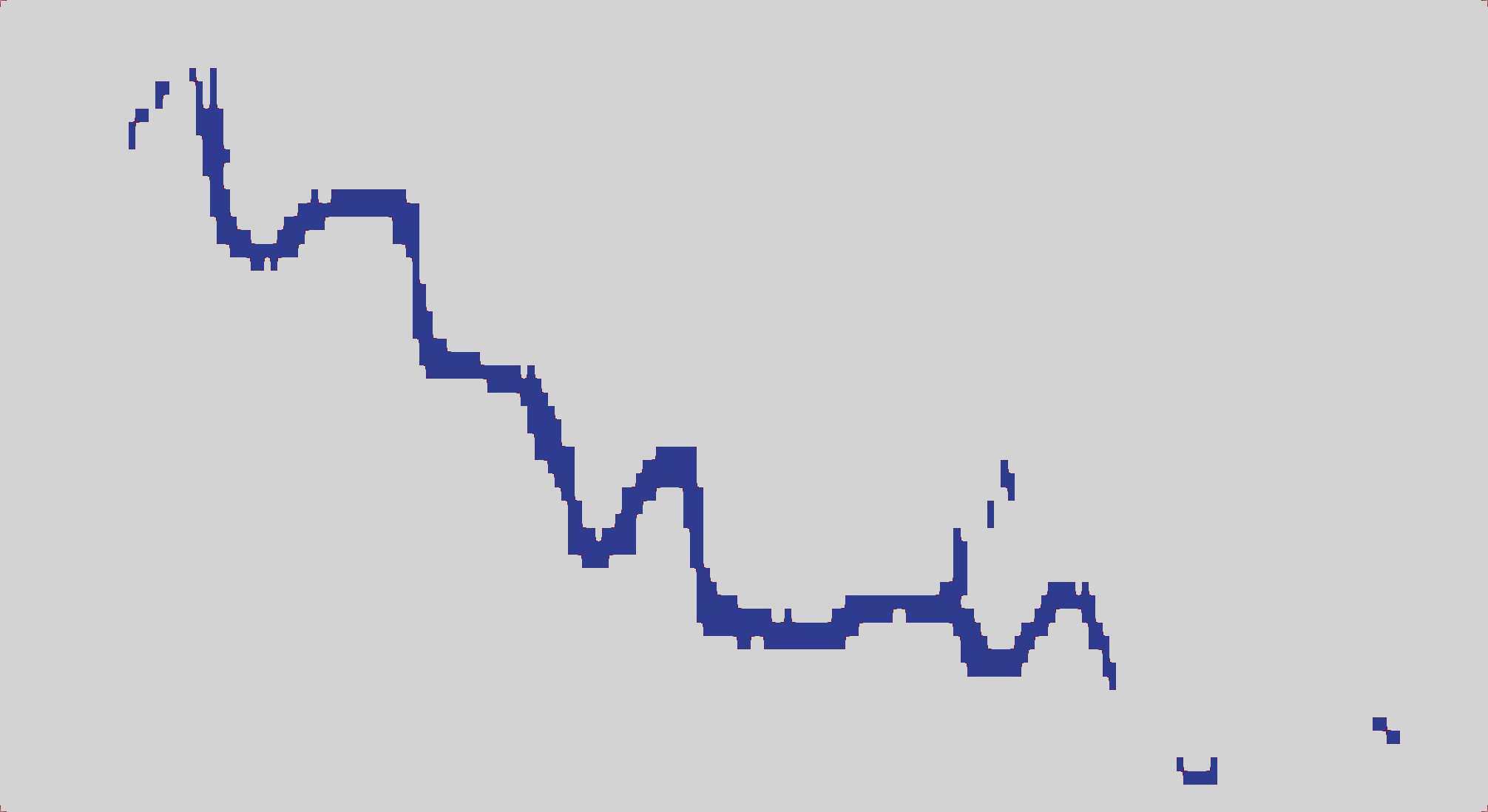}
    \includegraphics[width=0.45\textwidth]{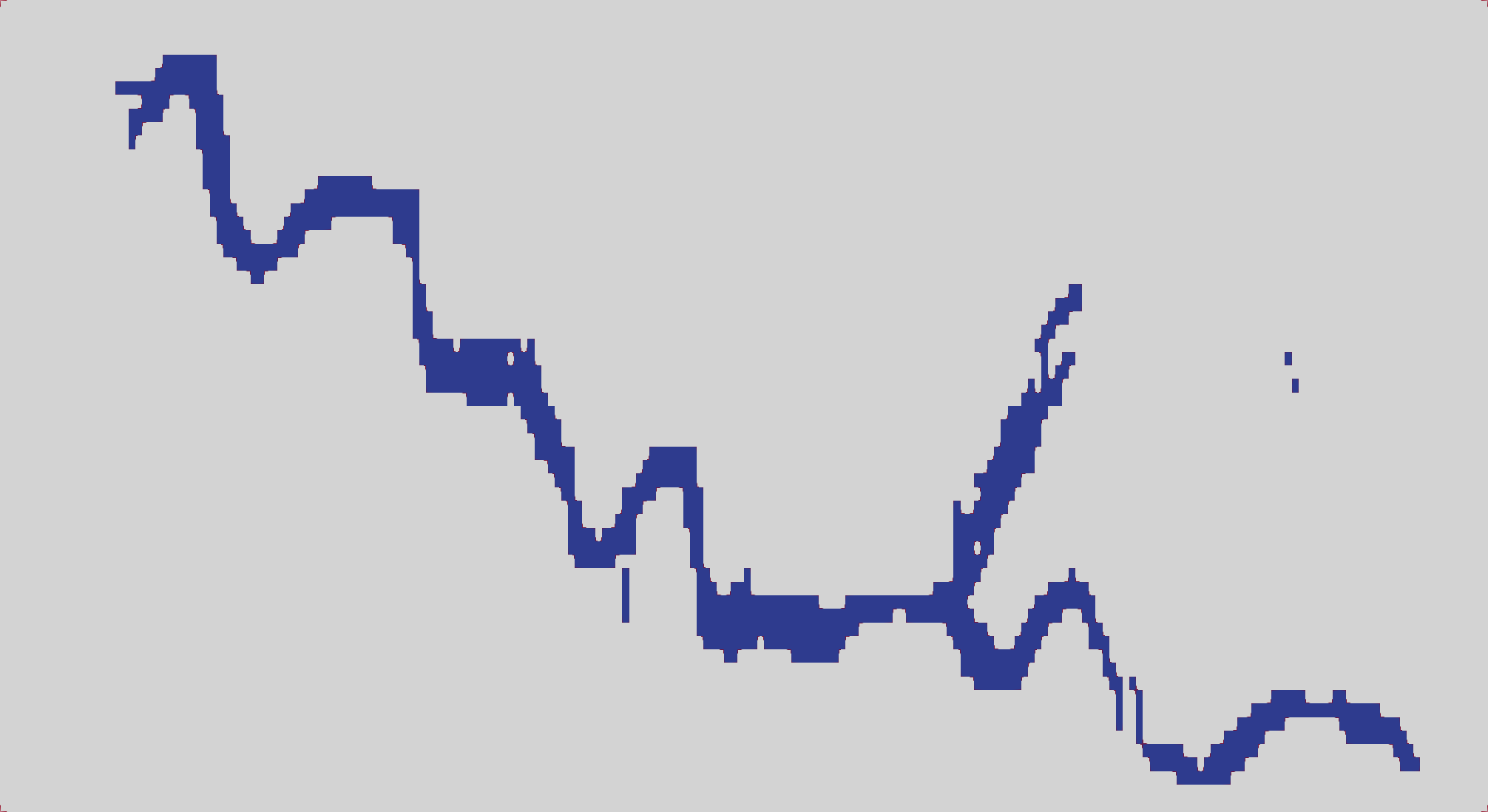}%
    \hspace*{0.00625\textwidth}%
    \includegraphics[width=0.45\textwidth]{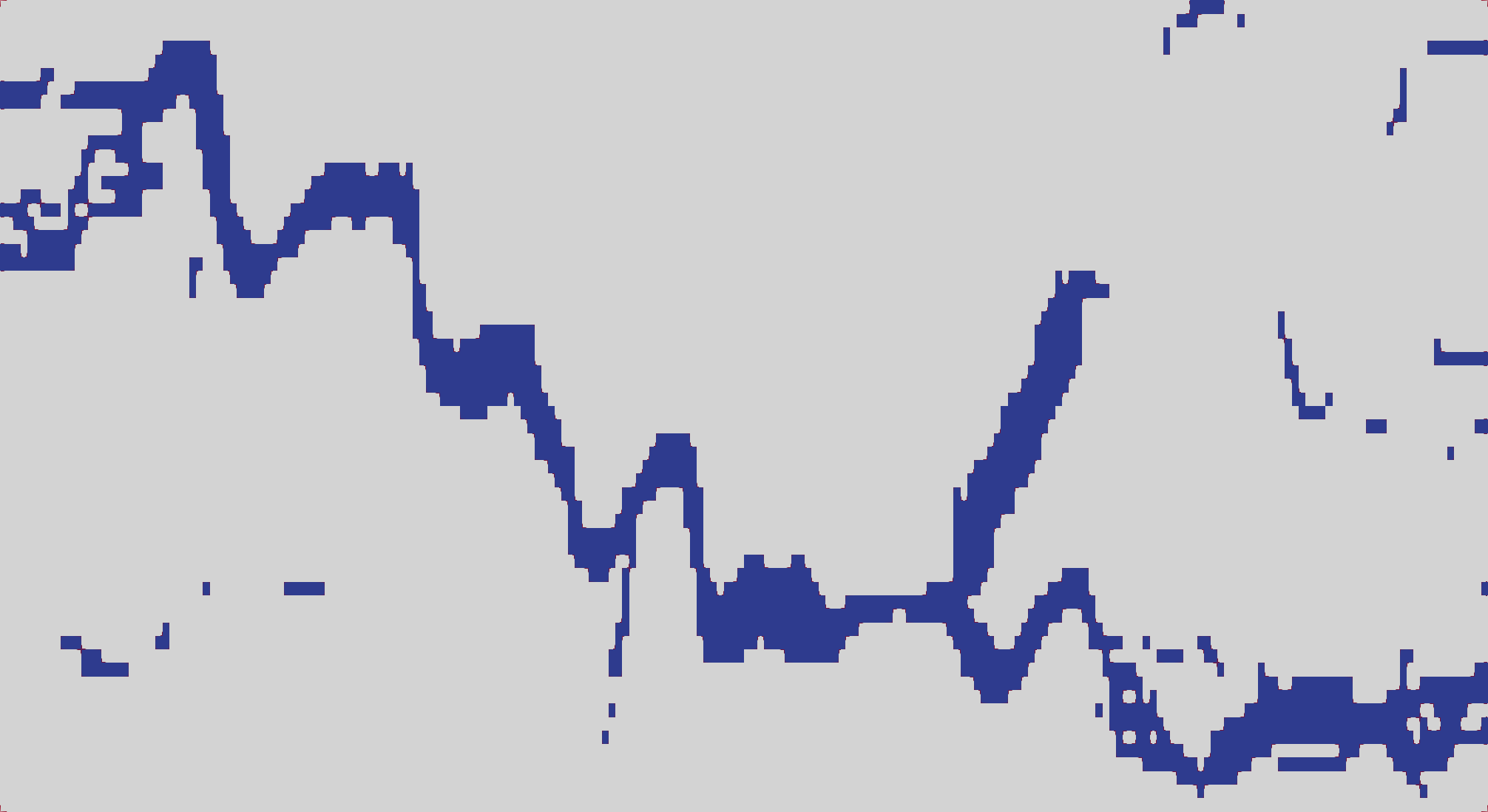}
    \includegraphics[width=0.45\textwidth]{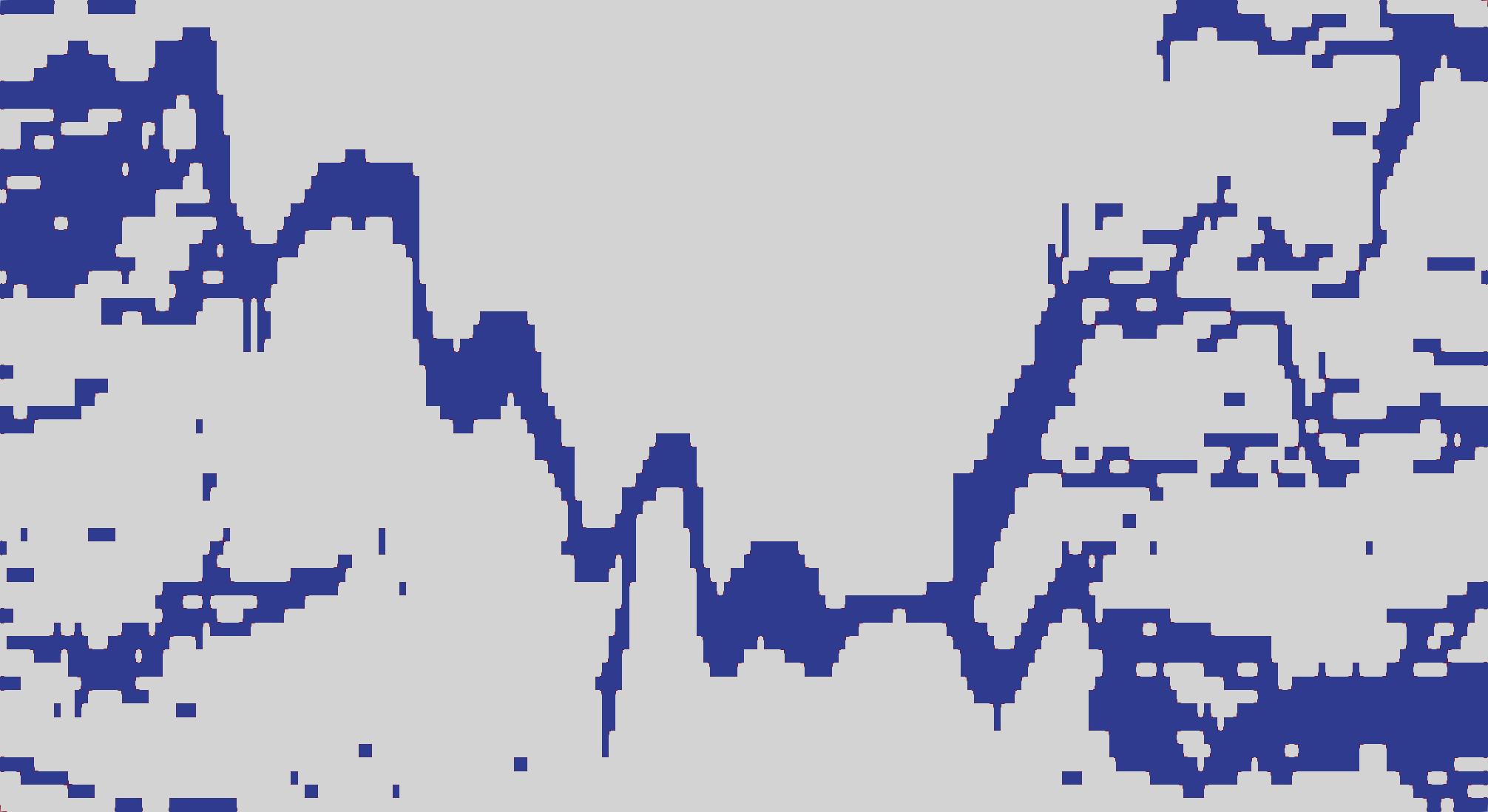}%
    \hspace*{0.00625\textwidth}%
    \includegraphics[width=0.45\textwidth]{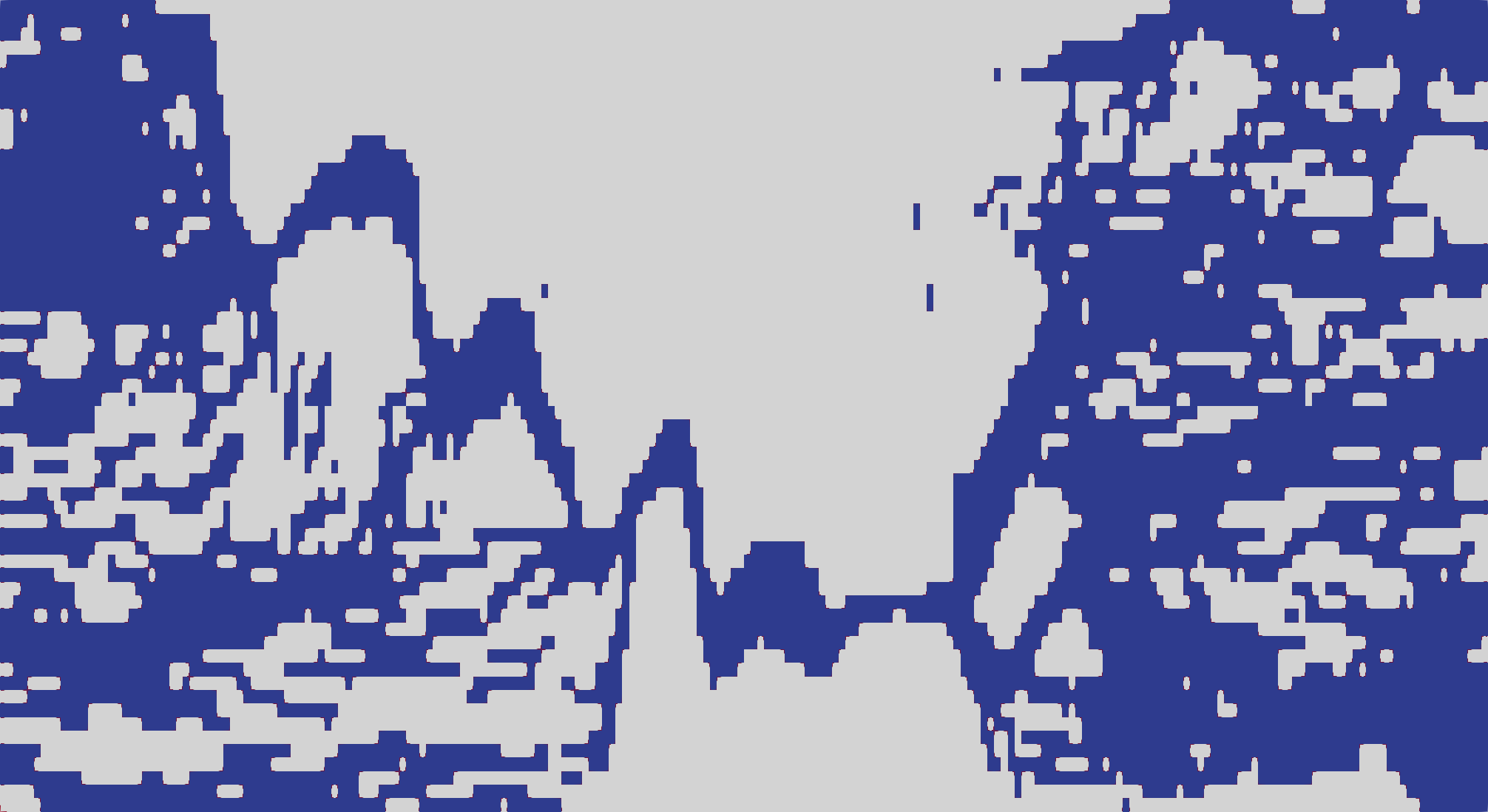}
    \caption{Configurations of $\Omega_1(\bu)$ (gray) and $\Omega_2(\bu)$ (blue) for the test case of Section~\ref{subsec:2dexample}, with different values of $\alpha$ in decreasing order going from left to right and from top to bottom, and $\beta=10$}%
    \label{fig:ex2d_thresold}
\end{figure}
In Figure~\ref{fig:ex2d_thresold}, we report the configurations obtained for $\Omega_1(\bu)$ and $\Omega_2(\bu)$. We notice that for high values of $\alpha$, only a narrow channel allows the presence of the Darcy-Forchheimer model, mostly where the background permeability is already high. When the value of $\alpha$ gets smaller, more intricate configurations appear, showing that the smaller channels ``attract'' the fast flow model. The number of iterations needed for the convergence are $\{2, 2, 2, 3, 4, 6\}$, respectively, for decreasing values of $\alpha$.

As a second experiment, we fix the value of $\alpha=2^{-5}$ and we increase $\beta$ in $\{10, 100, 500, 1000\}$. The expected effect is that the effective permeability in the Darcy--Forchheimer becomes smaller and so does the velocity. As a consequence, more cells should belong to $\Omega_1(\bu)$ and fewer to $\Omega_2(\bu)$.
\begin{figure}[tbp]
    \centering
    \includegraphics[width=0.45\textwidth]{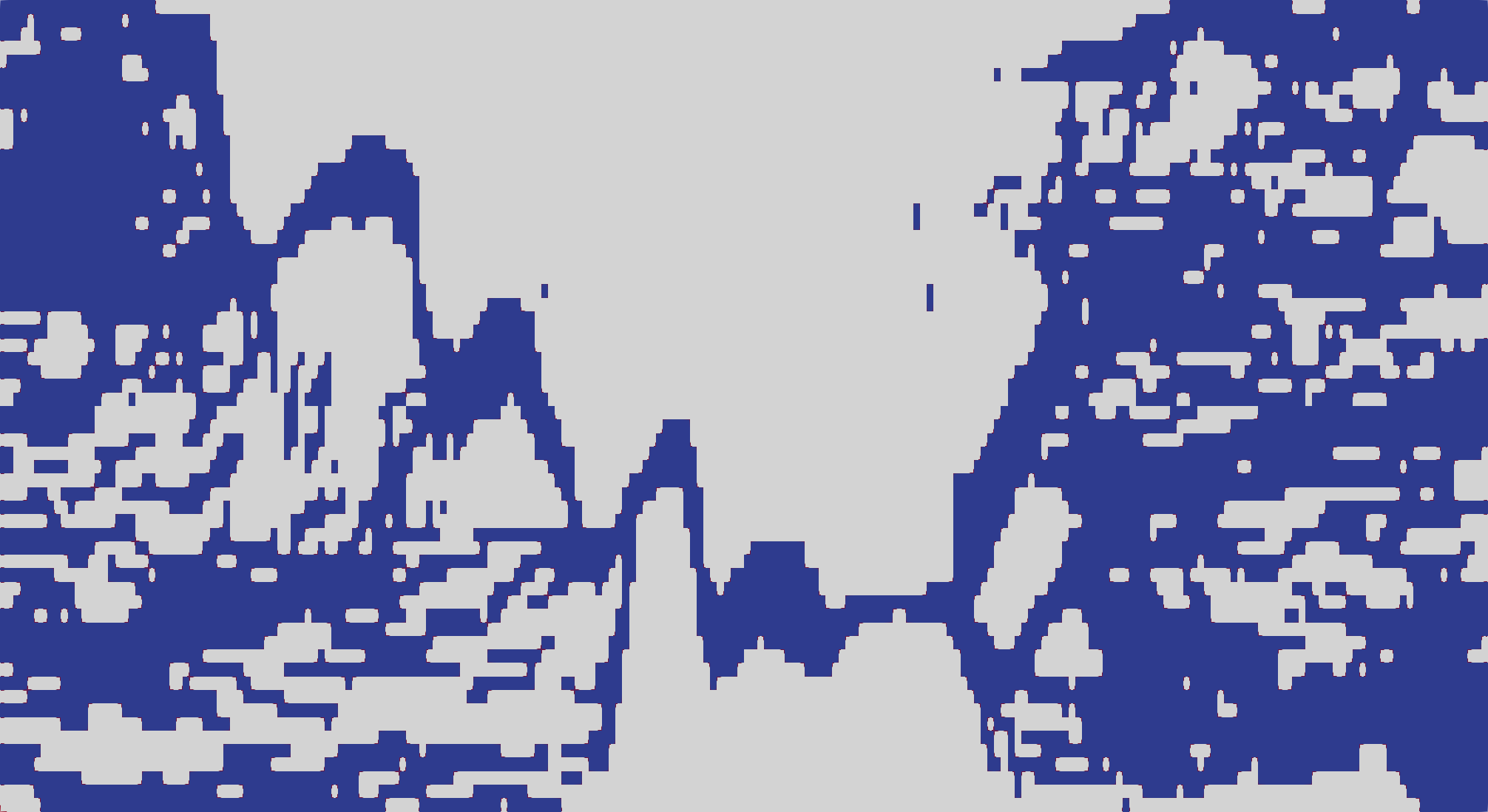}%
    \hspace*{0.00625\textwidth}%
    \includegraphics[width=0.45\textwidth]{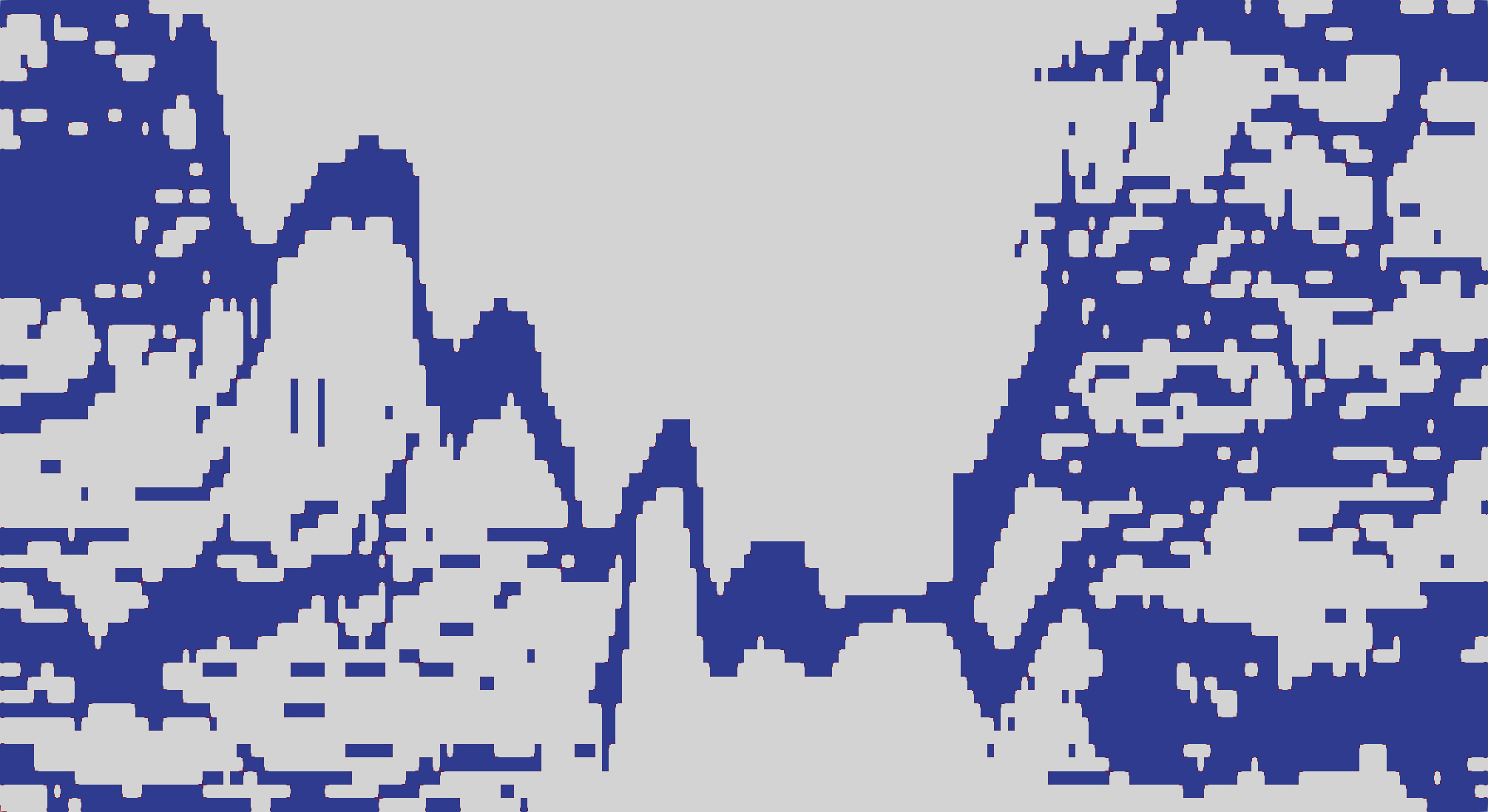}
    \includegraphics[width=0.45\textwidth]{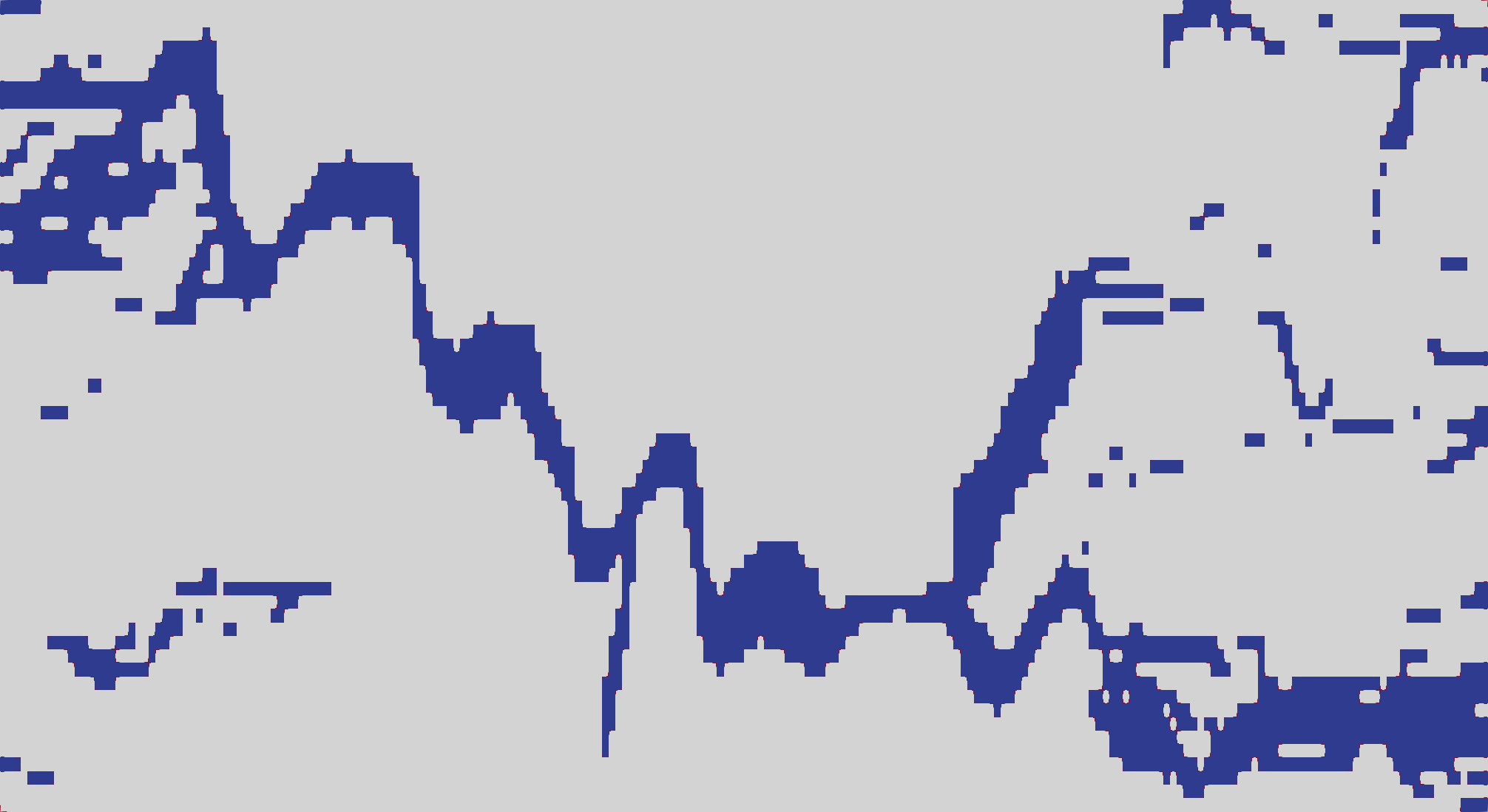}%
    \hspace*{0.00625\textwidth}%
    \includegraphics[width=0.45\textwidth]{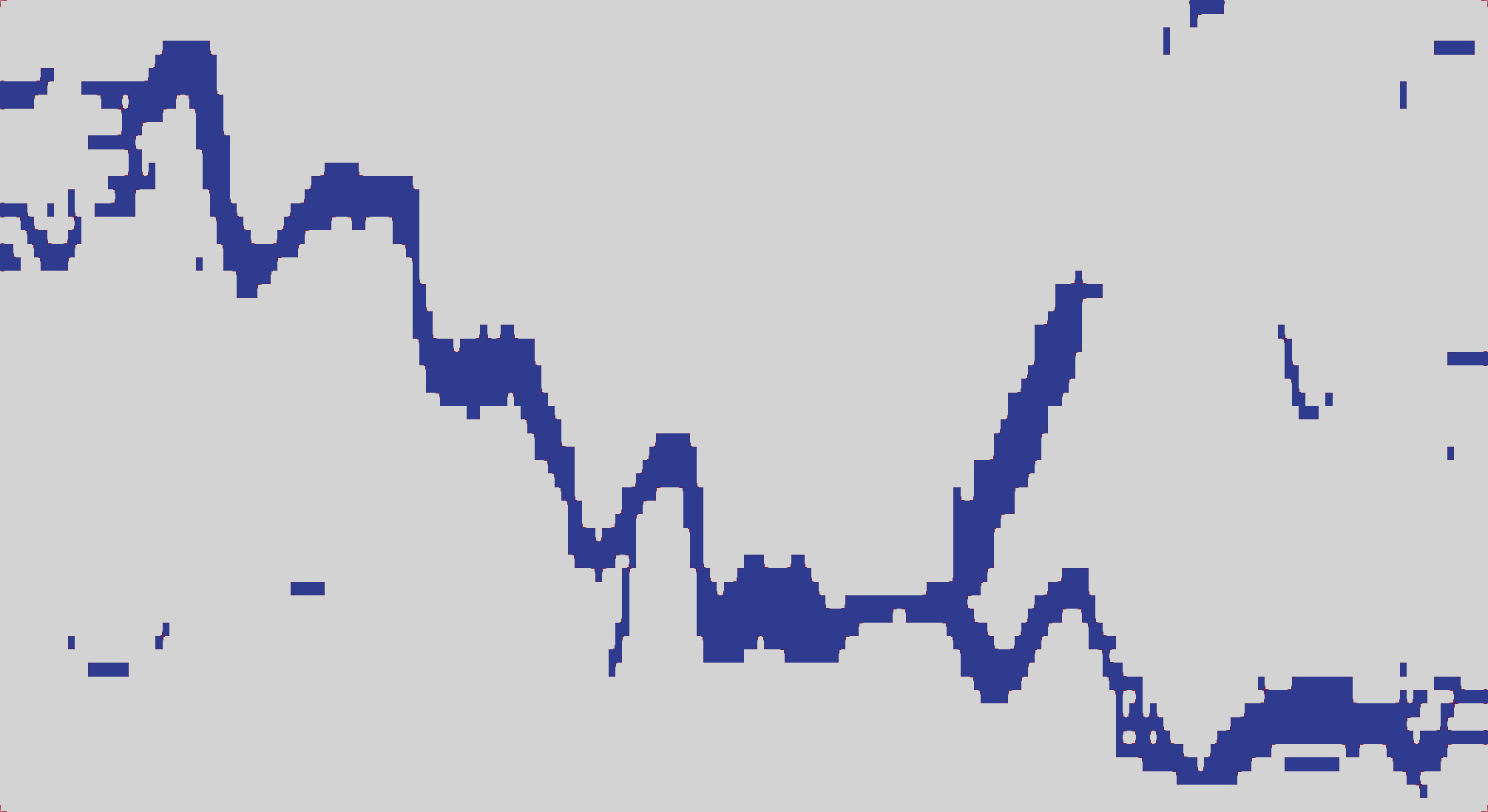}
    \caption{Configurations of $\Omega_1(\bu)$ (gray) and $\Omega_2(\bu)$ (blue) for the test case of Section~\ref{subsec:2dexample} with different values of $\beta$ in increasing order going from left to right and from top to bottom, and $\alpha=2^{-5}$}%
    \label{fig:ex2d_nonlinear}
\end{figure}
Figure~\ref{fig:ex2d_nonlinear} shows the following phenomenon: the higher the value of $\beta$, the smaller the effective permeability and so the smaller the velocity. Only with a background permeability already quite high is it possible to allow a Darcy--Forchheimer model even for high values of the Forchheimer coefficient $\beta$.
 The number of iterations needed for the convergence are $\{6, 2, 2, 2\}$, respectively, for increasing values of $\beta$.

What this test case displays is that our regularized scheme provides results which, albeit qualitatively expected, are complex to forecast otherwise. It also shows that the setup for two-dimensional (and, in fact, three-dimensional) simulations is immediate thanks to the mollification approach which is dimension-independent.


\section{Conclusion} \label{sec:conclusion}

In this work, we have presented a mathematical framework for adaptively choosing the
most appropriate constitutive law depending on the developed fluid velocity. The
problem is mathematically formulated as a multivalued problem and, under the hypothesis of maximal monotonicity of the drag operator, the problem has been shown to be weakly well posed. If the drag operator fails to be monotone, we have shown existence of weak solutions when $d=1$. Moreover, we have derived a monovalued regularization of the drag operator, which yields a well posed problem converging to the original multivalued problem when the convexity of the dissipation holds. When the convexity fails, the existence of solutions to the regularized problem is at least ensured if $d=1$, although the convergence is not. Compared to the transition-zone tracking algorithm presented in~\cite{FP21}, the resulting variational scheme is easier to implement (as no transition-zone tracking is required), generalizes to multiple flow regimes, and extends to higher space dimensions. We have validated the approach on three test cases, and thus showcased its applicability and flexibility.

As possible directions for future investigation, we may think of generalizing the approach presented here to drag operators involving space derivatives of the velocity (such as in the case of Brinkman's law). We may also want to study the non-monotone, nonconvex case, for which we only have existence when $d=1$ and have no convergence result of the regularized problem. Finally, in the context of fractured porous media, we may want to explore the possibility of applying such an adaptive and regularized approach in the fractures only and couple it with a classical Darcy flow in the rest of the porous medium.


\appendix

\section{Multivalued operators and functionals}
\label{sec:noti-mult-oper}

We give here the basic definitions and facts regarding multivalued, and in fact monovalued, operators which are used in the main body of the text.

We call \emph{duality system} a triple $(X,Y,b)$ if $X$ and $Y$ are real Banach spaces and $b$ is a nondegenerate bilinear form on $Y\times X$ (referred to as a \emph{dual pairing}). 
In particular, if $X$ is a real Banach space and $X^*$ its topological dual, then there exists a canonical dual pairing $b^*$ such that $(X,X^*,b^*)$ is a duality system; whenever given such a pair $(X,X^*)$, we always assume it is equipped with its canonical dual pairing.

Recall that, given a pair of real Banach spaces $(X,Y)$ and a multivalued operator $A\: X \rightrightarrows Y$, we call \emph{graph} of $A$ the set $\graph(A)$ defined by
\begin{equation*}
  \graph(A) = \{ (x,y)\in X\times Y \st y \in A(x)\}.
\end{equation*}
 
\begin{defn}[monotone operator]
  \label{defn:monotone-op}
  Let $(X,Y,b)$ be a duality system and let $A\:X \rightrightarrows Y$ be a multivalued operator. We say that $A$ is \emph{monotone} if 
  \begin{equation}
    \label{eq:monotone-op}
    b(y_1-y_2,x_1-x_2) \geq 0 \quad \text{for all $(x_1,y_1),(x_2,y_2)\in\graph(A)$}.
  \end{equation}
In this case, we say that $A$ is \emph{maximal} if there is no monotone operator $B\:X\rightrightarrows Y$ such that $\graph(A)$ is strictly included in $\graph(B)$. We furthermore say that $A$ is \emph{strictly} monotone if the inequality in \eqref{eq:monotone-op} is strict whenever $x_1\neq x_2$.
\end{defn}
\noindent Note that, according to the above definition, a monotone operator $A$ is positive semidefinite in the sense that $b(y,x)\geq 0$ for all $(x,y)\in\graph(A)$.

Coercivity, boundedness and set-continuity are important notions on operators:
\begin{defn}[coercive operator]
  \label{defn:coercive-op}
  Let $(X,Y,b)$ be a duality system, denote by $\norm{\cdot}_X$ the norm on $X$ and fix $\sigma\geq1$. We say that a multivalued operator $A\:X \rightrightarrows Y$ is \emph{$\sigma$-coercive} if there exists a map $c\:\R_+\to\R$ so that $c(a) \to \infty$ as $a\to\infty$ and 
  \begin{equation*}
    b(y,x) \geq c(\norm{x}_X)\norm{x}_X^\sigma \quad \text{for all $(x,y)\in\graph(A)$}.
  \end{equation*}
\end{defn}
\noindent Note that a coercive operator $A\:X\rightrightarrows \R$ is necessarily bounded below.

\begin{defn}[bounded operator]
  \label{defn:bounded-op}
  Let $(X,\norm{\cdot}_X)$ and $(Y,\norm{\cdot}_Y)$ be real Banach spaces and fix $\sigma\geq0$. We say that a multivalued operator $A\:X \rightrightarrows Y$ is \emph{$\sigma$-bounded} if there is a constant $C>0$ such that
  \begin{equation*}
    \norm{y}_Y \leq C(1+\norm{x}_X^\sigma) \quad \text{for all $(x,y)\in\graph(A)$}.
  \end{equation*}
\end{defn}

\begin{defn}[set-continuous operator]
  \label{defn:lip-op}
  Let $X$ and $Y$ be real Banach spaces. We say that a multivalued operator $A\:X \rightrightarrows Y$ is \emph{set-continuous} if, for every sequence $(x_n)_n \subset X$ converging to some $x\in X$, we have $\dist(A(x_n),A(x)) \to 0$ as $n\to\infty$, i.e., there exists a sequence $(y_n)_n\subset Y$ with $y_n\in A(x_n)$ for all $n$ converging to an element of $A(x)$. If $A$ is monovalued, we simply say that $A$ is \emph{continuous}.
\end{defn}
\noindent Extending the definitions of coercivity, boundedness and set-continuity to operators defined on merely a convex subset of a Banach space is immediate.

\section{Functionals}
\label{sec:functionals}

We recall the main notions and facts on functionals used in this paper. In particular, we discuss the concepts of subdifferential and $\Gamma$-convergence.

\subsection{Subdifferentials}
\label{sec:subdifferentials}

Let us start with the notion of Clarke subdifferential and some of its properties:
\begin{defn}[Clarke subdifferential~\cite{Clarke90}]
  Let $(X,Y,b)$ be a duality system. Given $\F\:X\to\R$ locally Lipschitz continuous, we call \emph{Clarke subdifferential} of $\F$ the multivalued operator $\p\F\: X\rightrightarrows Y$ defined, for all $x\in X$, by
  \begin{equation*}
    \p\F(x) = \left\{ y \in Y \st \forall\, z\in X,\; \limsup_{x'\to x,\, \delta\downarrow 0} \frac{\F(x'+\delta z) - F(x')}{\delta} \geq b(y,z) \right\}.
  \end{equation*}
We say that $\F$ is \emph{differentiable} if its subdifferential is monovalued, in which case we write $\p\F(x) = \{\grad \F(x)\}$ for all $x\in X$.
\end{defn}

\begin{prop}[properties of the Clarke subdifferential]
  \label{prop:clarke}
  Let $(X,Y,b)$ be a duality system, and let $\F\:X\to\R$ be locally Lipschitz continuous. Then, $\p \F(x)$ is nonempty, convex and compact for all $x\in X$. If moreover $\F$ is convex, then, for all $x\in X$, the Clarke and \emph{Fréchet} subdifferentials of $\F$ coincide, i.e.,
  \begin{equation*}
    \p\F(x) = \left\{ y \in Y \st \forall\, z\in X,\; \liminf_{\delta\downarrow 0} \frac{\F(x+\delta z) - F(x)}{\delta} \geq b(y,z) \right\};
  \end{equation*}
  in this case, we simply refer to $\p\F$ as the subdifferential of $\F$.
\end{prop}
\noindent Extending the definition of Clarke subdifferential to functionals defined on a convex subset of a Banach space is straightforward, resulting in multivalued operators defined on the subset.

Let us give the definition of critical point and minimizer and then provide some additional, more or less obvious, useful properties:
\begin{defn}[critical point and minimizer]
  Let $X$ be a real Banach space and let $\F\:X\to\R$ and $x\in X$. We say that $x$ is a \emph{critical point} of $\F$ if $0\in\p\F(x)$. We say that $x$ is a \emph{local minimizer} of $\F$ if there exists $\eta > 0$ such that for all $z\in X$ we have $\F(x+ \delta z) > \F(x)$ for all $\delta \in [0,\eta)$. We say that $x$ is a \emph{global minimizer} of $\F$ if $\F(x)\leq \F(z)$ for all $z\in X$.
\end{defn}

\begin{prop}[properties of critical points and minimizers]
  \label{prop:prop-critical-min}
  Let $X$ be a real Banach space and let $\F\:X\to\R$ and $x\in X$. The following assertions hold:
  \begin{enumerate}[label=(\roman*)]
    \item If $x$ is a global minimizer of $\F$, then $x$ is a local minimizer of $\F$.
    \item If $x$ is a local minimizer of $\F$, then $x$ is a critical point of $\F$.
    \item If $\F$ is convex, then $x$ is a local minimizer of $\F$ if and only if $x$ is a critical point of $\F$.
    \item If $\F$ is strictly convex, then there can exist at most one local minimizer of $\F$. 
  \end{enumerate}
\end{prop}

Let us also recall the definition of saddle point:
\begin{defn}[saddle point]
  Let $X$ and $Y$ be real Banach spaces and let $\F\:X\times Y\to\R$ and $(x,y)\in X\in Y$. We say that $(x,y)$ is a \emph{saddle point} of $\F$ if 
  \begin{equation*}
    \F(\xi,y) \leq \F(x,y) \leq \F(x,\upsilon) \quad \text{for all $(\xi,\upsilon)\in X\times Y$}.
  \end{equation*}
\end{defn}

\subsection{$\Gamma$-convergence}
\label{sec:gamma-convergence}

We give the definition of $\Gamma$-convergence in its minimally general form needed here:
\begin{defn}[$\Gamma$-convergence~\cite{Braides02}]
  \label{defn:gamma-cv}
  Let $X$ be a real Banach space, and let $\F\:X\to\R$ and $(\F_\e)_{\e>0}$ be such that $\F_\e\: X\to \R$ for all $\e>0$. We say that $(\F_\e)_{\e>0}$ \emph{$\Gamma$-converges} to $\F$, and write $\F_\e \to_\Gamma \F$, as $\e\to 0^+$ if both conditions below are satisfied:
  \begin{enumerate}[label=(\roman*)]
    \item for all $x\in X$ and $(x_\e)_{\e>0} \subset X$ such that $x_\e \wto x$ as $\e\to0^+$, there holds
    \begin{equation*}
      \liminf_{\e\to 0^+} \F_\e(x_\e) \geq \F(x);
    \end{equation*}
    \item for all $x\in X$, there exists $(x_\e)_{\e>0} \subset X$ (referred to as \emph{recovery sequence} for $x$) such that $x_\e \wto x$ as $\e\to0^+$ and 
    \begin{equation*}
      \limsup_{\e\to 0^+} \F_\e(x_\e) \leq \F(x).
    \end{equation*}
  \end{enumerate}
We say that $(\F_\e)_{\e>0}$ \emph{$\Gamma$-converges} to $\Gamma$ as $\e\to 0^+$ \emph{along global minimizers} if the ``liminf'' condition above is only checked for some $(x_\e)_{\e>0}$ such that $x_\e$ is a global minimizer of $\F_\e$ for all $\e>0$.
\end{defn}

The following is a fundamental property of $\Gamma$-convergence, for which we provide the quick proof:
\begin{prop}[$\Gamma$-convergence and convergence of minimizers]
  \label{prop:gamma-cv-min}
  Let $X$ be a real Banach space, and let $\F\:X\to\R$ and $(\F_\e)_{\e>0}$ be such that $\F_\e\: X\to \R$ for all $\e>0$. Suppose that $\F_\e \to_\Gamma \F$ as $\e\to0^+$ along global minimizers, and assume that $(x_\e)_{\e>0}\subset X$ is such that $x_\e$ is a global minimizer of $\F_\e$ for all $\e>0$ and that there exists $x\in X$ with $x_\e\wto x$ as $\e\to0^+$. Then, $x$ is a global minimizer of $\F$. 
\end{prop}

\begin{proof}
Let $y\in X$, and let $(y_\e)_\e$ be a recovery sequence for $y$. Then, using the ``limsup'' and ``liminf'' conditions in the definition of $\Gamma$-convergence and the minimality of $(x_\e)_{\e>0}$, we get
\begin{equation*}
  \F(y) \geq \limsup_{\e\to 0^+} \F_\e(y_\e) \geq \liminf_{\e\to 0^+} \F_\e(y_\e) \geq \liminf_{\e\to 0^+} \F_\e(x_\e) \geq \F(x),
\end{equation*}
which shows that $x$ is a global minimizer of $\F$.
\end{proof}

\section{Mollification}
\label{sec:conv-moll}

Let us recall some well known notions and facts on the convolution of one-variable functions which lead to the concept of mollification.

\begin{defn}[Schwartz class~\cite{Schwartz66}]
  A smooth function $f\:\R\to\R$ is said to be in the \emph{Schwartz class} if, for all $\alpha,\beta\in\N\cup\{0\}$, it satisfies
\begin{equation*}
  \sup_{x\in\R} \abs{x^\alpha f^{(\beta)}(x)} < \infty,
\end{equation*}
where $f^{(\beta)}$ stands for the $\beta$th derivative of $f$.
\end{defn}
\noindent Any smooth and compactly supported function is in the Schwartz class, and so is the normal distribution.

The Schwartz class allows us to define the convolution product between functions that are not necessarily integrable but have the ``right'' growth at infinity.
\begin{defn}[convolution]
  \label{defn:conv}
  Let $f\:\R\to\R$ be in the Schwartz class and $g\:\R\to\R$ be continuous and $\sigma$-bounded for some $\sigma\geq0$ in the sense of Definition~\ref{defn:bounded-op}. The \emph{convolution} $f*g\:\R\to\R$ of $f$ and $g$ is given by
  \begin{equation*}
    f*g(x) = \int_{-\infty}^\infty f(x-y)g(y) \d y = \int_{-\infty}^\infty f(y)g(x-y) \d y \quad \text{for all $x\in\R$},
  \end{equation*}
where the second equality is obtained by a change of variable.
\end{defn}

\begin{prop}[smoothness of convolution]
  With the notation of Definition~\ref{defn:conv}, it holds that $f*g$ is smooth and $(f*g)^{(\beta)} = f^{(\beta)}*g$ for all $\beta\in\N\cup\{0\}$.
\end{prop}

Functions in the Schwartz class can be used to approximate nonsmooth functions smoothly following the definitions and the proposition below.

\begin{defn}[Schwartz mollifier]
  A function $f\:\R\to\R$ is said to be a \emph{Schwartz mollifier} if it is in the Schwartz class, it is even and nonnegative, and $\int_{-\infty}^{\infty} f = 1$.
\end{defn}

\begin{defn}[mollifying sequence]
Given $f\:\R\to\R$ a Schwartz mollifier, the family $\{f_\e\}_{\e>0}$  of functions from $\R$ to $\R$ defined, for all $\e>0$, by
\begin{equation*}
  f_\e(x) = \frac1\e f\left(\frac x\e\right) \quad \text{for all $x\in\R$}
\end{equation*}
is called a \emph{mollifying sequence}.
\end{defn}

\begin{prop}[convergence, convexity and monotonicity of mollification]
  \label{prop:cv-moll}
  Let $\{f_\e\}_{\e>0}$ be a mollifying sequence and let $g\:\R\to\R$ be continuous and $\sigma$-bounded for some $\sigma\geq0$. Then, $\{g_\e\}_{\e>0} := \{f_\e*g\}_{\e>0}$ is referred to as a \emph{mollification} of $g$, and it satisfies
  \begin{equation*}
    \lim_{\;\e\to 0^+} g_\e(x) = g(x) \quad \text{for all $x\in\R$}.
  \end{equation*}
  Moreover, if $g$ is convex (respectively, strictly convex), then, for all $\e>0$, we have that $g_\e$ is convex (respectively, strictly convex) and 
  \begin{equation*}
    g_\e(x) \geq g(x) \quad \text{for all $x\in\R$};
  \end{equation*}
  if instead $g$ is nondecreasing (respectively, strictly increasing), then, for all $\e>0$, we have that $g_\e$ is nondecreasing (respectively, strictly increasing).
  
\end{prop}




\bibliographystyle{abbrv}

\end{document}